\documentclass[reqno,11pt]{amsart}
\usepackage{amsmath,amsfonts,amssymb,amsxtra,latexsym,amscd,enumerate,amsthm,verbatim,color,physics, marginnote, xcolor, bbm, ulem}

\usepackage[hidelinks]{hyperref}
\usepackage{changepage}
\allowdisplaybreaks
\usepackage[margin=1.40in]{geometry}
\numberwithin{equation}{section}

\newcommand{\red}[1]{\textcolor{red}{#1}}

\newcommand{\p}{\partial}

\newcommand{\D}{D_\lambda}

\newcommand{\R}{\mathbb{R}}

\newcommand{\T}{\mathbb{T}}

\newcommand{\Z}{\mathbb{Z}}

\numberwithin{equation}{section} 

\usepackage{comment}

\newtheorem{theorem}{Theorem}[section]
\newtheorem{lemma}[theorem]{Lemma}
\newtheorem{proposition}[theorem]{Proposition}

\newtheorem{definition}[theorem]{Definition}
\newtheorem{remark}[theorem]{Remark}

\newtheorem{assumption}[theorem]{Assumption}

\begin{document}

\title[Inviscid damping and vorticity depletion]{Optimal Linear Inviscid Damping and Vorticity Depletion for Non-monotonic shear flows}

\author{Rajendra Beekie}
\address{Tufts University }
\email{rajendra.beekie@tufts.edu}

\author{Shan Chen}
\address{University of Minnesota}
\email{chen7081@umn.edu}

\author{Hao Jia}
\address{University of Minnesota}
\email{jia@umn.edu}




\begin{abstract}
{\small}
We give a proof of optimal pointwise linear inviscid damping and vorticity depletion for a large class of spectrally stable shear flows $b(y)$ on the rectangular torus. This extends the linear inviscid estimates in \cite{CJWZ25} for Kolmogorov flow $b(y) =\cos y$ using completely different methods. Crucially, both methods remove an apparent logarithmic correction to the depletion and inviscid damping rates. Our method combines pointwise bounds for a parametrix with an explicit solution of a simpler inner problem which makes the removal of the logarithm clear. 
\end{abstract}

\maketitle

\bigskip


\setcounter{tocdepth}{1}

\tableofcontents

\section{Introduction}
Consider a 2d fluid on a doubly periodic domain:
\begin{subequations}
    \begin{align}\label{Euler-Hao}
    \p_t \omega +\nabla^\perp \psi \cdot \nabla \omega &=  \nu \Delta \omega , \quad \nu \geq 0, \\
    \Delta \psi & = \omega.
\end{align}
\end{subequations}
 Here $\omega$ is the vorticity, $\psi$ is the stream function, and $\nu \geq 0$ is the viscosity. For  a large class of initial data, the two-dimensional incompressible Euler ($\nu=0$) and Navier--Stokes ($\nu>0$) equations on the torus are globally well posed. However, there are still many interesting open problems regarding the dynamics of the long-time behavior.
    
One of the most remarkable features of 2d fluids is the formation of coherent states from quite general initial data. For weakly viscous fluids governed by the Navier-Stokes equation the long-time behavior is eventually dominated by viscous decay to the trivial zero solution on a timescale like $\mathcal{O}(\nu^{-1})$ (assuming mean-zero condition on the initial data). However, it has been observed numerically  that there are metastable states that act as attractors for the dynamics on much shorter time scales than the viscous one \cite{BW13,BCS19}. 
    
Numerical studies show that the selection between ``bar states'' (parallel shear flows) and dipoles (pairs of counter-rotating vortices) depends on both the domain geometry and the initial data. On the square torus $\T^2$, a canonical dipole has stream function $a_1 \cos x + a_2 \cos y$ with $a_1 a_2 \neq 0$ \cite{BCS19}. In direct numerical simulations of freely decaying turbulence on this domain, Yin, Montgomery and Clercx \cite{YMC03} observed relaxation to a dipole from random, broad-band initial data, but to a bar state from certain initial vorticity patches. In the distinct setting of weakly forced and dissipative stochastic Navier--Stokes flow, Bouchet and Simonnet \cite{BS09} observed dipoles on the square torus, nearly parallel shear flows on sufficiently elongated rectangular tori, and random transitions between the two at intermediate aspect ratios. For the unforced problem, Beck, Cooper and Spiliopoulos \cite{BCS19} combined analysis and numerical simulations of a reduced model involving the lowest eight Fourier modes to support the preference for dipoles on the square torus and bars on nearby rectangular tori $\T_{2 \pi \kappa} \times \T$; for $\kappa < 1$, the favored bars have velocity in the $x$ direction.

In this paper we focus on the rectangular torus. In the recent work \cite{CJWZ25}, the phenomenon of rapid convergence to metastable states was studied near the Kolmogorov flow $ \omega = \sin y$. It was proved that $H^3$ vorticity perturbations within the threshold $\epsilon\nu^{1/3}$ ($0<\epsilon\ll1$) of the Kolmogorov flow relax to shear flows close to the Kolmogorov flow with quantitative enhanced rates over the surprisingly short timescale $\nu^{-1/3}$, and then slowly decay to the Kolmogorov flow due to viscous damping.  A key ingredient of the proof is a quantitative linear inviscid result establishing a large family of optimal vorticity depletion and inviscid damping for the Kolmogorov flow. The method of proof for these estimates relied on a remarkable almost conserved quantity and it is not clear if an analogue of such a quantity exists for more general flows. Prior to \cite{CJWZ25}, the previous best known results included logarithmic corrections to the pointwise behavior for general spectrally stable shear flows \cite[Theorem~2.1]{BG24}, where it was suggested that this might be optimal. At first glance, the conjectured loss of a logarithmic factor for general shear flows seems natural, since the Rayleigh equation for the resolvent indeed contains a logarithmic factor over the quadratic depletion gain. This potentially suggests that the Kolmogorov flow may have distinguished features which make it more stable than other shear flows on the torus.

In this paper we show, at least on the linearized level, that Kolmogorov flow does not appear to be special. Namely, we show that the optimal bounds hold for shear flows satisfying the natural Assumption~\ref{MaAs}. In fact, we are also able to improve one of the estimates in \cite[Theorem~1.2]{CJWZ25}; see the fifth inequality in Theorem \ref{thm2}. We remark that this is likely the only corresponding bound in \cite[Theorem~1.2]{CJWZ25} which is non-optimal, in view of both the scaling analysis and the analysis of the resolvents.

\subsection{Main equations} 
Consider the two dimensional Euler equation linearized around a shear flow $(b(y),0)$, in the doubly periodic rectangular box $ (x,y,t)\in \mathbb{T}_{2 \pi \kappa} \times\T \times[0,\infty)$ with $\kappa < 1$:
\begin{equation}\label{Main1}
\begin{split}
\partial_t\omega+b(y)\partial_x\omega-b''(y)\p_x \psi &=0,\\
\Delta \psi &= \omega.
\end{split}
\end{equation}
Our goal is to understand the long time behavior of $\omega(t)$ as $t\to\infty$, with Sobolev regular initial vorticity $\omega_0$.

\subsection{Main results}
Our main assumptions on the background flow $(b(y),0), y\in \T$ are the following.
\begin{assumption}\label{MaAs}
We assume that the shear flow $(b(y),0)$, $y\in \T$ satisfies the following conditions:
\begin{itemize}
\item $b\in C^6(\T)$. The set $S:=\{y\in\T: b'(y)=0\}=\{y_{1\ast}, y_{2\ast}\}$, and for some $\varkappa \in(0,1)$, we have
\begin{equation}\label{asp1}
|b''(y)|\in[\varkappa,1/\varkappa]\,\,{\rm for\,\,}y\in S, \quad{\rm and}\quad\|b\|_{C^6(\T)}\leq1/\varkappa.
\end{equation}

\item The linearized operator $L_k: L^2(\T)\to L^2(\T)$ with $ \kappa k\in\Z\backslash\{0\}$ defined for $g\in L^2(\T)$ as
\begin{equation}\label{asp2}
L_kg(y):=b(y)g(y)-b''(y)h(y), \quad{\rm where}\,\,(-k^2+\partial_y^2)h(y)=g(y),\,\,y\in\T,
\end{equation}
has no eigenvalues in $\mathbb{C}$, nor generalized embedded eigenvalues (also called singular limiting modes) in $\{b(y), y\in\T\}$ (see Definition \ref{emb1}).
\end{itemize}
\end{assumption}
\begin{remark}
    Generalizing to the case of more non-degenerate critical points is more technically difficult but does not introduce any conceptual challenges. For this reason we have opted to write the result in the simplest relevant setting.
\end{remark}

Our main result is the following theorem.

\begin{theorem}\label{thm1}
Assume that $\omega\in C([0,\infty), H^{4}(\T_{2\pi \kappa} \times \T ))$ with the associated stream function $ \psi(t,\cdot)$ is the unique solution to \eqref{Main1}, with initial data $\omega_0\in H^4(\T_{2\pi \kappa}\times\T)$ satisfying for all $y\in\T$,
\begin{equation}\label{thm0}
\int_{\T_{2\pi \kappa}}\omega_0(x,y)\,dx=0, \quad \int_{\T_{2 \pi \kappa} \times \T } \psi(t,x,y) \, dx dy = 0.
\end{equation}
Then we have the following bounds. 

(i) Inviscid damping estimates: For $u = \nabla^\perp \psi, $ 
\begin{equation}\label{th0.5}
\|\psi(t,\cdot)\|_{L^\infty(\T_{2\pi \kappa}\times \T) }\lesssim \frac{1}{\langle t\rangle^2}\|\omega_0\|_{H^4(\T_{2\pi \kappa}\times\T)},
\end{equation}
\begin{equation}\label{thm1v}
\|u^x(t,\cdot)\|_{L^\infty(\T_{2\pi \kappa}\times\T)}\lesssim \frac{1}{\langle t\rangle}\|\omega_0\|_{H^4(\T_{2\pi \kappa}\times\T)}, \quad \|u^y(t,\cdot)\|_{L^\infty(\T_{2\pi \kappa}\times\T)}\lesssim \frac{1}{\langle t\rangle^2}\|\omega_0\|_{H^4(\T_{2\pi \kappa}\times\T)}. 
\end{equation} 

(ii) Vorticity depletion estimates: there exists a decomposition
\begin{equation}\label{thm1d}
\omega(t,x,y):=\omega_{\rm loc}(t,x,y)+\omega_{\rm nloc}(t,x,y),
\end{equation}
where for $(x,y,t)\in\T_{2\pi \kappa}\times\T\times[0,\infty)$,
\begin{equation}\label{thm3}
|\omega_{\rm loc}(t,x,y)|\lesssim \min ( |y-y_{1\ast}|^{2}, |y-y_{2\ast}|^{2} )\|\omega_0\|_{H^4(\T_{2\pi \kappa}\times\T)},\quad |\omega_{\rm nloc}(t,x,y)|\lesssim \frac{1}{\langle t\rangle}\|\omega_0\|_{H^4(\T_{2\pi \kappa}\times\T)}.
\end{equation}
\end{theorem}

Taking the Fourier transform in $x$ of the equation \eqref{Main1} for $\omega$, we obtain that
 \begin{align}\label{F3.0}
 \partial_t\omega_k+ikb(y)\omega_k-ikb''(y)\psi_k&=0,\\
                      (\p_y^2 - k^2 )\psi_k &= \omega_k,\label{F3.0:elliptic}
 \end{align}
 for $ \kappa k\in\mathbb{Z}, t\ge0, y \in \T $. In the above, $\omega_k$ and $\psi_k$ are the $k$-th Fourier coefficients of $\omega,\psi$ in $x$ respectively. At the level of $(\omega_k, \psi_k)$ we have the following refined bounds.
\begin{theorem}
\label{thm2}
Let $\kappa k \in \mathbb{Z} \setminus \{ 0\}$, $\theta \in (0, 1/4]$, $y \in \T$ and suppose that $(\psi_k, \omega_k)$  satisfy \eqref{F3.0}--\eqref{F3.0:elliptic} with $\omega_k(0,y) = \omega_{0k}(y)$. Then the following bounds hold.
\begin{subequations}
\label{main:bdd}
    \begin{align}
&|\psi_k(t,y)| \lesssim \langle t \rangle^{-2} |k|^{-3 + \theta} M_k, \label{damp1}
\\
&|\partial_y \psi_k(t,y)| \lesssim
\Bigl(
\langle t \rangle^{-3/2} |k|^{-5/2 + \theta}
+ \langle t \rangle^{-1} |b'|\, |k|^{-2 + \theta}
\Bigr) M_k, \label{damp2}\\
&|\omega_k(t,y)| \lesssim
\min\Bigl(
\langle t \rangle^{-1} |k|^{-2 + \theta}
+ |b'|^2 |k|^{-1 + \theta},
\; |k|^{-3 + \theta}
\Bigr) M_k, \label{depl1}
\\
&\bigl|\partial_y (e^{ik t b}\,\omega_k)(t,y)\bigr|
\lesssim
\min\Bigl(
\langle t \rangle^{-1/2} |k|^{-3/2 + \theta}
+ |b'|\, |k|^{-1 + \theta},
\; |k|^{-2 + \theta}
\Bigr) M_k, \label{pro1} \\
&\bigl|\partial_y (e^{ik t b}\,\psi_k)(t,y)\bigr|
\lesssim
\langle t \rangle^{-3/2} |k|^{-5/2 + \theta} M_k, \label{pro2}\\
&\bigl\| (\partial_y, k)^2 (e^{ik t b}\,\omega_k) \bigr\|_{L^2}
\lesssim
|k|^{-1 + \theta} M_k \label{pro3},
\end{align}
\end{subequations}
where $M_k = \norm{\omega_{0k}}_{H_k^4}$ (see \eqref{hk} for a definition of $H_k^4$). 
\end{theorem}
\begin{remark}
    Estimate \eqref{pro2} gives $\langle t\rangle^{-3/2}$ decay for the first derivative of the stream-function profile, improving the $\langle t\rangle^{-6/5}$ rate in \cite[Theorem~1.2]{CJWZ25}. The present statement uses $H_k^4$ data, whereas \cite[Theorem~1.2]{CJWZ25} uses $H_k^3$ data so they are not directly comparable. The loss of $|k|^\theta$, for arbitrarily small $\theta>0$, comes from bounding the logarithmic correction in the inner solution $\Phi_\epsilon^\iota$ by an algebraic loss; see \eqref{log:loss} and the following discussion in Section~\ref{Inner}. The further loss of $|k|^{1/2}$ relative to dimensional scaling arises from the factor $|k\delta(\Lambda)|^{1/2}$ in the preliminary pointwise bound \eqref{prelim:pointwise} at the start of Section \ref{sec:main:results}. The other principal time-decay and quadratic-depletion rates agree with the Kolmogorov result and are obtained here for the shears under Assumption~\ref{MaAs}. 
\end{remark}





\subsection{Main Ideas and Outline}
We follow the basic strategy used in \cite{IIJ24} which we briefly review. By standard spectral theory, it is possible to derive the representation formula  \eqref{F3.4} for the evolution of the linearized Euler equations \eqref{F3.0}--\eqref{F3.0:elliptic}. The analysis then reduces to the regularity of the spectral density function (sdf) $\psi_{k, \epsilon}(y, \lambda)$ which satisfies the non-homogeneous Rayleigh equation
\begin{align}
\label{ray:toy}
    \left(-\p_y^2 + k^2 + \frac{b''(y)}{b(y) - \lambda + i \epsilon} \right) \psi_{k, \epsilon}(y,\lambda) = \frac{\omega_{0k}(y)}{b(y) - \lambda + i  \epsilon}, \quad \epsilon > 0.
\end{align}
Let $y_*$ be a critical point and $b(y_*)$ the associated critical value. We work locally near $y_*$. The analysis breaks into two regimes: one where $\lambda $ is away from the critical value and one where $\lambda$ is close to the critical value. We refer to these cases as \textit{non-degenerate} and \textit{degenerate} respectively. For the non-degenerate case, the singular term $(b(y) - \lambda + i \epsilon)^{-1}$ is lower order and can be treated as a relatively compact perturbation of the Laplacian $-\p_y^2 + k^2$. To capture this, we reformulate \eqref{ray:toy}  as an integral equation  
\begin{align}
\label{fred}
    \psi_{k, \epsilon} -  \Delta_k^{-1} \frac{b''(y)\psi_{k, \epsilon}}{b(y) - \lambda + i \epsilon} = -\Delta_k^{-1}  \frac{\omega_{0k}}{b(y) - \lambda + i \epsilon   } , \quad \Delta_k := \p_y^2 - k^2.
\end{align}
The solvability of \eqref{fred} with uniform-in-$\epsilon$ bounds in $H^1$ follows from Assumption \ref{MaAs}. This goes under the name of the Limiting Absorption Principle (LAP) and can be viewed as the engine that powers the proof, see Lemma~\ref{T5} and Proposition~\ref{DeC7}. In order to deduce vorticity depletion, inviscid damping, and ``profile control'' we also need to consider higher order spectral regularity. However, spectral derivatives of the sdf are not uniformly bounded in $H^1$ because of the singularities at the critical layer.  An important point, already made in \cite{IIJ24}, is that these singularities appear in a rather explicit form. Therefore, suitable singular terms can be subtracted off from the spectral density function, and the  remainder can then be bounded using the LAP.  This is done in Sections~\ref{sec:nd:bdds} and~\ref{sec:d:bdds}.

Now we consider the more difficult degenerate case. Here, we make use of the fact that we are working with a non-degenerate critical point to define the modified sdf $\phi_{k, \epsilon}$ as 
\begin{align}
    \psi_{k, \epsilon}(y ,\lambda) = \phi_{k, \epsilon}(y,\lambda) + \frac{\omega_{0k}(y)}{b''(y)},
\end{align}
where $b''(y) \neq 0$ for $y$ near $y_*$, see \eqref{asp1}.
$\phi_{k,\epsilon}$ satisfies 
\begin{align}
\label{ray:mod}
    \left(-\p_y^2 + k^2 + \frac{b''(y)}{b(y) - \lambda + i \epsilon} \right) \phi_{k, \epsilon}(y,\lambda) = (\p_y^2 - k^2) \left(\frac{\omega_{0k}}{b''(y)} \right).
\end{align}
If we assume that the initial data is $C^2$, then the right hand side of \eqref{ray:mod} is bounded. The key observation is that  when $\lambda$ is close to the critical value $b(y_*)$, $\epsilon \ll  |b(y_*) - \lambda |$, and $y$ lies in the annulus $\ell:=|\lambda-b(y_*)|^{1/2} \ll |y_*-y| \ll 1$, we approximately have
\begin{align}
    \frac{b''(y)}{b(y) -\lambda + i\iota \epsilon} \approx  \frac{2}{|y - y_*|^2 }.
\end{align}

The potential is \textit{critically singular with a favorable sign}. Due to the criticality, it can not be treated perturbatively and must be included when constructing a modified Green's function. Pointwise bounds for the Green's function of the approximate operator can then be computed, see Lemmas \ref{mGk50}, \ref{mGk30}. With the pointwise bounds for the approximate operator in hand, \eqref{ray:mod} can  be reformulated as an integral equation  
\begin{align}
\label{fred:d}
    \phi_{k, \epsilon} + T_{k, \epsilon} \phi_{k, \epsilon} = \mathcal{L}^{-1} ( (\p_y^2 - k^2) (\frac{\omega_{0k}}{b''(y)})), 
\end{align}
where $\mathcal{L}$ satisfies, roughly speaking,
\begin{align}
    \mathcal{L} \approx -\p_y^2 + k^2 + \frac{2}{|y-y_*|^2}\mathbf{1}_{|y-y_\ast|\ge \delta}.
\end{align}
Through the incorporation of the potential $\frac{2}{|y - y_*|^2}$,    $T_{k,\epsilon}$ can be viewed perturbatively. To capture the smallness induced by the favorable potential, we prove uniform bounds with respect to $\epsilon$ in weighted spaces $X^{\sigma_1, \sigma_2}$, see  \eqref{DeC1} for the definition.  With an appropriate weight, this method can capture up to, but not including, quadratic depletion at the critical point. The reason for this gap is the presence of a logarithmic factor in the sdf. Approximately, the sdf can be thought of as  
\begin{align}
    \phi_{k, \epsilon}(y, \lambda) \approx (b(y) - \lambda + i \epsilon) \log(b(y) - \lambda + i \epsilon). 
\end{align}
Such logarithmic factors are natural in the context of Frobenius expansions \cite[Chapter~4]{CL55}. More explicitly, when $\lambda = 0$ and $y\ne0$ we have 
\begin{align}
    \left(-\p_y^2 + \frac{2}{y^2}\right) \left(-\frac{y^2 \log|y|}{3}\right) = 1.
\end{align}
See Section \ref{Inner}, especially the proof of Lemma~\ref{lem:phi-bounds} for further discussion. The presence of the logarithmic factor was also noted in \cite{BG24,BG25}. In the arXiv version of \cite{BG24} it was suggested that this logarithmic factor at the level of the sdf leads to a logarithmic factor in the pointwise depletion and decay rate \cite[Theorem~2.1 and the preceding discussion]{BG24}. Here, we show that this is not the case.

Remarkably, the logarithmic factor drops out when the sdf is integrated with respect to $\lambda$. In order to understand how this cancellation comes about, consider the  following ``inner problem:'' 
\begin{align}
\label{toy:inner}
    -\p_y^2 \Phi_\epsilon^\iota (y,\lambda) + \frac{b''(y)}{b(y) - \lambda + i \iota \epsilon} \Phi_\epsilon^\iota(y, \lambda) = 1, \quad  \iota \in \{ \pm \}, \,  \Phi_\epsilon^\iota(\pm 1, \lambda) = 0.
\end{align}
If we assume that $b$ is even around 0, it is possible to use reflection to consider the problem on the interval $[0,1]$ with a Neumann boundary condition at 0 and a Dirichlet boundary condition at 1. In this case, $\varphi(y, \lambda): = b(y) - \lambda + i \iota \epsilon $ is a homogeneous solution which satisfies the left boundary condition. Using reduction of order, it is possible to construct another solution which satisfies the right boundary condition:
\begin{align}
    H_{r}(y,\lambda) := \varphi(y, \lambda) \int_y^1 \frac{ds}{\varphi^2(s ,\lambda)}.
\end{align}
From these two homogeneous solutions, we have an explicit representation formula of  the solution $\Phi_\epsilon^\iota$. We can decompose  $\Phi_\epsilon^\iota $ into two parts: one which vanishes at $y_\lambda$ with $b(y_\lambda) = \lambda$ as $\epsilon $ goes to 0 and another part which scales like $y_\lambda^2$ \textit{without} a logarithmic factor. From the representation formula \eqref{F3.4}, the dynamically relevant quantities are the \textit{difference} $\Phi_\epsilon^{+}(y,\lambda)- \Phi_\epsilon^-(y,\lambda)$ and $\Phi_\epsilon^{\iota}(y_\lambda, \lambda)$, see \eqref{vanishing:rep}, for example. In a distributional sense, the limiting difference sdf $\Phi(y,\lambda) :=  \lim_{\epsilon \to 0}\Phi_\epsilon^{+}(y,\lambda)- \Phi_\epsilon^-(y,\lambda)$ satisfies
\begin{align}
 &-\p_y^2 \Phi(y,\lambda) + \text{P.V.} \frac{b''(y)}{b(y) - \lambda} \Phi(y,\lambda)  \\
 &= \sum_{b(y_\lambda) = \lambda, b'(y_\lambda) \neq 0 }
 \frac{  \pi i  b''(y_\lambda) \lim_{\epsilon \to 0}( \Phi_\epsilon^-(y_\lambda, \lambda)  + \Phi_\epsilon^+(y_\lambda, \lambda) ) }{|b'(y_\lambda)|} \delta(y-y_{\lambda}).
\end{align}
The essential observation is that  $\Phi$ is generated by the critical layers where the logarithmic dependence vanishes. This is ultimately how we show the optimal quadratic depletion and is the primary new idea in this paper.
$\Phi$ only captures the leading order contribution to the modified sdf. Therefore, we also need a remainder term $E_{k,\epsilon}^\iota$ which handles the missing $k$ dependence and accounts for the fact that the right hand side of the modified sdf equation $\phi_{k,\epsilon}$ is not constant. The logarithmic factor appearing in $\Phi_\epsilon^\iota$ manifests itself in $E_{k,\epsilon}^\iota$ through  terms of the form $k^2 \Phi_\epsilon^\iota$. This leads to losses in $k$ quantified by the small parameter $\theta \in (0,1/4]$, see \eqref{eq:Erhs} and \eqref{gr:bdds}. Since we ultimately sum over $k$, we view this non-optimal $x$ regularity as being acceptable. Just as in the non-degenerate case, we also need to understand higher order spectral regularity. This follows along similar lines to the non-degenerate case. With the bounds on the sdf in hand, the main results follow from fairly standard (but at times tedious) manipulations of the representation formula \eqref{F3.4}. This is done in Section \ref{sec:main:results}.


\subsection{Related results and comparison with existing literature}

 For monotone shear flows in a finite channel, Wei, Zhang and Zhao \cite{WZZ18} established sharp linear velocity decay and scattering in Sobolev spaces. The third author \cite{J20} gave an alternative proof and obtained precise pointwise asymptotics. In \cite{J20G}, the third author established linear inviscid damping and scattering in Gevrey spaces for a class of spectrally stable monotone shear flows in a finite channel, assuming that the initial vorticity perturbation and the second derivative of the shear profile are supported away from the channel boundaries. At the nonlinear level, in a breakthrough work Bedrossian and Masmoudi \cite{BM15} proved inviscid damping near Couette flow for sufficiently small Gevrey perturbations. Ionescu and the third author \cite{IJ23}, and independently Masmoudi and Zhao \cite{MZ24} extended nonlinear damping to classes of spectrally stable monotone shears in a channel, under Gevrey regularity and suitable localization assumptions. 
 
 Critical points of the shear introduce the additional depletion mechanism studied here. An earlier manifestation of vorticity depletion was observed by Bassom and Gilbert \cite{BG98} in their asymptotic and numerical study of vortex axisymmetrization: vorticity perturbations near the center of a Gaussian vortex were more strongly suppressed than a passively advected scalar. Bouchet and Morita \cite{BM10} identified the depletion mechanism for shear flows and connected it to this earlier observation in the vortex setting. The basic observation is that the nonlocal term in the linearized Euler equations changes the behavior predicted by passive transport near a critical point. Namely, the limiting vorticity profile vanishes at the stationary streamlines, which suppresses the leading stationary phase contribution to the velocity. Wei, Zhang and Zhao \cite{WZZ19} gave a rigorous proof of vorticity depletion for a class of non-monotonic shear flows and established optimal $L^2$ inviscid damping estimates under a symmetry assumption. They also studied the Kolmogorov flow in \cite{WZZ20}, proving linear inviscid damping and enhanced dissipation. Compared with the current approach, Wei, Zhang and Zhao \cite{WZZ19,WZZ20} adopt a direct approach by solving for the sdf in a rather explicit way. A similar approach was used for radial vortices by Bedrossian, Coti Zelati and Vicol \cite{BCV19}, who proved linear axisymmetrization and vorticity depletion for a class of strictly monotone radial profiles. We also mention the work of Lin and Xu \cite{LX19} who proved a RAGE theorem (essentially time integrated inviscid damping) for Euler linearized around Kolmogorov flow.  

The work of Ionescu, Iyer and the third author \cite{IIJ24} and earlier work on the vortex symmetrization problem  by Ionescu and the third author \cite{IJ} are much closer in spirit to the current work as outlined above. We point out that the original motivation for the construction of $\Phi_\epsilon^\iota$ comes from the $k=1$ case for Euler linearized around a monotone vortex. In that case, the Rayleigh equation admits a particularly simple representation for the Green's function \cite{BCV19, IJ}. The cancellation between the two different branches of the sdf leads to crucial localization properties.

In \cite{BCJ26}, we proved inviscid damping, vorticity depletion and enhanced dissipation estimates that are uniform with respect to viscosity for periodic shear flows with two non-degenerate critical points, following a combination of the strategies developed in \cite{IIJ24} and \cite{J23}. The depletion estimates just miss the quadratic endpoint. 

Bian and Grenier \cite{BG24,BG25} study the Rayleigh equation by directly constructing the Green's function through a construction of appropriate homogeneous solutions. As noted above, they explicitly identify the logarithmic factors which appear around the critical layer.  A key difference with the current approach is that the representation formula used is more akin to a Laplace inversion so the contour is not a closed curve. This makes identifying the cancellation used in the present paper harder to extract from their analysis. However, this does not preclude that the logarithmic loss also integrates out from their representation as well. 

The recent work of Chen, the third author, Wei and Zhang \cite{CJWZ25} proves nonlinear asymptotic metastability of the Kolmogorov flow for vorticity perturbations of size $\nu^{1/3}$ in $H^3$. A key ingredient is a detailed linear analysis giving optimal inviscid damping and vorticity depletion. Their proof uses an almost conserved second order quantity and a coercive operator inequality which exploit the particular structure of the Kolmogorov flow. Our result shows that the optimal linear bounds extend to the class of spectrally stable periodic shears considered here. We emphasize that the endpoint in the depletion estimate does not imply optimal regularity assumptions for all the bounds.

\subsection{Notation}


Below we list some notations that we reserve for important quantities and briefly explain their roles. 
\begin{itemize}
    
\item \textit{Sobolev Spaces}:
For $\kappa k \in \mathbb{Z} \setminus \{ 0\}$ and $m \in \mathbb{N}$, we define
\begin{align}
\label{hk}
    \norm{f}_{H_k^m}^2 := \sum_{\alpha =0}^m |k|^{2(m - \alpha)} \norm{ \p_y^\alpha f}_{L^2}^2.
\end{align}


\item \textit{Continuous Spectrum}: We denote the continuous spectrum for the operator $L_k$ from \eqref{asp2} as
 \begin{align}
\label{rng:b}
    \Sigma :=  \{ b(y) \, |  \, y \in \T \}. 
\end{align}

\item {\it Derived parameters $\delta_0,\delta(\Lambda), \, \Sigma_{j,\delta_0}, \,S^j_d$}:
These parameters are introduced to simplify notations. It follows from Assumption \ref{MaAs} that there exists a $\delta_0\in(0,1/32)$ such that
\begin{equation}\label{IntN4.0}
\begin{split}
&|y_{1\ast}-y_{2\ast}|>10\delta_0,\\
&\delta_0\sup_{|y-y_{j\ast}|<4\delta_0}|b'''(y)|<|b''(y_{j\ast})|/10,\qquad j\in\{1,2\}.
\end{split}
\end{equation}
Define for $j\in\{1,2\}$ the set
\begin{equation}\label{IntN4.2}
\Sigma_{j,\delta_0}:=\Big[b(y_{j\ast})-|b''(y_{j\ast})|\delta_0^2/16,\, b(y_{j\ast})+|b''(y_{j\ast})|\delta_0^2/16\Big] \cap \Sigma.
\end{equation}
The constant $\epsilon_0$ is chosen so that
\begin{align*}
0<\epsilon_0\leq\frac{\delta_0^2}{10}.
\end{align*}

We introduce for $j\in\{1,2\}, d\in(0,1/10)$ the shorthand
\begin{equation}\label{IntN4.3}
S^j_d:=[y_{j\ast}-d,y_{j\ast}+d]. 
\end{equation}

\item {\it Composite notations $\Lambda, \Theta$}:
We use $\Lambda=(\lambda,\epsilon)$ to record the spectral parameter and resolvent regularization, and $\Theta=(k,\lambda,\epsilon)$ when the Fourier frequency is also needed. In a degenerate neighborhood $\Sigma_{j,\delta_0}$, the critical-point index $j$ is fixed and its dependence is suppressed in $\Lambda$ and $\Theta$. When the two resolvent branches are distinguished, their sign $\iota\in\{+,-\}$ is indicated separately. 


\item {\it Weights}: We use weighted estimates to quantify depletion. For $\lambda\in\Sigma_{j,\delta_0}$, define
\begin{equation}\label{Intd1}
\delta(\Lambda):= 8\sqrt{(\lambda-b(y_{j\ast}))/b''(y_{j\ast})  + |\epsilon|  }.
\end{equation}

For $\lambda\in\Sigma_{j, \delta_0} $ and $j\in\{1,2\}$ with 
\begin{equation}\label{Intd2}
   \delta(\Lambda)\leq  1/8,
\end{equation}
we define the weights for $y\in\T$,
\begin{equation}\label{Intd3}
    \varrho_j(y ; \Lambda):=|y-y_{j\ast}|+\delta(\Lambda),\quad \varrho_{j,k}(y ; \Lambda):=\varrho_j(y ; \Lambda)\wedge \frac{1}{|k|},\quad d_{j,k}:=\delta(\Lambda)\wedge \frac{1}{|k|}.
\end{equation}
In the above, $\alpha\wedge\beta$ is the minimum of two real numbers $\alpha,\beta$.

We also define the weight centered at 0:
\begin{equation}
    \varrho(y;\Lambda) := |y| + \delta(\Lambda).
\end{equation}

\item \textit{Cutoff notations}: Let $\chi\in C_c^\infty(-2,2)$ be an even smooth cutoff function with $\chi\equiv 1$ on $[-1,1]$, $0\leq\chi\leq1$, and $\chi$ nonincreasing on $[0,\infty)$. We define the following shorthand
\begin{subequations}
\begin{align}
    \chi_{j,l}(y)&:=\chi \left(\frac{y-y_{j,*}}{l} \right),\\
    \chi_{*,l}(y)&:=\chi_{1,l}(y)+\chi_{2,l}(y),\\
    \chi_{j, l, l'} (y)&:=\begin{cases}
    \chi_{j,l}(y)-\chi_{j,l'}(y),& l>l'\\
    0,& l\leq l'
    \end{cases}, \\
    \chi_l(y) &:= \chi\left(\frac{y}{l} \right),\\
    \chi_{bj, l} (\lambda) &:= \chi\left(\frac{\lambda-b(y_{j*})}{l} \right).
    \end{align}
\end{subequations}
We remark that $\chi_{bj, l}$ will be used in Section \ref{sec:main:results} for decompositions involving the spectral parameter $\lambda$ and its proximity to the critical values. We also introduce the following notation for the annular cutoff at scale $\delta(\Lambda)$
  \begin{align}
       \chi_{j,\delta}^*(y) = \chi_{j, \delta(\Lambda), c\delta(\Lambda)},
  \end{align}
where $c \in (0,1)$ is a small constant chosen so that $|\lambda - b(y)| \approx |b(y_{j*}) - \lambda|$ for $|y - y_{j*}|\lesssim c\delta(\Lambda)$
when $ \epsilon \lesssim  |b(y_{j*}) - \lambda|$.
\item \textit{Good Derivative}: To measure profile regularity we use the differential operator 
\begin{align}
    D_\lambda := b'(y)\p_\lambda + \p_y .
\end{align}
Through the spectral representation and integration by parts in $\lambda$, $D_\lambda$ corresponds to the physical-space derivative $\partial_y+t b'(y)\partial_x$, which commutes with $\partial_t+b(y)\partial_x$. 


\item \textit{Singularity Notation}: In Sections \ref{sec:nd:bdds} and \ref{sec:d:bdds} we use the following notation to track the singular expansions $\mathcal{S}_{a,b,\epsilon}^\iota$  and remainder terms $r_{a,b,\epsilon}^\iota$ 
\begin{equation}
    \partial_\lambda^aD_\lambda^b\psi_{k,\epsilon}^\iota = \mathcal{S}_{a,b,\epsilon}^\iota+r_{a,b,\epsilon}^\iota.
\end{equation}

We write $\mathcal{R}_{a,b,\epsilon}^\iota$ for the right-hand side of the Rayleigh equation satisfied by the remainder after subtracting the singular term.

\end{itemize}

\section{Spectral property and representation formula}\label{rep}
For each $\kappa k\in\mathbb{Z}\backslash\{0\}$, recall from \eqref{asp2} that for any $g\in L^2(\T)$,
 \begin{equation}\label{F3.1}
 L_kg(y)=b(y)g(y)-b''(y)h(y).
 \end{equation}
Then \eqref{F3.0}--\eqref{F3.0:elliptic} can be reformulated abstractly as
 \begin{equation}\label{F3.2}
 \partial_t\omega_k+ikL_k\omega_k=0.
 \end{equation}
  In contrast to the spectral property of the linearized operator around monotonic shear flows, the spectral property of $L_k$ is less understood, especially on the generation of discrete eigenvalues and embedded eigenvalues. From general spectral theory, we know that the spectrum of $L_k$ consists of the continuous spectrum  $\Sigma$ \eqref{rng:b} together with some discrete eigenvalues with nonzero imaginary part which can only accumulate at the set of continuous spectrum $\Sigma$. Unlike the case of monotonic shear flows where the discrete eigenvalues can accumulate only at inflection points of the background shear flow, there appears no simple characterization of the possible accumulation points for non-monotonic shear flows. See Faddeev \cite{F73} for further discussion of the spectral properties of linearized Euler.  
 
 Recall that $\lambda\in\Sigma$ is called an embedded eigenvalue if there exists a nontrivial $g\in L^2(\T)$, such that 
 \begin{equation}\label{F3.20}
 L_kg=\lambda g. 
 \end{equation}
 For non-monotonic shear flows, this definition is too restrictive, as accumulation points of discrete eigenvalues may no longer be embedded eigenvalues. To capture the discrete eigenvalues, we recall the following definition of ``generalized embedded eigenvalues'', which can be found in \cite{WZZ19}, adapted to our setting. 
 
 \begin{definition}\label{emb1}
 We call $\lambda\in\Sigma$ a generalized embedded eigenvalue, if one of the following conditions is satisfied. 
 \begin{itemize}
 \item  $\lambda$ is an embedded eigenvalue.
 
 \item $\lambda\neq b(y_{1*}), b(y_{2*})$ and there exists a nontrivial $\psi\in H^1(\T): \T \to \mathbb{C}$ such that in the sense of distributions,
 \begin{equation}\label{emb2}
 (k^2-\partial_y^2)\psi(y)+{\rm P.V.}\frac{b''(y)\psi(y)}{b(y)-\lambda}+i\pi\sum_{z\in \T, \,b(z)=\lambda}\frac{b''(z)\psi(z)}{|b'(z)|}\delta(y-z)=0.
 \end{equation}
 \end{itemize}
 \end{definition}
 
 


We have the following Proposition for the representation of the solution operator for solutions of \eqref{F3.2}. See \cite{IIJ24} for a proof.  
 \begin{proposition}\label{F3.3}
 Suppose that $\kappa k\in\Z\backslash\{0\}$ and $\omega_{0k}\in L^2(\T)$. Then the stream function $\psi_k(t,y)$ for $\kappa k\in\mathbb{Z}\backslash\{0\}, y\in \T , t\ge0$ has the representation 
 \begin{equation}\label{F3.4}
 \psi_k(t,y)=-\frac{1}{2\pi i}\lim_{\epsilon\to0+}\int_{\Sigma}e^{-ik\lambda t}\left[\psi_{k,\epsilon}^{-}(y,\lambda)-\psi_{k,\epsilon}^{+}(y,\lambda)\right]d\lambda,
 \end{equation}
 where $\psi_{k,\epsilon}^{\iota}(y,\lambda)$ for $\iota\in\{+,-\}, \,y\in \T ,\kappa k\in\mathbb{Z}\backslash\{0\}$,  and  sufficiently small $\epsilon\in[-1/4,1/4]\backslash\{0\}$, are the solutions to
 \begin{equation}\label{F7}
 \begin{split}
  \left(-\frac{d^2}{dy^2} + k^2+\frac{b''(y)}{b(y)-\lambda+i\iota\epsilon} \right) \psi_{k,\epsilon}^{\iota}(y,\lambda)=\frac{\omega_{0k}(y)}{b(y)-\lambda+i\iota\epsilon}.
 \end{split}
 \end{equation}
 \end{proposition}

\begin{remark}\label{F6.1}
The existence of $\psi^{\iota}_{k,\epsilon}$ for sufficiently small $\epsilon\neq0$ follows from Assumption \ref{MaAs}, which implies the solvability of \eqref{F7} for sufficiently small $\epsilon\neq0$.

\end{remark}

 
\section{Bounds on Green's function and Modified Green's function}



In this section we recall some useful estimates for Green's functions associated with Rayleigh equations, including the Green's function for the standard elliptic operator $k^2-\partial_y^2$ on $\T$ and for the elliptic operator $k^2-\partial_y^2+V(y)$ on $\T$ where $V$ is a potential term coming from the nonlocal term in the Rayleigh equations. Throughout the section we assume that $\kappa k\in\Z\cap[1,\infty)$.  The proofs of the following results can be found in \cite{BCJ26}, Section 2.
\subsection{Elementary properties of the standard Green's function}\label{sub1}
Recall that the Green's function $G_k(y,z)$ solves  for $y, z\in\T$,
\begin{equation}\label{eq:Helmoltz}
-\frac{d^2}{dy^2}G_k(y,z)+k^2G_k(y,z)=\delta(y-z).
\end{equation}
 $G_k$ has the symmetry
\begin{equation}\label{Gk2}
G_k(y,z)=G_k(z,y), \qquad {\rm for}\,\,y, z\in\T.
\end{equation}
We note the following point-wise bounds for $G_k$ with $y,z\in\T$,
\begin{equation}\label{Helm1.0}
    |G_k(y,z)|+|k|^{-1}|\partial_yG_k(y,z)|\lesssim |k|^{-1}e^{-|k||y-z|},
    \end{equation}
and the integral bounds
\begin{equation}\label{Gk1.1}
\begin{split}
\sup_{y\in\T,\beta\in\{0,1\}}\bigg[|k|^{3/2-\beta}\left\|\partial_{y,z}^{\beta}G_k(y,z)\right\|_{L^2(z\in\T)}\bigg]\lesssim 1.
\end{split}
\end{equation}

Define $F_k(y,z)$ for $ y, z\in\T$ as
\begin{equation}\label{GkFk}
F_k(y,z):=\partial_z\partial_yG_k(y,z)-\delta(y-z),\qquad{\rm for}\,\,y,z\in\T.
\end{equation}
Then $F_k$ satisfies the bounds
\begin{equation}\label{Fk.1}
    |F_k(y,z)|+|k|^{-1}|\partial_yF_k(y,z)|\lesssim |k|e^{-|k||y-z|},
\end{equation}
and
\begin{equation}\label{Gk3.1}
\begin{split}
&\sup_{y\in\T,\beta\in\{0,1\}}\bigg[|k|^{-1/2-\beta}\left\|\partial_{y,z}^{\beta}F_k(y,z)\right\|_{L^2(z\in\T)}\bigg]\lesssim 1.
\end{split}
\end{equation}



\subsection{Bounds on the modified Green's function}\label{sec:modifiedgreen}
For applications below, we need to study the ``modified Green's function'' $\mathcal{G}^j_{k}(y,z;\Lambda)$ defined for the parameters $j\in\{1,2\}, \lambda\in \Sigma_{j,\delta_0}$, and $0<|\epsilon|<\epsilon_0$, which satisfies for $y,z\in\T,$
\begin{equation}\label{mGk1}
(k^2-\partial_y^2)\mathcal{G}^j_{k}(y,z;\Lambda)+\frac{b''(y)}{b(y)-\lambda +i\iota\epsilon}\Big[\chi_{j, \delta_0, \delta(\Lambda)}(y) \Big]\mathcal{G}^j_{k}(y,z;\Lambda)=\delta(y-z).
\end{equation}

If $\delta(\Lambda)\geq\delta_0$, the annular cutoff vanishes and $\mathcal G_k^j=G_k$. Since $\delta(\Lambda)\lesssim\delta_0$, the regime $\delta(\Lambda)\sim\delta_0$ is covered by standard Green function and local elliptic estimates; the degenerate estimates are needed only when $\delta(\Lambda)\ll\delta_0$.

The bounds we need for the modified Green's function $\mathcal{G}^j_k(y,z;\Lambda)$ are the following.

\begin{lemma}\label{mGk50} 

Let $\mathcal{G}^{j}_{k}(y,z;\Lambda)$ be defined as in \eqref{mGk1} for $j\in\{1,2\}, \lambda\in \Sigma_{j,\delta_0}$ and $0<|\epsilon|<\epsilon_0,\delta(\Lambda)\leq1/8.$ 
Then we have for $\beta\in\{0,1\}$ and $y, z\in\T$, $\mathcal{G}^j_{k}(y,z;\Lambda)=\mathcal{G}^j_{k}(z,y;\Lambda)$, and

\begin{equation}\label{mGk51}
\varrho_{j,k}^\beta(y;\Lambda) \big|\partial_y^{\beta}\mathcal{G}^j_{k}(y,z;\Lambda)\big| \lesssim  \varrho_{j,k}(z;\Lambda) \min\Big\{e^{-|k||y-z|}, \frac{\varrho_j^2(y;\Lambda)}{\varrho_j^2(z;\Lambda)},\,\frac{\varrho_j(z;\Lambda) }{\varrho_j(y;\Lambda)}\Big\}.
\end{equation}

\end{lemma}

As a corollary of Lemma \ref{mGk50}, we also have the following bounds.

\begin{lemma}\label{mGk30}
Under the assumptions of Lemma \ref{mGk50},
 define for $y,z\in \T$,
 \begin{equation}\label{mGk31}
\mathcal{H}^j_k(y,z;\Lambda):=\Big[\partial_z+\chi_{j, 10\delta(\Lambda)}(y)\partial_y\Big]\mathcal{G}^j_{k}(y,z;\Lambda).
\end{equation}
Then we have the bounds for $y\in\T, z\in S^j_{4\delta(\Lambda)}$ and $\beta\in\{0,1\}$,
\begin{equation}\label{mGk52}
\begin{split}
&\varrho_{j,k}^\beta(y;\Lambda)\big|\partial_y^\beta\mathcal{H}^j_{k}(y,z;\Lambda)\big|\lesssim \min\Big\{e^{-|k||y-z|}, \frac{\varrho_j^2(y;\Lambda)}{\varrho_j^2(z;\Lambda)},\,\frac{\varrho_j(z;\Lambda) }{\varrho_j(y;\Lambda)}\Big\},
\end{split}
\end{equation}
and for $y,z\in\T$,
\begin{equation}\label{lem:mGK30:New}
\varrho_{j,k}^\beta(y;\Lambda)\big|\partial_y^\beta(\partial_z + \partial_y) \mathcal{G}^{j}_k(y,z;\Lambda) \big|\lesssim \min \left\{e^{-(|k|/2)|y-z|}, \frac{\varrho_{j,k}(y;\Lambda)}{\varrho_{j}(z;\Lambda)}, \frac{\varrho_{j,k}(z;\Lambda)}{\varrho_{j}(y;\Lambda)}    \right\}.
\end{equation}



\end{lemma}

\section{The Limiting Absorption Principle }

In this section we control the full solution of the Rayleigh equations by its main part. 
It turns out that the situation is very different for the spectral range $\lambda\in \Sigma\backslash\cup_j\Sigma_{j,\delta_0}$ (the non-degenerate case) and $\lambda\in\cup_j \Sigma_{j,\delta_0}$ (the degenerate case). We first consider the non-degenerate case. 




\subsection{The non-degenerate case}
Fix $\epsilon\in[-1/4,1/4]\backslash\{0\}, \lambda\in \Sigma\backslash \cup_{j} \Sigma_{j, \delta_0},  \kappa k\in\mathbb{Z}\backslash\{0\}$. Define for each $g\in L^2(\T)$ the operator
\begin{equation}\label{T1}
T_{\Theta}g(y):=\int_\T G_k(y,z)\frac{b''(z)g(z)}{b(z)-\lambda+i\epsilon}dz.
\end{equation}

\begin{lemma}\label{T2}
For  $\epsilon\in[-1/4,1/4]\backslash\{0\}, \lambda\in \Sigma\backslash \cup_{j} \Sigma_{j, \delta_0},  \kappa k\in\mathbb{Z}\backslash\{0\}$, the operator $T_\Theta$ satisfies the bound
\begin{equation}\label{T3}
\|T_\Theta g\|_{H^1_k(\T)}\lesssim |k|^{-1/3}\|g\|_{H^1_k(\T)},\qquad {\rm for\,\,all}\,\,g\in H^{1}_k(\T).
\end{equation}
In addition, we have the more precise regularity structure
\begin{equation}\label{T3.1}
\begin{split}
&\bigg\|\partial_yT_\Theta g(y)+\frac{b''(y)(1-\chi_{\ast,\delta_0/10})g(y)}{b'(y)}\log{(b(y)-\lambda+i\epsilon)}\bigg\|_{W^{1,1}(\T)}\\
&\lesssim \langle k\rangle^{4/3}\|g\|_{H^1_k(\T)}.
\end{split}
\end{equation}

\end{lemma}

\begin{lemma}\label{T5}
There exist $\epsilon_0, \mu_0>0$, such that the following statement holds. For all $\lambda\in \Sigma\backslash \cup_{j} \Sigma_{j, \delta_0}, \kappa k\in\mathbb{Z}\backslash\{0\}, 0<|\epsilon|<\epsilon_0$ and any $g\in H^1_k(\T)$, we have the bound
\begin{equation}\label{T6}
\|g+T_\Theta g\|_{H^1_k(\T)}\ge\mu_0 \|g\|_{H^1_k(\T)}.
\end{equation}
\end{lemma}
The proofs of both Lemmas follow from Lemmas 4.1 and 4.2 in \cite{IIJ24}.

\subsection{The degenerate case }

For $\lambda\in \Sigma_{j,\delta_0}, \kappa k\in  \Z\backslash\{0\}, y\in \T$ and $0<|\epsilon|<\epsilon_0$
 we define the family of spaces $X^{\sigma_1, \sigma_2}$ as 
\begin{align}\label{DeC1}
    \norm{g}_{X^{\sigma_1, \sigma_2}} =  &\sum_{\beta\in\{0,1\}}\delta^{-1/2}(\Lambda)\Big\| \delta^{\sigma_1}(\Lambda) d_{j,k}^{\sigma_2+\beta}\partial_y^\beta g\Big\|_{L^2(S^j_{\delta(\Lambda)})}\\
&+\sum_{\beta\in\{0,1\}}\Big\|\varrho_{j}^{\sigma_1} \varrho_{j,k}^{\sigma_2+\beta}(\cdot;\Lambda)\partial^\beta_yg\Big\|_{L^\infty(\T\backslash S^j_{\delta(\Lambda)})}. \notag
\end{align}

\begin{remark}
    We comment on the design of the norm in the $S_{\delta(\Lambda)}^j$ region. This region contains the critical layers which logarithmically fail to be bounded and are thus incompatible with the $L^\infty$ type norms used away from the critical layers. To capture the pointwise vanishing the $H^1$ norm is appropriately weighted by $\delta^{-1/2}(\Lambda)$ to ensure $L^\infty$ control by $(\delta(\Lambda))^{-\sigma_1- \sigma_2}.$ However, when $|k \delta(\Lambda)| \gg 1$, using this norm incurs a penalty like $ |k \delta(\Lambda)|^{1/2}$ when used to control $L^\infty$. Nevertheless, this loss can be compensated by taking sufficiently regular data. This is one of the main reasons for the difference in $k$ dependence between Theorem \ref{thm2} and \cite[Theorem~1.2]{CJWZ25}. 
\end{remark}

Fix $j\in\{1,2\}$, $0<|\epsilon|<\epsilon_0$, $\lambda\in\Sigma_{j,\delta_0}$, and $\kappa k\in\mathbb{Z}\backslash\{0\}$. Define for each $g\in L^2(\T)$ the operator
\begin{equation}\label{DeC3}
T_\Theta g(y):= \int_\T \mathcal{G}_{k}^j(y,z;\Lambda)\bigg[1-\chi_{j,\delta_0, \delta(\Lambda)}(z) \bigg]\frac{b''(z)g(z)}{b(z)-\lambda+i\epsilon}dz.
\end{equation}
Then we have the following bounds for $T_\Theta$.
\begin{lemma}\label{DeC4}
 For $0<|\epsilon|<\epsilon_0, \lambda\in\Sigma_{j,\delta_0},\,  \kappa k\in\mathbb{Z}\backslash\{0\},$ the operator $T_\Theta$ satisfies the bound for $X^{\sigma_1, \sigma_2}  $ with $\sigma_1 \in [0,1], \sigma_2 \in [-2,0], \sigma_1 + \sigma_2 \in [-2, 1]$,
\begin{equation}\label{DeC5}
\|T_\Theta g\|_{X^{\sigma_1, \sigma_2}}\lesssim (1 + |k| \delta(\Lambda))^{-1/4} \|g\|_{X^{\sigma_1, \sigma_2}},\quad {\rm for\,\,all}\,\,g\in H^{1}_k(\T) \cap W^{1, \infty}(\T).
\end{equation}
Additionally, we have the expansion
\begin{align}
\label{DeC6}
\p_y T_\Theta g = -\chi_{j,\delta(\Lambda)}^*(y)\frac{b''(y)}{b'(y)}g(y)\log\frac{b(y)-\lambda+i\epsilon}{\delta^2(\Lambda)}+L_0(y,\lambda),
\end{align}
where $L_0$ satisfies the bounds 
\begin{align}
\label{DeC6.4}
  \norm{\p_y L_0(y,\lambda)}_{L^1(S_{3 \delta(\Lambda)}^j)} \lesssim    
\frac{(\delta(\Lambda))^{-\sigma_1 - \sigma_2 }}{d_{j,k}}  \|g\|_{X^{\sigma_1, \sigma_2}}.
\end{align}
\end{lemma}

\begin{proof}
For simplicity of notation we drop $\Theta$ and decompose as 
\begin{equation}\label{DeC5.1}
T_\Theta  g(y):=T_1 g(y)+T_2g(y),
\end{equation}
where 
\begin{equation}\label{DeC5.2}
\begin{split}
&T_1g(y):=\int_{\T}\mathcal{G}_{k}^j(y,z;\Lambda)\bigg[1- \chi_{j, \delta_0}(z)\bigg]\frac{b''(z)g(z)}{b(z)-\lambda+i\epsilon}dz,\\
&T_2g(y):=\int_\T \mathcal{G}_{k}^j(y,z;\Lambda)\chi_{j,\delta_0\wedge\delta(\Lambda)}(z)\frac{b''(z)g(z)}{b(z)-\lambda+i\epsilon}dz.
\end{split}
\end{equation}
For the rest of the proof we assume $\delta(\Lambda)<\delta_0$, so that the cutoff in $T_2$ is $\chi_{j,\delta(\Lambda)}$. Otherwise, for $\delta(\Lambda)\geq\delta_0$ we have either $|\lambda-b(y_{j*})|\gtrsim\delta_0^2$ or $|\epsilon|\gtrsim\delta_0^2$. The conclusions \eqref{DeC5} and \eqref{DeC6.4} follow from similar calculations as in the non-degenerate case.

It follows from the bounds on modified Green's function $\mathcal{G}_{k}^j(y,z;\Lambda)$, the support of $1 - \chi_{j, \delta_0}(y)$, and $\lambda \in \Sigma_{j, \delta_0}$ that 
\begin{equation}\label{DeC13.4}
    \left|\varrho_{j,k}^\beta(y;\Lambda)\partial_y^\beta T_1g(y)\right|\lesssim\varrho_j^2(y;\Lambda)|k|^{-1}\|g\|_{L^\infty(\T\backslash S_{\delta_0}^j)},
\end{equation}
\begin{equation}\label{DeC5.3}
\big\|T_1g\big\|_{X^{\sigma_1, \sigma_2}}\lesssim |k|^{-1/2}\big\|g\big\|_{X^{\sigma_1, \sigma_2}}.
\end{equation}
In addition, by \eqref{mGk1},
\begin{equation}\label{DeC5.3.1}
    \begin{split}
        \big\|\partial_y^2 T_1g\big\|_{L^1(S^j_{3\delta(\Lambda)})}\lesssim&(\delta(\Lambda))^{-2}(1+|k|\delta(\Lambda))^2\|T_1g\|_{L^1(S^j_{3\delta(\Lambda)})}\\
        \lesssim&\delta(\Lambda)(1+|k|\delta(\Lambda))^2|k|^{-1+\sigma_2}\|g\|_{X^{\sigma_1, \sigma_2}}\\
        \lesssim&\frac{(\delta(\Lambda))^{-\sigma_1 - \sigma_2 }}{d_{j,k}}  \|g\|_{X^{\sigma_1, \sigma_2}}.
    \end{split}
\end{equation}
To prove \eqref{DeC5}, it suffices to prove 
\begin{equation}\label{DeC5.4}
\|T_2g\|_{X^{\sigma_1, \sigma_2} }\lesssim  (1 + |k| \delta(\Lambda))^{-1/4}\|g\|_{X^{\sigma_1, \sigma_2}}.
\end{equation}
We assume momentarily that $|\epsilon|\lesssim |\lambda-b(y_{j\ast})|$ and explain how to remove this assumption at the end of the proof.  
We decompose further as
\begin{equation}\label{DeC5.5}
\begin{split}
T_2 g(y)&=\int_\T \mathcal{G}_{k}^j(y,z;\Lambda)  \chi_{j, c\delta(\Lambda)}(z)   \frac{b''(z)g(z)}{b(z)-\lambda+i\epsilon}dz\\
&+\int_\T \mathcal{G}_{k}^j(y,z;\Lambda) \chi_{j, \delta(\Lambda)}^*(z)\frac{b''(z)g(z)}{b(z)-\lambda+i\epsilon}dz\\
&:=T_{2,R}g(y)+T_{2,S}g(y),
\end{split}
\end{equation}
where $c=1/16$ in this section.
It suffices to prove for $\diamond\in\{R,S\}$
\begin{equation}\label{DeC5.50}
\|T_{2,\diamond}g\|_{X^{\sigma_1, \sigma_2}}\lesssim  (1 + |k| \delta(\Lambda))^{-1/4}\|g\|_{X^{\sigma_1, \sigma_2}}. 
\end{equation}

In the following process, one can simultaneously obtain the pointwise bounds on $T_2$ which will be used in Proposition~\ref{DeC7}. For $\delta(\Lambda)|k| \lesssim 1$, it yields
\begin{equation}\label{DeC10.3}
    \begin{split}
        \left|T_2g(y)\right|
        \lesssim&\varrho_{j}^{-1}(y;\Lambda) \delta^{1/2}(\Lambda)\left[\|g\|_{L^2(S_{3\delta(\Lambda)}^j)}+\delta(\Lambda)\|\partial_y g\|_{L^2(S_{3\delta(\Lambda)}^j)}\right]\\
        \lesssim& \varrho_{j}^{-1}(y;\Lambda)\left(\delta(\Lambda)\right)^{1-\sigma_1-\sigma_2}\|g\|_{X^{\sigma_1,\sigma_2}},\qquad\ \text{for}\ y\in\mathbb{T}.
    \end{split}
\end{equation}
In addition,
\begin{equation}\label{DeC10.3.1}
\begin{aligned}
        &\left|\varrho_{j,k}(y;\Lambda)\partial_y T_2g(y)\right|
        \lesssim\varrho_{j}^{-1}(y;\Lambda) \delta^{1/2}(\Lambda)\left[\|g\|_{L^2(S_{3\delta(\Lambda)}^j)}+\delta\|\partial_y g\|_{L^2(S_{3\delta(\Lambda)}^j)}\right]\\
        &\hspace{31mm}\lesssim\varrho_{j}^{-1}(y;\Lambda)\left(\delta(\Lambda)\right)^{1-\sigma_1-\sigma_2}\|g\|_{X^{\sigma_1,\sigma_2}},\qquad\ \text{for}\ y\in\mathbb{T}\backslash S_{\delta(\Lambda)}^j;\\
        &\left\|\delta(\Lambda)\partial_y T_2g(y)\right\|_{L^2(S_{\delta(\Lambda)}^j)}\lesssim\|g\|_{L^2(S_{3\delta(\Lambda)}^j)}+\delta(\Lambda)\|\partial_y g\|_{L^2(S_{3\delta(\Lambda)}^j)}\lesssim \left(\delta(\Lambda)\right)^{1/2-\sigma_1-\sigma_2}\|g\|_{X^{\sigma_1,\sigma_2}}.
\end{aligned}
\end{equation}
{\bf Step 1.} We first prove \eqref{DeC5.50} with $\diamond=R$. 

{\it Case I: $\delta(\Lambda)|k| \lesssim 1 $}.  In this case for $|z-y_{j*}|\leq2c\delta(\Lambda)$  we have the bound 
\begin{equation}\label{DeC5.6}
\big|\mathcal{G}_{k}^j(y,z;\Lambda)\big|\lesssim \frac{\delta^2(\Lambda)}{\varrho_j(y; \Lambda)},\quad \big|\partial_y\mathcal{G}_{k}^j(y,z;\Lambda)\big|\lesssim\frac{\delta^2(\Lambda)}{\varrho_{j,k}(y; \Lambda) \varrho_j(y; \Lambda)  }. 
\end{equation}
It follows from the bound \eqref{DeC5.6} that 
\begin{equation}\label{DeC5.7}
\|T_{2,R}g\|_{X^{\sigma_1, \sigma_2 }}\lesssim \|g\|_{X^{\sigma_1, \sigma_2}}.
\end{equation}

{\it Case II: $\delta(\Lambda) |k| \gg 1$.} In this case, we have for $|z-y_{j*}|\leq2c\delta(\Lambda)$ that
\begin{equation}\label{DeC5.8}
\big|\mathcal{G}_{k}^j(y,z;\Lambda)\big|+|k|^{-1}\big|\partial_y\mathcal{G}_{k}^j(y,z;\Lambda)\big|\lesssim |k|^{-1}e^{-|k||y-z|}. 
\end{equation}
The desired bound 
\begin{equation}\label{DeC5.9}
\|T_{2,R}g\|_{X^{\sigma_1, \sigma_2}}\lesssim \big(1+|k|\delta(\Lambda)\big)^{-1/4}\|g\|_{X^{\sigma_1, \sigma_2}}
\end{equation}
follows from \eqref{DeC5.8}. 

Furthermore, we have 
\begin{equation}\label{DeC5.10}
    \begin{split}
        \big\|\partial_y^2 T_{2,R}g\big\|_{L^1(S^j_{3\delta(\Lambda)})}\lesssim&\left(\delta(\Lambda)\right)^{-2}(1+|k|\delta(\Lambda))^2\|T_{2,R}g\|_{L^1(S^j_{3\delta(\Lambda)})}+\left\|\chi_{j, c\delta(\Lambda)}(y)   \frac{b''(y)g(y)}{b(y)-\lambda+i\epsilon}\right\|_{L^1(S^j_{3\delta(\Lambda)})}\\
        \lesssim&\left[(1+|k|\delta(\Lambda))^2+ d_{j,k}^{-1}\delta(\Lambda)\right]\left(\delta(\Lambda)\right)^{-2-\sigma_1}d_{j,k}^{1-\sigma_2}\|g\|_{X^{\sigma_1, \sigma_2}}\\
        \lesssim&\frac{(\delta(\Lambda))^{-\sigma_1 - \sigma_2 }}{d_{j,k}}  \|g\|_{X^{\sigma_1, \sigma_2}}.
    \end{split}
\end{equation}
{\bf Step 2.} We now turn to the proof of \eqref{DeC5.50} with $\diamond=S$ and still consider two cases. 

{\it Case I: $\delta(\Lambda)|k| \lesssim 1$}.
 We can rewrite 
\begin{equation}\label{DeC6.2}
\begin{split}
T_{2,S}g(y)&=\int_\T\mathcal{G}_{k}^j(y,z;\Lambda)\chi_{j,\delta(\Lambda)}^\ast(z)\frac{b''(z)g(z)}{b'(z)}\partial_z\log{\frac{b(z)-\lambda+i\epsilon}{\delta^2(\Lambda)}}\,dz\\
&=-\int_\T\partial_z\bigg[\mathcal{G}_{k}^j(y,z;\Lambda)\chi_{j,\delta(\Lambda)}^\ast(z)\frac{b''(z)g(z)}{b'(z)}\bigg]\log{\frac{b(z)-\lambda+i\epsilon}{\delta^2(\Lambda)}}dz.
\end{split}
\end{equation}
As a consequence of \eqref{DeC6.2}, we also have
\begin{equation}\label{DeC6.3}
\begin{split}
\partial_y\Big[T_{2,S}g(y)\Big]&=\partial_y\int_\T\mathcal{G}_{k}^j(y,z;\Lambda)\chi_{j,\delta(\Lambda)}^\ast(z)\frac{b''(z)g(z)}{b'(z)}\partial_z\log{\frac{b(z)-\lambda+i\epsilon}{\delta^2(\Lambda)}}dz\\
&=-\int_\T\bigg[\partial_y(\partial_z+\partial_y)\mathcal{G}^j_{k}(y,z;\lambda,\epsilon)\chi_{j,\delta(\Lambda)}^\ast(z)\frac{b''(z)g(z)}{b'(z)}\bigg]\log{\frac{b(z)-\lambda+i\epsilon}{\delta^2(\Lambda)}}dz\\
&\quad+\int_\T\bigg[\partial^2_y\mathcal{G}_{k}^j(y,z;\Lambda)\chi_{j,\delta(\Lambda)}^\ast(z)\frac{b''(z)g(z)}{b'(z)}\bigg]\log{\frac{b(z)-\lambda+i\epsilon}{\delta^2(\Lambda)}}dz\\
&\quad-\int_\T\partial_y\mathcal{G}_{k}^j(y,z;\Lambda)\partial_z\bigg[\chi_{j,\delta(\Lambda)}^\ast(z)\frac{b''(z)g(z)}{b'(z)}\bigg]\log{\frac{b(z)-\lambda+i\epsilon}{\delta^2(\Lambda)}}dz.
\end{split}
\end{equation}
Note that on the support of $\chi_{j,\delta(\Lambda)}^\ast(z)$, we have 
\begin{equation}\label{DeC6.31}
|b'(z)|\approx \delta(\Lambda), \quad \varrho_j(z; \Lambda)\approx \delta(\Lambda).
\end{equation}
The desired bound \eqref{DeC5.50} for $\diamond=S$ follows from \eqref{DeC6.2}-\eqref{DeC6.3}, Lemmas \ref{mGk50}, \ref{mGk30} and we have, in addition,
\begin{equation}\label{Dec6.32}
\begin{split}
&(\delta(\Lambda))^{-1}\bigg\|\partial_y\bigg\{\partial_yT_{2,S}g(y)+\chi_{j,\delta(\Lambda)}^*(y)\frac{b''(y)g(y)}{b'(y)}\log{\frac{b(y)-\lambda+i\epsilon}{\delta^2(\Lambda)}}\bigg\}\bigg\|_{L^1(S_{3 \delta(\Lambda)}^j)}\\
&\lesssim\delta(\Lambda)^{-\sigma_1 - \sigma_2 -2} \Big[1+|k|\delta(\Lambda)\Big]^{-1/4}\|g\|_{X^{\sigma_1, \sigma_2}}.
\end{split}
\end{equation}

{\it Case II: $\delta(\Lambda) |k| \gg 1$.} This case is analogous to {\it Case I}, using Lemmas \ref{mGk50}, \ref{mGk30}. 

Finally, it remains to consider $|\epsilon|\geq|\lambda-b(y_{j*})|$. In this case $\delta^2(\Lambda)\approx|\epsilon|$, and 
\begin{align*}
|b(z)-\lambda+i\epsilon|\approx |z-y_{j*}|^2+|\epsilon|
\approx\varrho_j^2(z;\Lambda),\qquad |z-y_{j*}|\leq2\delta_0.
\end{align*}
In particular, on the support of $\chi_{j,\delta(\Lambda)}$, the reciprocal denominator is $O(\delta^{-2}(\Lambda))$ and the normalized logarithm is bounded. Each spatial derivative costs at most another factor $\delta^{-1}(\Lambda)$. Combining these bounds with Lemma~\ref{mGk50} gives \eqref{DeC5.4}. By similar estimates as in \eqref{DeC5.10}, we have the following bound
\begin{equation}\label{Dec6.33}
    \big\|\partial_y^2 T_2g\big\|_{L^1(S^j_{3\delta(\Lambda)})}\lesssim\frac{(\delta(\Lambda))^{-\sigma_1 - \sigma_2 }}{d_{j,k}}  \|g\|_{X^{\sigma_1, \sigma_2}},
\end{equation}
which gives \eqref{DeC6.4}. Thus both epsilon regimes are covered.





\end{proof}

\begin{proposition}
    \label{DeC7}
 For $\sigma_1 \in [0,1], \sigma_2 \in [-2,0], \sigma_1 + \sigma_2 \in [-2, 1]$, there exist positive numbers $\epsilon_0, \mu_0$ such that   the following statement holds. For $\epsilon\in[-\epsilon_0,\epsilon_0]\backslash\{0\}$, $\lambda\in \cup_{j}\Sigma_{j, \delta_0}$,  $\kappa k\in\mathbb{Z}\backslash\{0\}$ we have
\begin{equation}\label{DeC8}
\|(I+T_\Theta) g\|_{X^{\sigma_1, \sigma_2}}\ge \mu_0\| g\|_{X^{\sigma_1, \sigma_2}}, \quad {\rm for\,\,all}\,\,g\in H^{1}_k(\T) \cap W^{1, \infty}(\T).
\end{equation}
\end{proposition}
\begin{remark}
    In practice, we only need the proposition for a finite number of $(\sigma_1, \sigma_2)$ so the implicit dependence of $\epsilon_0, \mu_0$ on $(\sigma_1, \sigma_2)$ is not important as we can take the minimum over the relevant cases. 
\end{remark}
\begin{proof}
    We prove \eqref{DeC8} by a contradiction argument.  Assume \eqref{DeC8} does not hold for any $\epsilon_0>0$. Since there are only two degenerate spectral neighborhoods, we may choose the contradiction sequence below in a fixed $\Sigma_{j,\delta_0}$. Then there exist for $\ell\in\mathbb{Z}\cap[1,\infty)$,
    \begin{equation}\label{DeC9}
        \lambda_\ell\to\lambda_\ast\in \Sigma_{j,\delta_0}, \,\,\epsilon_\ell\neq0 \,\,{\rm with}\,\,\epsilon_\ell\to0, \,\,\kappa k_\ell\to \kappa k_\ast\in (\mathbb{Z}\backslash\{0\})\cup\{\pm\infty\},
    \end{equation}
    and functions $g_\ell$ satisfying
    \begin{equation}\label{DeC9.1}
        \|g_\ell\|_{X^{\sigma_1, \sigma_2}(\T)}=1,
    \end{equation} 
    such that 
    \begin{equation}\label{DeC10}
        \big\|(I+T_{\Theta_\ell}(\lambda_\ell+i\epsilon_\ell))g_\ell\big\|_{X^{\sigma_1, \sigma_2}(\T)}\to 0\quad\ \text{as}\ \ell\to\infty.
    \end{equation}
    We can assume that $\lambda_\ast=b(y_{j\ast})$, otherwise the proof follows from the arguments in the non-degenerate case, see Lemma 4.2 in \cite{IIJ24}. For convenience, we use the shorthand $\delta_\ell:=\delta(\Lambda_\ell)$, which tends to zero as $\ell\to\infty$ by \eqref{DeC9}.
    By Lemma \ref{DeC4}, it suffices to consider $|k_\ell|\delta_\ell\lesssim 1$. 
    We consider the following cases.

    {\it Case I: $\limsup_{\ell\to\infty}\sum_{\beta\in\{0,1\}}\|\varrho_{j,k_\ell}^{\beta}(\cdot;\Lambda_\ell)\partial_y^\beta g_\ell\|_{L^\infty(\T\backslash S_{\delta_0})}=0$.} In this case,  we assert that (passing to a subsequence if necessary)
    \begin{equation}\label{DeC13}
        \|\delta_\ell^{\sigma_1+\sigma_2-1/2}g_\ell\|_{L^2(S_{3\delta_\ell})}+\|\delta_\ell^{\sigma_1+\sigma_2+1/2}\partial_y g_\ell\|_{L^2(S_{3\delta_\ell})}\geq\sigma>0.
    \end{equation}
    Otherwise, if
    \begin{equation}\label{DeC13.1}
        \limsup_{\ell\to\infty}\|\delta_\ell^{\sigma_1+\sigma_2-1/2}g_\ell\|_{L^2(S_{3\delta_\ell})}+\|\delta_\ell^{\sigma_1+\sigma_2+1/2}\partial_y g_\ell\|_{L^2(S_{3\delta_\ell})}=0,
    \end{equation}
    together with \eqref{DeC13.4}, \eqref{DeC5.3}, \eqref{DeC10.3} and the Case I assumption, we would obtain
    \begin{align}\label{DeC13.2}
        \|T_1g_\ell\|_{X^{\sigma_1,\sigma_2}},\, \|T_2g_\ell\|_{X^{\sigma_1,\sigma_2}}\to 0\quad\  \text{as}\ \ell\to\infty.
    \end{align}
    Thus, combining \eqref{DeC10}, we have 
    \begin{equation}\label{DeC13.3}
        \|g_\ell\|_{X^{\sigma_1,\sigma_2}}\lesssim\|(I+T_{\Theta_\ell})g_\ell\|_{X^{\sigma_1,\sigma_2}}+\|T_1g_\ell\|_{X^{\sigma_1,\sigma_2}}+\|T_2 g_\ell\|_{X^{\sigma_1,\sigma_2}}\to 0\quad\  \text{as}\ \ell\to\infty,
    \end{equation}
    contradicting \eqref{DeC9.1}.

    Now, under \eqref{DeC13}, recall from \eqref{Intd1} that
    \begin{equation}\label{DeC14}
        \delta_\ell\approx|\epsilon_\ell|^{1/2}+|\lambda_\ell-b(y_{j\ast})|^{1/2}. 
    \end{equation}
    We divide into several subcases.

    {\it Subcase I.1: $|\epsilon_\ell|^{1/2}\approx |\lambda_\ell-b(y_{j\ast})|^{1/2}$ for a subsequence.}
    Define the change of variables for $\ell\ge1, y\in \T$,
    \begin{equation}\label{DeC15}
        y-y_{j\ast}=s_\ell Y, \quad g_\ell(y):=\delta_\ell^{-\sigma_1-\sigma_2}H_\ell(Y). 
    \end{equation}
    In this subcase, take the scale $s_\ell=|\lambda_\ell-b(y_{j\ast})|^{1/2}$.
    It follows from \eqref{DeC6}-\eqref{DeC6.4} that we can extract a nontrivial limit $H\in H^1(\mathbb{R})$ of $H_\ell$ satisfying for $Y\in \mathbb{R}$,
    \begin{equation}\label{DeC16}
        (\beta^2-\partial_Y^2)H(Y)+\frac{b''(y_{j\ast})}{b''(y_{j\ast})Y^2/2+\gamma+i\alpha}H(Y)=0,
    \end{equation}
    where $\beta=\lim_{\ell\to\infty}|k_\ell||\lambda_\ell-b(y_{j\ast})|^{1/2}\in\mathbb{R}, \alpha=\lim_{\ell\to\infty}\epsilon_\ell|\lambda_\ell-b(y_{j\ast})|^{-1}, \gamma={\rm{sgn}}(b(y_{j\ast})-\lambda_\ell)\in\mathbb{R}\backslash\{0\}.$ This is impossible since the shear flow $(b''(y_{j\ast})Y^2/2,0), Y\in\mathbb{R}$ is spectrally stable. 
    
    To obtain this, we first prove the local convergence. More specifically, for any fixed $L$, 
    \begin{equation}
        \begin{split}
            \|H_\ell\|_{L^2(-L,L)}=\delta_\ell^{\sigma_1+\sigma_2}s_\ell^{-1/2}\|g_\ell\|_{L^2(S_{Ls_\ell})},\\
            \|\partial_Y H_\ell\|_{L^2(-L,L)}=\delta_\ell^{\sigma_1+\sigma_2}s_\ell^{1/2}\|\partial_y g_\ell\|_{L^2(S_{Ls_\ell})}.\\
        \end{split}
    \end{equation}
    By \eqref{DeC9.1}, we have 
    \begin{equation}\label{DeC16.0}
        \sup_{\ell}\|H_\ell\|_{H^1(-L,L)}\lesssim_{L}1,\quad\ \text{for every fixed }\ L.
    \end{equation}
    Furthermore, we assert 
    \begin{equation}\label{Dec16.1}
    \left\|\frac{T_{\Theta_\ell}(\lambda_\ell+i\epsilon_\ell)g_\ell(y_{j\ast}+s_\ell Y)}{\delta_\ell^{-\sigma_1-\sigma_2}}\right\|_{W^{2,1}_Y([-L,L])}\lesssim1.
    \end{equation}
    Then by \eqref{DeC10}, 
    \begin{equation}\label{Dec16.1.01}
        \left\|H_\ell+\frac{T_{\Theta_\ell}(\lambda_\ell+i\epsilon_\ell)g_\ell(y_{j\ast}+s_\ell Y)}{\delta_\ell^{-\sigma_1-\sigma_2}}\right\|_{H^1_Y([-L,L])}\to 0\quad\  \text{as}\ \ell\to\infty,
    \end{equation}
    which provides strong convergence $H_\ell\to H$ in $H^1_{\rm{loc}}(\mathbb{R})$.
    For relatively regular operators $T_1$, $T_{2,R}$, their local $H^2$ estimates are straightforward: For any fixed $L\geq 1$, $y\in S_{Ls_\ell}$,
    \begin{align}
        &\partial_y^2 T_1 g_\ell=k^2 T_1 g_\ell+\frac{b''(y)}{b(y)-\lambda_\ell+i\epsilon_\ell}\left(1-\chi_{j,\delta_\ell}\right)T_1g_\ell,\\
        &\partial_y^2 T_{2,R} g_\ell=k^2 T_{2,R} g_\ell+\frac{b''(y)}{b(y)-\lambda_\ell+i\epsilon_\ell}\left(1-\chi_{j,\delta_\ell}\right)T_{2,R} g_\ell-\chi_{j,\delta_\ell}\chi_{j,c\delta_\ell}\frac{b''(y)}{b(y)-\lambda_\ell+i\epsilon_\ell}g_\ell.
    \end{align}
    Since $|\epsilon_\ell|^{1/2}\approx \delta_\ell$ and $|k_\ell|\lesssim \delta_\ell^{-1}$, 
    \begin{equation}\label{Dec16.1.1}
        |b(y)-\lambda_\ell+i\epsilon_\ell|^{-1}\lesssim|\epsilon_\ell|^{-1}\lesssim\delta_\ell^{-2}.
    \end{equation}
    Combining \eqref{DeC5.3} and \eqref{DeC5.50}, we have 
    \begin{align}
        &\|\partial_y^2 T_1 g_\ell\|_{L^2(S_{Ls_\ell})}\lesssim\delta_\ell^{-2}\|T_1 g_\ell\|_{L^2(S_{Ls_\ell})}\lesssim \delta_\ell^{-\sigma_1-\sigma_2-3/2}|k_\ell|^{-1/2}\|g_\ell\|_{X^{\sigma_1,\sigma_2}},\\
        &\|\partial_y^2 T_{2,R} g_\ell\|_{L^2(S_{Ls_\ell})}\lesssim\delta_\ell^{-2}\left(\|T_{2,R} g_\ell\|_{L^2(S_{Ls_\ell})}+\|g_\ell\|_{L^2(S_{Ls_\ell})}\right)\lesssim \delta_\ell^{-\sigma_1-\sigma_2-3/2}\|g_\ell\|_{X^{\sigma_1,\sigma_2}}.
    \end{align}
Therefore, by \eqref{DeC9.1} and the rescaling \eqref{DeC15},
\begin{align}
&\left\|\partial_Y^2\left(\frac{T_1g_\ell(y_{j\ast}+s_\ell Y)}{\delta_\ell^{-\sigma_1-\sigma_2}}\right)\right\|_{L^2_Y([-L,L])}\lesssim s_\ell^{3/2}\delta_\ell^{\sigma_1+\sigma_2}\|\partial_y^2T_1g_\ell\|_{L^2(S_{Ls_\ell})}\lesssim |k_\ell|^{-1/2},\label{Dec16.2}\\
&\left\|\partial_Y^2\left(\frac{T_{2,R}g_\ell(y_{j\ast}+s_\ell Y)}{\delta_\ell^{-\sigma_1-\sigma_2}}\right)\right\|_{L^2_Y([-L,L])}\lesssim s_\ell^{3/2}\delta_\ell^{\sigma_1+\sigma_2}\|\partial_y^2T_{2,R}g_\ell\|_{L^2(S_{Ls_\ell})}\lesssim 1.\label{Dec16.3}
\end{align}

{\color{black}
For this application of \eqref{Dec6.32}, its integration interval may
be enlarged to $S^j_{Ls_\ell}$, with a constant depending on fixed $L$.
Indeed, set $a=\sigma_1+\sigma_2$ and $u_\ell=T_{2,S}g_\ell$.
On $S^j_{Ls_\ell}\setminus S^j_{3\delta_\ell}$ the source and logarithmic
correction vanish. Since $s_\ell\approx\delta_\ell$ and
$|k_\ell|\delta_\ell\lesssim1$, \eqref{mGk1} and \eqref{DeC5.50} give
\[
 |u_\ell''|\lesssim_L\delta_\ell^{-2}|u_\ell|,
 \qquad
 \|u_\ell''\|_{L^1(S^j_{Ls_\ell}\setminus S^j_{3\delta_\ell})}
 \lesssim_L\delta_\ell^{-a-1}\|g_\ell\|_{X^{\sigma_1,\sigma_2}}.
\]
Multiplication by the rescaling factor $\delta_\ell^a s_\ell$ gives
the required $O_L(1)$ bound on the added annulus.
\par}
By \eqref{Dec6.32}, $T_{2,S}$ is controlled in $W^{2,1}$ after extracting the logarithmic singularity, namely
\begin{equation}\label{DeC16.4}
\begin{aligned}
&\left\|\partial_Y^2\left(\frac{T_{2,S}g_\ell(y_{j\ast}+s_\ell Y)}{\delta_\ell^{-\sigma_1-\sigma_2}}\right)\right.\\
&\qquad\quad\left.+s_\ell\partial_Y\left\{\chi_{j,\delta_\ell}^*(y_{j\ast}+s_\ell Y)\frac{b''(y_{j\ast}+s_\ell Y)H_\ell(Y)}{b'(y_{j\ast}+s_\ell Y)}\log\frac{b(y_{j\ast}+s_\ell Y)-\lambda_\ell+i\epsilon_\ell}{\delta_\ell^2}\right\}\right\|_{L^1_Y([-L,L])}\\
&\lesssim \delta_\ell^{\sigma_1+\sigma_2}s_\ell
\left\|\partial_y\left\{\partial_yT_{2,S}g(y)+\chi_{j,\delta_\ell}^*(y)\frac{b''(y)g(y)}{b'(y)}
\log\frac{b(y)-\lambda_\ell+i\epsilon_\ell}{\delta_\ell^2}\right\}\right\|_{L_y^1(S_{Ls_\ell}^j)}\\
&\lesssim1.
\end{aligned}
\end{equation}
Under the assumption $|\epsilon_\ell|^{1/2}\approx\delta_\ell$, the logarithmic singularity satisfies
\begin{equation}
\begin{split}
    &\left\|s_\ell \chi_{j,\delta_\ell}^*(y_{j\ast}+s_\ell Y)\frac{b''(y_{j\ast}+s_\ell Y)H_\ell(Y)}{b'(y_{j\ast}+s_\ell Y)}\log\frac{b(y_{j\ast}+s_\ell Y)-\lambda_\ell+i\epsilon_\ell}{\delta_\ell^2}\right\|_{H^1_Y([-L,L])}\\
    &\lesssim\|H_\ell\|_{H^1_Y([-L,L])}\lesssim1.
\end{split}
\end{equation}
Together with \eqref{Dec16.2} and \eqref{Dec16.3}, we proved \eqref{Dec16.1}. 
Finally, under assumptions for Case I and Subcase I.1, as $\ell\to\infty$,
\begin{align}
    &\left|\delta_\ell^{\sigma_1+\sigma_2}T_1g_\ell(y_{j\ast}+s_\ell Y)\right|\lesssim(|Y|+1)^2\delta_\ell^{2+\sigma_1+\sigma_2}|k_\ell|^{-1}\|g_\ell\|_{L^\infty(\T\backslash S_{\delta_0})}\to 0,\label{DeC16.6}\\
    &\left|\delta_\ell^{\sigma_1+\sigma_2}T_2g_\ell(y_{j\ast}+s_\ell Y)\right|\lesssim(|Y|+1)^{-1}\|g_\ell\|_{X^{\sigma_1,\sigma_2}}\to (|Y|+1)^{-1}.
\end{align}
Combining this with the consequence 
\begin{equation}
    \left|\delta_\ell^{\sigma_1+\sigma_2}\left(I+T_{\Theta_\ell}\right)g_\ell(y_{j\ast}+s_\ell Y)\right|\lesssim\|\left(I+T_{\Theta_\ell}\right)g_\ell\|_{X^{\sigma_1,\sigma_2}}\to0,
\end{equation}
of \eqref{DeC10}, we obtain
\begin{equation}\label{DeC16.7}
    |H(Y)|=\lim_{\ell\to\infty}|H_\ell(Y)|\lesssim(|Y|+1)^{-1}.
\end{equation}
Thus, $H\in H^1(\mathbb{R})$. The contradiction argument based on \eqref{DeC16} then follows.

{\it Subcase I.2: $|\epsilon_\ell|^{1/2}=o\left( |\lambda_\ell-b(y_{j\ast})|^{1/2}\right)$ for a subsequence.} Passing to a subsequence and using the same rescaling \eqref{DeC15} with $s_\ell=|\lambda_\ell-b(y_{j\ast})|^{1/2}$, we can extract a nontrivial limit $H\in H^1(\mathbb{R})$, such that 
\begin{equation}\label{DeC17}
    (\beta^2-\partial_Y^2)H(Y)+\lim_{\epsilon\to0}\frac{b''(y_{j\ast})}{b''(y_{j\ast})Y^2/2+\gamma+i\epsilon}H(Y)=0,
\end{equation}
where $\epsilon=\epsilon_\ell|\lambda_\ell-b(y_{j\ast})|^{-1}\to 0$ as $\ell\to\infty$. This is again impossible since the shear flow $(b''(y_{j\ast})Y^2/2,0), Y\in\mathbb{R}$ is spectrally stable and has no embedded eigenvalues.

More precisely, in this subcase \eqref{Dec16.2} and \eqref{Dec16.3} still hold, but instead of controlling the singularity by $|\epsilon_\ell|$ as in \eqref{Dec16.1.1}, we use the fact that for $y\in{\rm{supp}}(1-\chi_{j,\delta_\ell})\cup{\rm{supp}}(\chi_{j,c\delta_\ell})$,
\begin{equation}
    |b(y)-\lambda_\ell+i\epsilon_\ell|^{-1}\lesssim|b(y_{j\ast})-\lambda_\ell|^{-1}\lesssim\delta_\ell^{-2}.
\end{equation}
For the operator $T_{2,S}$, we no longer prove its $H^2$ boundedness. Instead, \eqref{DeC16.0} gives $H_\ell\to H$ uniformly on every compact interval by the Sobolev embeddings $H^1\Subset C^{0}$. Then 
\begin{equation}
s_\ell \chi_{j,\delta_\ell}^*(y_{j\ast}+s_\ell Y)\frac{b''(y_{j\ast}+s_\ell Y)H_\ell(Y)}{b'(y_{j\ast}+s_\ell Y)}\log\frac{b(y_{j\ast}+s_\ell Y)-\lambda_\ell+i\epsilon_\ell}{\delta_\ell^2}
\end{equation}
converges strongly in $L^2_{\rm loc}$. Meanwhile, \eqref{DeC16.4} gives strong $L^2_{\rm loc}$ convergence for the remainder
\begin{equation}
\partial_Y\left(\frac{T_{2,S}g_\ell(y_{j\ast}+s_\ell Y)}{\delta_\ell^{-\sigma_1-\sigma_2}}\right)+s_\ell \chi_{j,\delta_\ell}^*(y_{j\ast}+s_\ell Y)\frac{b''(y_{j\ast}+s_\ell Y)H_\ell(Y)}{b'(y_{j\ast}+s_\ell Y)}\log\frac{b(y_{j\ast}+s_\ell Y)-\lambda_\ell+i\epsilon_\ell}{\delta_\ell^2}.
\end{equation}
Thus, the rescaled $\partial_Y T_{2,S}$ strongly converges in $L^2_{\rm{loc}}$, which means $T_{2,S}$ strongly converges in $H^1_{\rm{loc}}$. Together with \eqref{Dec16.1.01}, this yields \eqref{DeC17}. $H\in H^1(\mathbb{R})$ follows from the same arguments as in Subcase I.1.

{\it Subcase I.3: $|\lambda_\ell-b(y_{j*})|^{1/2}=o(|\epsilon_\ell|^{1/2})$ for a subsequence. } 
In this case, we follow a slightly different rescaling \eqref{DeC15} with $s_\ell=|\epsilon_\ell|^{1/2}$.
Then similar arguments as in Subcase I.1 yield \eqref{Dec16.2}-\eqref{DeC16.7}.
Thus, we can extract a nontrivial limit $H\in H^1(\mathbb{R})$ of $H_\ell$ satisfying for $Y\in \mathbb{R}$,
    \begin{equation}\label{DeC18.1}
        (\beta^2-\partial_Y^2)H(Y)+\lim_{\gamma\to 0}\frac{b''(y_{j\ast})}{b''(y_{j\ast})Y^2/2+\gamma+i\alpha}H(Y)=0,
    \end{equation}
    where $\beta=\lim_{\ell\to\infty}|k_\ell||\epsilon_\ell|^{1/2}\in\mathbb{R}, \alpha={\rm{sgn}}\epsilon_\ell \in\mathbb{R}\backslash\{0\}, \gamma=\lim_{\ell\to\infty}(b(y_{j\ast})-\lambda_\ell)\cdot|\epsilon_\ell|^{-1}.$ This is impossible since the shear flow $(b''(y_{j\ast})Y^2/2,0), Y\in\mathbb{R}$ is spectrally stable.

    \medskip
    {\it Case II: $\limsup_{\ell\to\infty}\sum_{\beta\in\{0,1\}}\|\varrho_{j,k_\ell}^{\beta}(\cdot;\Lambda_\ell)\partial_y^\beta g_\ell\|_{L^\infty(\T\backslash S_{\delta_0})}>0$.} 
    
    {\it Subcase II.1:} For $(\sigma_1,\sigma_2)\neq (1,0)$, we can directly assume that $\kappa k_*\in\mathbb{Z}\backslash\{0\}$. 
    Otherwise, $|k_\ell|\to\infty$ would bring $\varrho_{j,k_\ell}(y;\Lambda_\ell)=|k_\ell|^{-1}$ for any fixed $y$. Then by \eqref{DeC1}, \eqref{DeC5.3} and $\sigma_2\leq 0$,
    \begin{equation}\label{DeC10.1}
        \sum_{\beta\in\{0,1\}}\|\varrho_{j,k_\ell}^{\beta}(\cdot;\Lambda_\ell)\partial_y^\beta T_1 g_\ell\|_{L^\infty(\T\backslash S_{\delta_0})}\leq|k_\ell|^{\sigma_2}\|T_1g_\ell\|_{X^{\sigma_1,\sigma_2}}\lesssim|k_\ell|^{\sigma_2-1/2}\|g_\ell\|_{X^{\sigma_1,\sigma_2}}\to 0,
    \end{equation}
    as $\ell\to\infty$.
    Meanwhile, we claim that 
    \begin{equation}\label{DeC10.2}
        \sum_{\beta\in\{0,1\}}\|\varrho_{j,k_\ell}^{\beta}(\cdot;\Lambda_\ell)\partial_y^\beta T_2g_\ell\|_{L^\infty(\T\backslash S_{\delta_0})}\to 0\quad\  \text{as}\ \ell\to\infty.
    \end{equation}
    Indeed, it follows from \eqref{DeC5.4} when $\delta_\ell|k_\ell|\gg1$, and from \eqref{DeC10.3}-\eqref{DeC10.3.1} with the assumption $\sigma_1+\sigma_2<1$ in this subcase when $\delta_\ell|k_\ell|\lesssim1$. 
    Furthermore, by \eqref{DeC10} and $\sigma_2\leq 0$,
    \begin{equation}\label{DeC10.4}
\begin{aligned}
&\sum_{\beta\in\{0,1\}}\|\varrho_{j,k_\ell}^{\beta}(\cdot;\Lambda_\ell)\partial_y^\beta(I+T_{\Theta_\ell})g_\ell\|_{L^\infty(\T\backslash S_{\delta_0})}\\
&\quad\lesssim\delta_0^{-\sigma_1}|k_\ell|^{\sigma_2}\big\|(I+T_{\Theta_\ell})g_\ell\big\|_{X^{\sigma_1,\sigma_2}(\T)}
\to0\quad\text{as }\ell\to\infty.
\end{aligned}
    \end{equation}
    Combining \eqref{DeC10.1}, \eqref{DeC10.2} and \eqref{DeC10.4}, we have as $\ell\to\infty$,
    \begin{equation}
        \begin{split}
            &\sum_{\beta\in\{0,1\}}\|\varrho_{j,k_\ell}^{\beta}(\cdot;\Lambda_\ell)\partial_y^\beta g_\ell\|_{L^\infty(\T\backslash S_{\delta_0})}\\
            &\lesssim\sum_{\beta\in\{0,1\}}\left(\|\varrho_{j,k_\ell}^{\beta}(\cdot;\Lambda_\ell)\partial_y^\beta(I+T_{\Theta_\ell})g_\ell\|_{L^\infty(\T\backslash S_{\delta_0})}+\|\varrho_{j,k_\ell}^{\beta}(\cdot;\Lambda_\ell)\partial_y^\beta T_1g_\ell\|_{L^\infty(\T\backslash S_{\delta_0})}\right.\\
            &\hspace{20mm}+\left.\|\varrho_{j,k_\ell}^{\beta}(\cdot;\Lambda_\ell)\partial_y^\beta T_2g_\ell\|_{L^\infty(\T\backslash S_{\delta_0})}\right)\\
            &\to 0,
        \end{split}
    \end{equation}
    contradicting our non-concentration assumption. By \eqref{DeC9.1}, \eqref{DeC10}, we can assume (passing to a subsequence if necessary) that 
    \begin{equation}\label{DeC11}
        g_\ell\to g, \,\,{\rm in}\,\, H^1_{{\rm loc}}(\T\backslash \{y_{j\ast}\})\,\,{\rm as}\,\,\ell\to\infty.
    \end{equation}
    Then one can obtain that $g$ satisfies
    \begin{equation}\label{DeC11.1}
        |g(y)|\lesssim |y-y_{j\ast}|^{2},
    \end{equation}
    and for $y\in \T$, in the $L^2$ sense,
    \begin{equation}\label{DeC12}
        (k_\ast^2-\partial_y^2)g(y)+\frac{b''(y)}{b(y)-b(y_{j\ast})}g(y)=0.
    \end{equation}
    In fact, for any fixed $\mathfrak{d}>0$, we have $2\delta_{\ell}<\mathfrak{d}$ for sufficiently large $\ell$. In the compact domain $\T\backslash S_\mathfrak{d}$, $\chi_{j,\delta}\equiv 0$, so $T_{\Theta_\ell}g(y)$ satisfies the simplified equation
    \begin{equation}\label{DeC12.1}
        \begin{split}
            \partial_y^2 T_{\Theta_\ell}g_{\ell}=\left\{k_\ell^2+\frac{b''(y)\chi_{j,\delta_0}}{b(y)-\lambda_\ell+i\epsilon_\ell}\right\}T_{\Theta_\ell}g_{\ell}-(1-\chi_{j,\delta_0})\frac{b''(y)g_\ell(y)}{b(y)-\lambda_\ell+i\epsilon_\ell}, 
        \end{split}
    \end{equation}
    with the following bounds independent of $\ell$:
    \begin{align}
        &\|T_{\Theta_\ell}g_\ell\|_{L^\infty(\T\backslash S_\mathfrak{d})}\lesssim_{\mathfrak{d}}\|T_{\Theta_\ell}g_\ell\|_{X^{\sigma_1,\sigma_2}}\lesssim_{\mathfrak{d}}1,\\
        &\|g_\ell\|_{L^\infty(\T\backslash S_\mathfrak{d})}\lesssim_{\mathfrak{d}}\|g_\ell\|_{X^{\sigma_1,\sigma_2}}\lesssim_{\mathfrak{d}}1,\\
        &\left|b(y)-\lambda_\ell+i\epsilon_\ell \right|^{-1}\lesssim |\mathfrak{d}|^{-2}.
    \end{align}
    Therefore, 
    \begin{equation}\label{DeC12.2}
        \limsup_{\ell\to \infty}\|T_{\Theta_\ell}(\lambda_\ell+i\epsilon_{\ell})g_\ell\|_{W^{2,\infty}(\T\backslash S_\mathfrak{d})}\lesssim_{\mathfrak{d}}1.
    \end{equation}
    By Sobolev embeddings, $T_{\Theta_\ell}(\lambda_\ell+i\epsilon_{\ell})g_\ell$ converges strongly in $H^1_{{\rm loc}}(\T\backslash \{y_{j\ast}\})$. Together with \eqref{DeC10}, we obtain \eqref{DeC11}.
    Moreover, by \eqref{DeC10} and \eqref{DeC10.3}, 
    \begin{equation}
        \lim_{\ell\to\infty}T_2g_\ell=0,\quad\ g+ \lim_{\ell\to\infty}T_1g_\ell=0\qquad\ \text{in}\ H^1_{\rm{loc}}(\T\backslash \{y_{j\ast}\}).
    \end{equation}
    By \eqref{DeC5.2}
    \begin{equation}
        \left|T_1g_\ell(y)\right|\lesssim\varrho_{j}^2(y;\Lambda_\ell)|k_\ell|^{-1}\|g_\ell\|_{L^\infty(\T\backslash S_{\delta_0})}\lesssim\varrho_{j}^2(y;\Lambda_\ell)\delta_0^{-\sigma_1}|k_\ell|^{\sigma_2-1}\|g_\ell\|_{X^{\sigma_1,\sigma_2}}\lesssim \varrho_{j}^2(y;\Lambda_\ell).
    \end{equation}
    Thus we have 
    \begin{equation}
        |g(y)|\lesssim|y-y_{j*}|^2,
    \end{equation}
    which gives \eqref{DeC12} by sending $\ell\to\infty$ in \eqref{DeC12.1} and no delta distribution is introduced at $y_{j*}$. Moreover, by \eqref{DeC12.2} and the Case II assumption, we have the limit function $g$ nontrivial. Thus, \eqref{DeC12} implies that $b(y_{j\ast})$ is an embedded eigenvalue for $L_k$, a contradiction to Assumption \ref{MaAs}. 

    {\it Subcase II.2:} If $(\sigma_1,\sigma_2) = (1,0)$, one needs to be careful about asserting $\kappa k_*\in\mathbb{Z}\backslash\{0\}$ since \eqref{DeC10.3}-\eqref{DeC10.3.1} no longer provide extra smallness as $\ell\to\infty$. If additionally,
    \begin{equation}
        \limsup_{\ell\to\infty}\left[\|\delta_\ell^{1/2}g_\ell\|_{L^2(S_{3\delta_\ell})}+\|\delta_\ell^{3/2}\partial_yg_\ell\|_{L^2(S_{3\delta_\ell})}\right]=0,
    \end{equation}
    then \eqref{DeC10.3}-\eqref{DeC10.3.1} still imply \eqref{DeC10.2}, which allows us to assert $\kappa k_*\in\mathbb{Z}\backslash\{0\}$ and follow the same arguments as in Subcase II.1 to obtain the contradiction.
    Otherwise, we have (passing to a subsequence if necessary)
    \begin{equation}
        \|\delta_\ell^{1/2}g_\ell\|_{L^2(S_{3\delta_\ell})}+\|\delta_\ell^{3/2}\partial_yg_\ell\|_{L^2(S_{3\delta_\ell})}\geq\sigma>0.
    \end{equation}
    One can still achieve the contradiction by following the blow-up arguments in Case I with $(\sigma_1,\sigma_2)=(1,0)$. A slight difference is that here Case I assumption no longer holds. Therefore, in the process of obtaining the global convergence of the limit function $H\in H^1(\mathbb{R})$, \eqref{DeC16.6} is obtained by $\|g_\ell\|_{L^\infty(\T\backslash S_{\delta_0})}\lesssim\|g_\ell\|_{X^{1,0}}$ and $\delta_\ell^{2+\sigma_1+\sigma_2}=\delta_\ell^3\to0$ as $\ell\to\infty$.

The proposition is now proved. 
\end{proof}

\section{Bounds on \texorpdfstring{$\psi_{k,\epsilon}^{\iota}$}{psi} and higher order spectral regularity:  Non-Degenerate case}
\label{sec:nd:bdds}
In this section, for convenience, we suppress the indices $k,\epsilon,\iota$ for the spectral density function and its derivatives.
\begin{lemma}\label{ndint}
    Assume that $\kappa k\in\mathbb{Z}\backslash\{0\}, \lambda\in \Sigma\backslash \cup_{j}\Sigma_{j, \delta_0}$, $\iota\in\{\pm\}, \epsilon \in (0, \epsilon_0)$, $\alpha\in\{0,1,2,\cdots\}$. The implicit constants may depend on $\alpha$.
    \begin{enumerate}
    \item For any function $f(y)\in L^2$, $\beta\in\{0,1\}$, we have the following estimates
    \begin{equation}\label{ndlogl2}
        |k|^{1-\beta}\left\|\int_{\T}\partial_y^{\beta} G_k(y,z)f(z)(\log(b(z)-\lambda+i\iota\epsilon))^\alpha dz\right\|_{L^2}\lesssim\frac{1}{|k|^{1/2}}\|f\|_{L^2}.
    \end{equation}
    \item For any function $f(y)\in H^1_k$, we have the following estimates
    \begin{equation}\label{ndfracl2}
        \left\|\int_{\T}\partial_yG_k(y,z)f(z)\frac{(\log(b(z)-\lambda+i\iota\epsilon))^\alpha}{b(z)-\lambda+i\iota\epsilon}dz\right\|_{L^2}\lesssim\frac{1}{|k|^{1/2}}\|f\|_{H^1_k}.
    \end{equation}
    \end{enumerate}
\end{lemma}
\begin{proof}
    \eqref{ndlogl2} follows from \eqref{Helm1.0}, H\"older's and Young's inequalities. In order to obtain \eqref{ndfracl2}, we use the integration by parts in the singular region
    \begin{equation}
        \begin{split}
            &\int_{\T}\partial_yG_k(y,z)(1-\chi_{\ast,\delta_0/10}(z))f(z)\frac{(\log(b(z)-\lambda+i\iota\epsilon))^\alpha}{b(z)-\lambda+i\iota\epsilon}dz\\
            =&\int_{\T}\partial_yG_k(y,z)(1-\chi_{\ast,\delta_0/10}(z))f(z)\frac{1}{b'(z)}\partial_z\left(\frac{1}{\alpha+1}(\log(b(z)-\lambda+i\iota\epsilon))^{\alpha+1}\right)dz\\
            =&-\int_{\T}\partial_z\left(\frac{\partial_yG_k(y,z)f(z)(1-\chi_{\ast,\delta_0/10}(z))}{(\alpha+1)b'(z)}\right)(\log(b(z)-\lambda+i\iota\epsilon))^{\alpha+1}dz\\
            =&-(1-\chi_{\ast,\delta_0/10}(y))\frac{f(y)}{(\alpha+1)b'(y)}(\log(b(y)-\lambda+i\iota\epsilon))^{\alpha+1}\\
            &-\int_{\T}\frac{F_k(y,z)f(z)(1-\chi_{\ast,\delta_0/10}(z))}{(\alpha+1)b'(z)}(\log(b(z)-\lambda+i\iota\epsilon))^{\alpha+1}dz\\
            &-\int_{\T}\partial_yG_k(y,z)\partial_z\left(\frac{(1-\chi_{\ast,\delta_0/10}(z))f(z)}{b'(z)}\right)\frac{1}{\alpha+1}(\log(b(z)-\lambda+i\iota\epsilon))^{\alpha+1}dz.
        \end{split}
    \end{equation}
    Together with \eqref{GkFk}
    and the same routine as in obtaining \eqref{ndlogl2}, we have the following estimates for $\beta\in\{0,1\}$, $\alpha\in\{0,1,2,\cdots\}$:
    \begin{equation}\label{ndFlogl2}
        |k|^{-1-\beta}\left\|\int_{\T}\partial_y^\beta F_k(y,z)f(z)(\log(b(z)-\lambda+i\iota\epsilon))^\alpha dz\right\|_{L^2}\lesssim\frac{1}{|k|^{1/2}}\|f\|_{L^2}.
    \end{equation}
    \eqref{ndfracl2} then follows from \eqref{ndlogl2}, \eqref{ndFlogl2} and H\"older's inequality.
\end{proof}

\begin{lemma}\label{ndraygeneral}
    Assume that $\kappa k\in\mathbb{Z}\backslash\{0\},\  y\in\mathbb{T},\  \lambda\in \Sigma\backslash \cup_{j}\Sigma_{j, \delta_0}$, $\iota\in\{\pm\}, \epsilon \in (0, \epsilon_0)$, $\alpha\in\{1,2,3\}$, and let $f(y,\lambda)$ satisfy
    \begin{equation}\label{ndrayf}
        \begin{split}
            &\left(-\partial_y^2+k^2+\frac{b''(y)}{b(y)-\lambda+i\iota\epsilon}\right)f(y,\lambda) \\
            &= c_1(y)\frac{\log(b(y)-\lambda+i\iota\epsilon)}{b(y)-\lambda+i\iota\epsilon}+c_2(y)\frac{1}{b(y)-\lambda+i\iota\epsilon}
            +c_{3}(y)(\log(b(y)-\lambda+i\iota\epsilon))^\alpha+c_4(y),
        \end{split}
    \end{equation}
    with coefficient functions $c_1(y), c_2(y)\in H^1_k$, and $c_{3}(y), c_{4}(y)\in L^2$. Then we have the following bounds for $f(y,\lambda)$:
    \begin{equation}\label{ndraygeneral1}
        \left\|f(\cdot,\lambda)\right\|_{H^1_k}\lesssim \frac{1}{|k|^{1/2}}\left(\|c_1\|_{H^1_k}+\|c_2\|_{H^1_k}+\|c_{3}\|_{L^2}+\|c_{4}\|_{L^2}\right),
    \end{equation}
    and for any $1\leq p<2$,
    \begin{equation}\label{ndraygeneral2}
        \begin{split}
            \left\|\partial_y f(y,\lambda)-(1-\chi_{\ast,\delta_0\slash10}(y))\left(C_{1}(y)\left(\log(b(y)-\lambda+i\iota\epsilon)\right)^{2}+C_{2}(y)\log(b(y)-\lambda+i\iota\epsilon)\right)\right\|_{W^{1,p}}\\
        \lesssim_p \|c_1\|_{H^1_k}+\|c_2\|_{H^1_k}+\|c_{3}\|_{L^2}+\|c_{4}\|_{L^2},
        \end{split}
    \end{equation}
    with $\|(1-\chi_{\ast,\delta_0\slash10})C_1\|_{H^1_k}\lesssim\|c_1\|_{H^1_k}$, $\|(1-\chi_{\ast,\delta_0\slash10})C_2\|_{H^1_k}\lesssim\|c_2\|_{H^1_k}+|k|^{-1/2}\left(\|c_1\|_{H^1_k}+\|c_{3}\|_{L^2}+\right.$ $\left.\|c_{4}\|_{L^2}\right)$.
\end{lemma}
\begin{remark}
    Note that logarithmic singularities of the form $c(y)\cdot\log((b(y)-\lambda+i\iota\epsilon))^{\alpha}$ are not uniformly in $L^2$ as $\epsilon\to0$ if $c(y)$ is merely in $L^2$. This is the reason why the source term $c_{3}(y)\log(b(y)-\lambda+i\iota\epsilon)$ is singled out in \eqref{ndrayf}. As shown in the proof, the source term\\ $c_{3}(y)\log(b(y)-\lambda+i\iota\epsilon)$ plays a role that is almost the same as the source term $c_4(y)$. However, by taking the logarithmic source term out of the general term, we get to track the lowest regularity required for the coefficient functions. This is also the reason why $p=2$ is not achievable in \eqref{ndraygeneral2}. When taking higher derivatives of $f$, the $\delta$ term from the second $y$-derivative of the Green's function produces singularities such as $c_1'(y)\log^2(b(y)-\lambda+i\iota\epsilon)$, as shown in \eqref{ndraygeneral4}. Under the stated regularity assumptions on the coefficients, these need not belong to $L^2$. Lemma \ref{ndraygeneral} is particularly useful in estimating the highest derivatives of the spectral density function.
\end{remark}

\begin{proof}
    For convenience, we focus on the estimates on the highest derivatives as they are most singular. The lower order estimates follow similar calculations.
    We denote the right-hand side of \eqref{ndrayf} as $\mathcal{R}(y,\lambda)$. The integral form of \eqref{ndrayf} reads
    \begin{equation}\label{ndraygeneral3}
        f(y,\lambda)+ T_{\Theta}f(y,\lambda) = \int_{\T}G_k(y,z)\mathcal{R}(z,\lambda)dz.
    \end{equation}
    By the limiting absorption principle in Lemma \ref{T5}, in order to obtain \eqref{ndraygeneral1}, it suffices to prove that 
    \begin{equation}
        \begin{split}
            \left\|\int_{\T}G_k(y,z)\mathcal{R}(z,\lambda)dz\right\|_{H^1_k}
            \lesssim\frac{1}{|k|^{1/2}}\left(\|c_1\|_{H^1_k}+\|c_2\|_{H^1_k}+\|c_{3}\|_{L^2}+\|c_{4}\|_{L^2}\right),
        \end{split}
    \end{equation}
    which follows directly from \eqref{ndlogl2} and \eqref{ndfracl2}.

    Moreover, \eqref{ndraygeneral3} gives
    \begin{equation}
        \begin{split}
            \partial_yf(y,\lambda) =& -\partial_y(T_{\Theta}f)(y,\lambda)+\partial_y\int_{\T}G_k(y,z)\mathcal{R}(z,\lambda)dz\\
            =&-\int_{\T}\partial_yG_k(y,z)\frac{b''(z)f(z,\lambda)}{b(z)-\lambda+i\iota\epsilon}dz+\int_{\T}\partial_yG_k(y,z)\mathcal{R}(z,\lambda)dz.\\
        \end{split}
    \end{equation}
    Notice that since $\lambda\in \Sigma\backslash \cup_{j}\Sigma_{j, \delta_0}$, the denominator is uniformly bounded away from zero for $z\in \rm{supp}(\chi_{\ast,\delta_0\slash10})$. Therefore, 
    \begin{equation}\label{ndraygeneral3.1}
        \left\|\int_{\T}\partial_yG_k(y,z)\chi_{\ast,\delta_0\slash10}(z)\left(-\frac{b''(z)f(z,\lambda)}{b(z)-\lambda+i\iota\epsilon}+\mathcal{R}(z,\lambda)\right)dz\right\|_{H^1_k}\lesssim\|f\|_{L^2}+\sum_{i=1}^{4}\|c_i\|_{L^2},
    \end{equation}
    while in the remaining region, applying integration by parts gives
    \begin{equation}\label{ndraygeneral3.2}
        \begin{split}
            &\int_{\T}\partial_yG_k(y,z)(1-\chi_{\ast,\delta_0\slash10}(z))\left(-\frac{b''(z)f(z,\lambda)}{b(z)-\lambda+i\iota\epsilon}+\mathcal{R}(z,\lambda)\right)dz\\
            =&(1-\chi_{\ast,\delta_0\slash10}(y))\left(C_{1}(y)\left(\log(b(y)-\lambda+i\iota\epsilon)\right)^{2}+C_{2}(y)\log(b(y)-\lambda+i\iota\epsilon)\right)+r(y,\lambda),
        \end{split}
    \end{equation}
    where 
    \begin{align}
        C_1(y)=&-\frac{c_1(y)}{2b'(y)},\quad\ C_2(y) = \frac{b''(y)f(y,\lambda)-c_2(y)}{b'(y)},\\
        r(y,\lambda)=&-\int_{\T}F_k(y,z)(1-\chi_{\ast,\delta_0\slash10}(z))\frac{c_1(z)}{2b'(z)}\left(\log(b(z)-\lambda+i\iota\epsilon)\right)^2dz\notag\\
        &+\int_{\T}F_k(y,z)(1-\chi_{\ast,\delta_0\slash10}(z))\frac{b''(z)f(z,\lambda)-c_2(z)}{b'(z)}\log(b(z)-\lambda+i\iota\epsilon)dz\notag\\
        &+\int_{\T}\partial_yG_k(y,z)\partial_z\left((1-\chi_{\ast,\delta_0\slash10}(z))\frac{b''(z)f(z,\lambda)-c_2(z)}{b'(z)}\right)\log(b(z)-\lambda+i\iota\epsilon)dz\notag\\
        &-\int_{\T}\partial_yG_k(y,z)\partial_z\left(\frac{(1-\chi_{\ast,\delta_0\slash10}(z))c_1(z)}{2b'(z)}\right)\left(\log(b(z)-\lambda+i\iota\epsilon)\right)^2dz\notag\\
        &+\int_{\T}\partial_yG_k(y,z)(1-\chi_{\ast,\delta_0\slash10}(z))c_3(z)(\log(b(z)-\lambda+i\iota\epsilon))^\alpha dz\notag\\
        &+\int_{\T}\partial_yG_k(y,z)(1-\chi_{\ast,\delta_0\slash10}(z))c_4(z)dz.\label{ndraygeneral3.3}
    \end{align}

    It is clear $\|(1-\chi_{\ast,\delta_0\slash10})C_1\|_{H^1_k}\lesssim\|c_1\|_{H^1_k}$, $\|(1-\chi_{\ast,\delta_0\slash10})C_2\|_{H^1_k}\lesssim\|c_2\|_{H^1_k}+|k|^{-1/2}\left(\|c_1\|_{H^1_k}+\right.$ $\left.\|c_{3}\|_{L^2}+\|c_{4}\|_{L^2}\right)$. The remainder term $r(y,\lambda)$ has higher regularity:
    \begin{align}\label{ndraygeneral4}
            \|\partial_yr(y,\lambda)\|_{L^p}&\lesssim\left\|\int_{\T}\partial_yF_k(y,z)(1-\chi_{\ast,\delta_0\slash10})(z)\frac{b''(z)f(z,\lambda)-c_2(z)}{b'(z)}\log(b(z)-\lambda+i\iota\epsilon)dz\right\|_{L^p}\notag\\
            &\quad+\left\|\int_{\T}-\partial_yF_k(y,z)(1-\chi_{\ast,\delta_0\slash10})(z)\frac{c_1(z)}{2b'(z)}\left(\log(b(z)-\lambda+i\iota\epsilon)\right)^2dz\right\|_{L^p}\notag\\
            &\quad+\left\|\partial_y\left((1-\chi_{\ast,\delta_0\slash10})(y)\frac{b''(y)f(y,\lambda)-c_2(y)}{b'(y)}\right)\log(b(y)-\lambda+i\iota\epsilon)\right\|_{L^p}\\
            &\quad+\left\|\partial_y\left(\frac{(1-\chi_{\ast,\delta_0\slash10})(y)c_1(y)}{2b'(y)}\right)\left(\log(b(y)-\lambda+i\iota\epsilon)\right)^2\right\|_{L^p}\notag\\
            &\quad+\left\|(1-\chi_{\ast,\delta_0\slash10})(y)c_3(y)(\log(b(y)-\lambda+i\iota\epsilon))^\alpha \right\|_{L^p}+\left\|(1-\chi_{\ast,\delta_0\slash10})(y)c_4(y)\right\|_{L^p}\notag\\
            &\lesssim_p\|c_1\|_{H^1_k}+\|c_2\|_{H^1_k}+\|c_3\|_{L^2}+\|c_4\|_{L^2}.
    \end{align}
    The conclusion follows from \eqref{ndraygeneral3.1}, \eqref{ndraygeneral3.2}, \eqref{ndraygeneral3.3} and \eqref{ndraygeneral4}.
\end{proof}
\begin{lemma}\label{ndrayg}
    Suppose that $f(y,\lambda)\in H^1_k$ is the solution of \eqref{ndrayf} as in Lemma \ref{ndraygeneral}. For any $g(y,\lambda)$, $y\in\T$ satisfying 
    \begin{equation}
        \left(-\partial_y^2+k^2+\frac{b''(y)}{b(y)-\lambda+i\iota\epsilon}\right)g(y,\lambda)=\mathfrak{c}(y)\frac{\partial_y f(y,\lambda)}{b(y)-\lambda+i\iota\epsilon},\quad\ \|\mathfrak{c}(y)\|_{C^1}\lesssim1,
    \end{equation}
    we have 
    \begin{equation}\label{ndrayg1}
        \|g(\cdot,\lambda)\|_{H^1_k}\lesssim\|c_1\|_{H^1_k}+\|c_2\|_{H^1_k}+\|c_3\|_{L^2}+\|c_4\|_{L^2}.
    \end{equation}
\end{lemma}
\begin{remark}
    We emphasize that $f$ in Lemma~\ref{ndrayg} is not a generic $H^1_k$ function. It must satisfy the Rayleigh equation \eqref{ndrayf}, thus having all properties obtained from Lemma \ref{ndraygeneral}. This is essential for us to obtain the regularity of $g$ even when the source term is more singular than that in \eqref{ndrayf}.
\end{remark}
\begin{proof} 
    By the limiting absorption principle, it suffices to prove 
    \begin{equation}\label{ndrayg2}
        \left\|\int_{\T}G_k(y,z)\frac{\mathfrak{c}(z)\partial_zf(z,\lambda)}{b(z)-\lambda+i\iota\epsilon}dz\right\|_{H^1_k}\lesssim\|c_1\|_{H^1_k}+\|c_2\|_{H^1_k}+\|c_3\|_{L^2}+\|c_4\|_{L^2}.
    \end{equation}
    In fact, in the singular region
    \begin{equation}\label{ndintpsi_y1}
    \begin{split}
        &\int_{\T}\partial_yG_k(y,z)\frac{(1-\chi_{\ast,\delta_0\slash10}(z))\mathfrak{c}(z)\partial_z f(z,\lambda)}{b(z)-\lambda+i\iota\epsilon}dz\\
        =&-\int_{\T}\partial_z\left(\frac{\partial_yG_k(y,z)(1-\chi_{\ast,\delta_0\slash10}(z))\mathfrak{c}(z)\partial_zf(z,\lambda)}{b'(z)}\right)\log(b(z)-\lambda+i\iota\epsilon)dz\\
        =&-(1-\chi_{\ast,\delta_0\slash10}(y))\frac{\mathfrak{c}(y)\partial_yf(y)}{b'(y)}\log(b(y)-\lambda+i\iota\epsilon)\\
        &-\int_{\T}\frac{F_k(y,z)(1-\chi_{\ast,\delta_0\slash10}(z))\mathfrak{c}(z)\partial_zf(z,\lambda)}{b'(z)}\log(b(z)-\lambda+i\iota\epsilon)dz\\
        &-\int_{\T}\partial_yG_k(y,z)\partial_zf(z,\lambda)\partial_z\left((1-\chi_{\ast,\delta_0\slash10}(z))\frac{\mathfrak{c}(z)}{b'(z)}\right)\log(b(z)-\lambda+i\iota\epsilon)dz\\
        &-\int_{\T}(1-\chi_{\ast,\delta_0\slash10}(z))\frac{\partial_yG_k(y,z)\mathfrak{c}(z)\partial_z^2f}{b'(z)}\log(b(z)-\lambda+i\iota\epsilon)dz.
    \end{split}
\end{equation}
{\color{black}
In the following equation, $r$ includes the regular-region integral in
\eqref{ndraygeneral3.1} in addition to the remainder in
\eqref{ndraygeneral3.3}.
\par}

By Lemma \ref{ndraygeneral}, the first term, which is the most singular among all, has the $L^2$ bound
\begin{equation}
    \begin{split}
        &\left\|(1-\chi_{\ast,\delta_0\slash10}(y))\frac{\mathfrak{c}(y)\partial_yf}{b'(y)}\log(b(y)-\lambda+i\iota\epsilon)\right\|_{L^2}\\
        \lesssim&\|(1-\chi_{\ast,\delta_0\slash10}(y))\bigl[C_{1}(y)\left(\log(b(y)-\lambda+i\iota\epsilon)\right)^{3}+C_{2}(y)(\log(b(y)-\lambda+i\iota\epsilon))^2\bigr]\|_{L^2}\\
        &+\|r(y,\lambda)\log(b(y)-\lambda+i\iota\epsilon)\|_{L^2}\\
\lesssim&\|(1-\chi_{\ast,\delta_0\slash10})C_1\|_{H^1_k}+\|(1-\chi_{\ast,\delta_0\slash10})C_2\|_{H^1_k}+\|r\|_{W^{1,p}}\\
\lesssim&\|c_1\|_{H^1_k}+\|c_2\|_{H^1_k}+\|c_3\|_{L^2}+\|c_4\|_{L^2}.
    \end{split}
\end{equation}
The second and third terms follow from \eqref{ndlogl2} and \eqref{ndraygeneral2}. The only term requiring extra care is the last one. By \eqref{ndrayf}, we have $\partial_z^2f=k^2f+\frac{b''(z)f}{b(z)-\lambda+i\iota\epsilon}-\mathcal{R}$. Together with \eqref{ndlogl2}, \eqref{ndfracl2} and \eqref{ndraygeneral1}, we obtain
\begin{align}\label{ndintpsi_y2}
        &\left\|\int_{\T}(1-\chi_{\ast,\delta_0\slash10}(z))\frac{\partial_yG_k(y,z)\mathfrak{c}(z)\partial_z^2f}{b'(z)}\log(b(z)-\lambda+i\iota\epsilon)dz\right\|_{L^2}\notag\\
        \lesssim&\left\|\int_{\T}\partial_yG_k(y,z)(1-\chi_{\ast,\delta_0\slash10}(z))\frac{\mathfrak{c}(z)c_1(z)}{b'(z)}\frac{(\log(b(z)-\lambda+i\iota\epsilon))^2}{b(z)-\lambda+i\iota\epsilon}dz\right\|_{L^2}\notag\\
        &+\left\|\int_{\T}\partial_yG_k(y,z)(1-\chi_{\ast,\delta_0\slash10}(z))\frac{\mathfrak{c}(z)}{b'(z)}\left(b''(z)f(z)-c_2(z)\right)\frac{\log(b(z)-\lambda+i\iota\epsilon)}{b(z)-\lambda+i\iota\epsilon}dz\right\|_{L^2}\notag\\
        &+\left\|\int_{\T}-\partial_yG_k(y,z)(1-\chi_{\ast,\delta_0\slash10}(z))\frac{\mathfrak{c}(z)c_3(z)}{b'(z)}(\log(b(z)-\lambda+i\iota\epsilon))^{\alpha+1}dz\right\|_{L^2}\notag\\
        &+\left\|\int_{\T}\partial_yG_k(y,z)(1-\chi_{\ast,\delta_0\slash10}(z))\frac{\mathfrak{c}(z)}{b'(z)}(-c_4(z)+k^2f(z))\log(b(z)-\lambda+i\iota\epsilon)dz\right\|_{L^2} \notag\\
        \lesssim&\frac{1}{|k|^{1/2}}\left(\|c_1\|_{H^1_k}+\|c_2\|_{H^1_k}+|k|\|f\|_{H^1_k}+\|c_3\|_{L^2}+\|c_{4}\|_{L^2}\right)\notag\\
        \lesssim&\|c_1\|_{H^1_k}+\|c_2\|_{H^1_k}+\|c_{3}\|_{L^2}+\|c_{4}\|_{L^2}.
\end{align}
Thus we proved \eqref{ndrayg2}.

\end{proof}
\begin{lemma}\label{ndpsi}
    Assume that $\kappa k\in\mathbb{Z}\backslash\{0\}, \lambda\in \Sigma\backslash \cup_{j}\Sigma_{j, \delta_0}$ and let $\psi_{k,\epsilon}^{\iota}(y,\lambda)$ with $\iota\in\{\pm\}, \epsilon \in (0, \epsilon_0)$ be as defined in \eqref{F7}. 
    Then we have the bound
    \begin{equation}
        \begin{split}
            \norm{\psi }_{H^1_{k}} \lesssim \frac{\big\|\omega_{0k}\big\|_{H^1_k}}{|k|^{1/2}}.
        \end{split}
    \end{equation}
    Moreover, for any $1\leq p<2$,
    \begin{equation}\label{ndpsi_y}
        \left\|\partial_y\psi_{k,\epsilon}^{\iota}(y,\lambda)-C(y)(1-\chi_{\ast,\delta_0\slash10}(y))\log(b(y)-\lambda+i\iota\epsilon)\right\|_{W^{1,p}}\lesssim\|\omega_{0k}\|_{H^1_k},
    \end{equation}
    with $\left\|C(y)\right\|_{H^1_k}\lesssim\|\omega_{0k}\|_{H^1_k}$.
\end{lemma}
\begin{proof}
    By the Rayleigh equation \eqref{F7}, the estimates follow directly from Lemma \ref{ndraygeneral} by taking $c_2(y) = \omega_{0k}(y)$, $c_1(y)=c_3(y)=c_4(y)\equiv 0$ in \eqref{ndrayf}. 
\end{proof}
\begin{lemma}\label{ndpsi_l}
    Assume that $ \kappa k\in\mathbb{Z}\backslash\{0\},\, \lambda\in \Sigma\backslash \cup_{j}\Sigma_{j, \delta_0},\, \epsilon \in (0, \epsilon_0),\,  \iota\in\{\pm\}$.
    Let $\psi_{k,\epsilon}^{\iota}(y,\lambda)$ be as defined in \eqref{F7}.
    Denote for $y\in \T $,
    \begin{equation}
        \begin{split}
            &\mathcal{S}^\iota_{1,0,\epsilon}(y,\lambda):=\, -(1-\chi_{\ast,\delta_0\slash10}(y))(b''(y) \psi(y, \lambda)-\omega_{0k}(y)) \frac{1}{(b'(y))^2}\log(b(y)-\lambda+i\iota\epsilon).\\
        \end{split}
    \end{equation}
    We have the bounds 
    \begin{equation}\label{ndpsi_l1}
        \norm{  \p_\lambda \psi_{k, \epsilon}^\iota  - \mathcal{S}^\iota_{1,0,\epsilon} }_{H^1_k} \lesssim \frac{\big\|\omega_{0k}\big\|_{H^2_k}}{|k|^{1/2}}.
    \end{equation}
\end{lemma}
\begin{proof}
    By the Rayleigh equation, $\partial_\lambda\psi_{k,\epsilon}^{\iota}(y,\lambda)-\mathcal{S}^\iota_{1,0,\epsilon}$ satisfies
    \begin{equation}
        \left(\partial_y^2-k^2-\frac{b''(y)}{b(y)-\lambda+i\iota\epsilon}\right)\left(\partial_\lambda\psi_{k,\epsilon}^{\iota}(y,\lambda)-\mathcal{S}^\iota_{1,0,\epsilon}\right)=\mathcal{R}^\iota_{1,0,\epsilon}(y,\lambda),
    \end{equation}
    where $\mathcal{R}^\iota_{1,0,\epsilon}$ is explicit by a direct computation.
    By the previous Lemmas, one can simplify it to the form
\begin{equation}
    \mathcal{R}^\iota_{1,0,\epsilon}(y,\lambda) = c_1(y)\frac{\log(b(y)-\lambda+i\iota\epsilon)}{b(y)-\lambda+i\iota\epsilon}+\frac{c_2(y)+\mathfrak{c}(y)\partial_y\psi}{b(y)-\lambda+i\iota\epsilon}+c_{3}(y)\log(b(y)-\lambda+i\iota\epsilon)+c_4(y),
\end{equation}
with $|k|\|c_1\|_{H^1_k}+\|c_2\|_{H^1_k}+\|c_{3}\|_{L^2}+\|c_{4}\|_{L^2}\lesssim\|\omega_{0k}\|_{H^2_k}$. Therefore, \eqref{ndpsi_l1} follows from Lemmas \ref{ndraygeneral}, \ref{ndrayg} directly.
\end{proof}
\begin{lemma}\label{ndpsi_D}
Assume that $\kappa k\in\mathbb{Z}\backslash\{0\}, \lambda\in \Sigma\backslash \cup_{j}\Sigma_{j, \delta_0}$ and let $\psi_{k,\epsilon}^{\iota}(y,\lambda)$ with $\iota\in\{\pm\}, \epsilon \in (0, \epsilon_0)$ be as defined in \eqref{F7}. 
Then we have the bound 
    \begin{equation}
        \begin{split}
            \norm{ D_\lambda \psi_{k, \epsilon}^\iota }_{H^1_k} \lesssim \frac{\big\|\omega_{0k}\big\|_{H^2_k}}{|k|^{1/2}}.
        \end{split}
    \end{equation}
And, for any $1\leq p<2$,
    \begin{equation}
        \left\|\partial_yD_\lambda\psi_{k,\epsilon}^{\iota}(y,\lambda) -C(y)(1-\chi_{\ast,\delta_0\slash10}(y))\log(b(y)-\lambda+i\iota\epsilon)\right\|_{W^{1,p}}\lesssim \|\omega_{0k}\|_{H^2_k},
    \end{equation}
    with $\|C(y)\|_{H^1_k}\lesssim\|\omega_{0k}\|_{H^2_k}$.
\end{lemma}
\begin{proof}
    Using the Rayleigh equation together with Lemmas \ref{ndpsi} and \ref{ndpsi_l}, we obtain 
    \begin{equation}\label{Raypsi_D}
        \begin{split}
            &\left(\partial_y^2-k^2-\frac{b''(y)}{b(y)-\lambda+i\iota\epsilon}\right)D_\lambda\psi_{k,\epsilon}^{\iota}(y,\lambda) \\
            &= \frac{b'''(y)\psi(y,\lambda)-\omega_{0k}'(y)}{b(y)-\lambda+i\iota\epsilon}+b'''(y)\partial_\lambda\psi(y,\lambda)+2b''(y)\partial_y\partial_\lambda\psi(y,\lambda)\\
            &=\frac{c_2(y)}{b(y)-\lambda+i\iota\epsilon}+c_3(y)\log(b(y)-\lambda+i\iota\epsilon)+c_4(y).
        \end{split}
    \end{equation}
    Here $\|c_2\|_{H^1_k}+\|c_3\|_{L^2}+\|c_4\|_{L^2}\lesssim\|\omega_{0k}\|_{H^2_k}$.
    The conclusion follows from Lemma \ref{ndraygeneral}.

\end{proof}

\begin{lemma}\label{ndpsi_l2}
    Assume that $\kappa k\in\mathbb{Z}\backslash\{0\}, \lambda\in \Sigma\backslash \cup_{j}\Sigma_{j, \delta_0},\, \epsilon \in (0, \epsilon_0),\, \iota\in\{\pm\}$, and let $\psi_{k,\epsilon}^{\iota}(y,\lambda)$ be as defined in \eqref{F7}. 
    For $y\in \T$, set
    \begin{equation}\label{ndS20}
        \begin{split}
            \mathcal{S}_{2, 0, \epsilon}^\iota(y, \lambda) &:=   C_1(1-\chi_{\ast,\delta_0\slash10}(y))\frac{b''(y)\psi(y,\lambda)-\omega_{0k}(y)}{b(y)-\lambda+i\iota\epsilon}\\
            &+\left[C_2(y)D_\lambda\psi(y,\lambda)+C_3(y)\psi(y,\lambda)+\Omega_{0k}(y)  \right](1-\chi_{\ast,\delta_0\slash10}(y))\log(b(y)-\lambda+i\iota\epsilon),
        \end{split}
    \end{equation}
    where 
    \begin{equation}
        \begin{split}
            &\Omega_{0k}(y)=C_4(y)\omega_{0k}'(y)+C_5(y)\omega_{0k}(y),\\
            &C_1(y)=\frac{1}{(b'(y))^2},\quad\ C_2(y)=-\frac{2b''(y)}{(b'(y))^3},\quad\ C_3(y)=-\frac{2b'''(y)}{(b'(y))^3}+\frac{3(b''(y))^2}{(b'(y))^4},\\
            &C_4(y)=\frac{2}{(b'(y))^3},\quad\ C_5(y)=-\frac{3b''(y)}{(b'(y))^4}.
        \end{split}
    \end{equation}
    Then we have the bounds
    \begin{equation}
        \norm{\p_\lambda^2 \psi_{k, \epsilon}^\iota   - \mathcal{S}_{2, 0, \epsilon}^\iota }_{H^1_k} \lesssim  \frac{\big\|\omega_{0k}\big\|_{H^3_k}}{|k|^{1/2}}.
    \end{equation}
\end{lemma}
\begin{proof}
    The Rayleigh equation for $\partial_\lambda^2\psi_{k,\epsilon}^{\iota}(y,\lambda)-\mathcal{S}_{2, 0, \epsilon}^\iota$ can be written as
    \begin{equation}
        \left(\partial_y^2-k^2-\frac{b''(y)}{b(y)-\lambda+i\iota\epsilon}\right)\left(\partial_\lambda^2\psi_{k,\epsilon}^{\iota}(y,\lambda)-\mathcal{S}_{2, 0, \epsilon}^\iota\right)=\mathcal{R}^\iota_{2,0,\epsilon}(y,\lambda).
    \end{equation}
    Similar to the proof of Lemma \ref{ndpsi_l}, one can reorganize 
    \begin{equation}
        \begin{split}
            \mathcal{R}^\iota_{2,0,\epsilon}(y,\lambda)=&c_1(y)\frac{\log(b(y)-\lambda+i\iota\epsilon)}{b(y)-\lambda+i\iota\epsilon}+\frac{c_2(y)+\mathfrak{c}_1(y)\partial_yD_\lambda\psi+\mathfrak{c}_2(y)\partial_y\psi}{b(y)-\lambda+i\iota\epsilon}\\
            &+c_3(y)\log(b(y)-\lambda+i\iota\epsilon)+c_4(y).
        \end{split}
    \end{equation}
    Here $k^2\|c_1\|_{H^1_k}+\|c_2\|_{H^1_k}+\|c_3\|_{L^2}+\|c_{4}\|_{L^2}\lesssim\|\omega_{0k}\|_{H^3_k}$. The conclusion follows directly from Lemmas \ref{ndraygeneral}, \ref{ndrayg}.

\end{proof}

\begin{lemma}\label{ndpsi_lD}
   Assume that $\kappa k\in\mathbb{Z}\backslash\{0\}, \lambda\in \Sigma\backslash \cup_{j}\Sigma_{j, \delta_0},\, \epsilon \in (0, \epsilon_0),\, \iota\in\{\pm\}$, and let $\psi_{k,\epsilon}^{\iota}(y,\lambda)$ be as defined in \eqref{F7}. 
    For $y\in \T$, set
    \begin{equation}
        \mathcal{S}_{1,1,\epsilon}^\iota(y,\lambda): =\big(C_1(y)D_\lambda\psi+C_2(y)\psi(y)+\Omega_{0k}(y)\big)(1-\chi_{\ast,\delta_0\slash10}(y))\log(b(y)-\lambda+i\iota\epsilon), 
    \end{equation}
    where 
    \begin{equation}
        \begin{split}
            &\Omega_{0k}(y)=C_3\omega_{0k}'(y)+C_4(y)\omega_{0k}(y),\\
            &C_1(y)=-\frac{b''}{(b')^2},\quad\ C_2(y)=-\frac{b'''}{(b')^2}+\frac{2(b'')^2}{(b')^3},\quad\ C_3(y)=\frac{1}{(b')^2},\quad\ C_4(y)=-\frac{2b''}{(b')^3}.
        \end{split}
    \end{equation}
    Then we have the bounds 
    \begin{equation}
        \|\partial_\lambda D_\lambda\psi_{k,\epsilon}^{\iota}-\mathcal{S}_{1,1,\epsilon}^\iota\|_{H^1_k}\lesssim\frac{\big\|\omega_{0k}\big\|_{H^3_k}}{|k|^{1/2}}.
    \end{equation}
\end{lemma}
\begin{proof}
    Notice that by \eqref{ndS20}, one can rewrite $\mathcal{S}_{1,1,\epsilon}^{\iota}$ as follows.
    \begin{equation}
        \begin{aligned}
        \mathcal{S}_{1,1,\epsilon}^{\iota}={}&b'\left(\partial_\lambda^2\psi-r_{2,0,\epsilon}^\iota\right)
        +\left(-\frac{b''}{b'}\psi+\frac{\omega_{0k}}{b'}\right)
        \frac{1-\chi_{\ast,\delta_0\slash10}(y)}{b-\lambda+i\iota\epsilon}\\
&+\left\{\begin{aligned}
&-\frac{\omega_{0k}'}{(b')^2}
+\left[\frac{b'''}{(b')^2}-\frac{(b'')^2}{(b')^3}\right]\psi\\
&+\frac{b''}{(b')^3}\omega_{0k}+\frac{b''}{(b')^2}D_\lambda\psi
\end{aligned}\right\}
(1-\chi_{\ast,\delta_0\slash10}(y))\log(b-\lambda+i\iota\epsilon).
        \end{aligned}
    \end{equation}
    Here we take $r_{2,0,\epsilon}^\iota(y,\lambda):=\p_\lambda^2 \psi_{k, \epsilon}^\iota   - \mathcal{S}_{2, 0, \epsilon}^\iota$. By Lemma \ref{ndpsi_l2}, $\|r_{2,0,\epsilon}^\iota\|_{H^1_k} \lesssim  |k|^{-1/2}\|\omega_{0k}\|_{H^3_k}$.
    Then $\partial_\lambda D_\lambda\psi_{k,\epsilon}^{\iota}-\mathcal{S}_{1,1,\epsilon}^\iota$ satisfies
    \begin{equation}
        \left(\partial_y^2-k^2-\frac{b''(y)}{b(y)-\lambda+i\iota\epsilon}\right)\left(\partial_\lambda D_\lambda\psi_{k,\epsilon}^{\iota}-\mathcal{S}_{1,1,\epsilon}^\iota\right)=\mathcal{R}^\iota_{1,1,\epsilon}.
    \end{equation}
    Here $\mathcal{R}^\iota_{1,1,\epsilon}$ can be reorganized as 
    \begin{equation}
        \begin{split}
            \mathcal{R}^\iota_{1,1,\epsilon}=&c_1(y)\frac{\log(b(y)-\lambda+i\iota\epsilon)}{b(y)-\lambda+i\iota\epsilon}+\frac{c_2(y)+\mathfrak{c}_1(y)\partial_yD_\lambda\psi+\mathfrak{c}_2(y)\partial_y\psi}{b(y)-\lambda+i\iota\epsilon}\\
            &+c_{3}(y)\log(b(y)-\lambda+i\iota\epsilon)+c_4(y),
        \end{split}
    \end{equation}
    with  $k^2\|c_1\|_{H^1_k}+\|c_2\|_{H^1_k}+\|c_{3}\|_{L^2}+\|c_{4}\|_{L^2}\lesssim\|\omega_{0k}\|_{H^3_k}$. Therefore, our conclusion follows from Lemmas \ref{ndraygeneral}, \ref{ndrayg} directly.


\end{proof}

\begin{lemma}
    Assume that $\kappa k\in\mathbb{Z}\backslash\{0\}, \lambda\in \Sigma\backslash \cup_{j}\Sigma_{j, \delta_0},\, \epsilon \in (0, \epsilon_0),\, \iota\in\{\pm\}$, and let $\psi_{k,\epsilon}^{\iota}(y,\lambda)$ be as defined in \eqref{F7}. 
    Then we have the bound 
    \begin{equation}\label{ndD2:bdd}
        \|D_\lambda^2\psi_{k,\epsilon}^{\iota}\|_{H^1_k}\lesssim\frac{\big\|\omega_{0k}\big\|_{H^3_k}}{|k|^{1/2}}.
    \end{equation}
    In addition, for any $1\leq p<2$,
    \begin{equation}
        \left\|\partial_yD_\lambda^2\psi_{k,\epsilon}^{\iota}(y,\lambda)-C(y)(1-\chi_{\ast,\delta_0\slash10}(y))\log(b(y)-\lambda+i\iota\epsilon)\right\|_{W^{1,p}}\lesssim\|\omega_{0k}\|_{H^3_k},
    \end{equation}
    with $\|C\|_{H^1_k}\lesssim\|\omega_{0k}\|_{H^3_k}$.
\end{lemma}
\begin{proof}
    We argue as in the proof of Lemma \ref{ndpsi_D}. According to the Rayleigh equation together with Lemmas \ref{ndpsi}-\ref{ndpsi_lD}, we have 
    \begin{equation}\label{ndRayD2psi}
        \begin{split}
        &\left(\partial_y^2-k^2-
        \frac{b''}{b-\lambda+i\iota\epsilon}\right)D_\lambda^2\psi\\
        =&\frac{b^{(4)}\psi+2b'''D_\lambda\psi-\omega_{0k}''}{b-\lambda+i\iota\epsilon}+4b''\partial_y\partial_\lambda D_\lambda\psi
        +2b'''\partial_\lambda D_\lambda\psi-2(b'')^2\partial_\lambda^2\psi+2b'''\partial_y\partial_\lambda\psi+b^{(4)}\partial_\lambda\psi\\
        =&\frac{c_2(y)}{b(y)-\lambda+i\iota\epsilon}+c_{3,1}(y)\log(b(y)-\lambda+i\iota\epsilon)+c_{3,2}(y)(\log(b(y)-\lambda+i\iota\epsilon))^2+c_4(y).
        \end{split}
    \end{equation}
    Here $\|c_2\|_{H^1_k}+\|c_{3,1}\|_{L^2}+\|c_{3,2}\|_{L^2}+\|c_{4}\|_{L^2}\lesssim \|\omega_{0k}\|_{H^3_k}$. The conclusion follows directly from Lemma \ref{ndraygeneral}.
\end{proof}


\begin{lemma}\label{ndpsi_l3}
    Assume that $\kappa k\in\mathbb{Z}\backslash\{0\}, \lambda\in \Sigma\backslash \cup_{j}\Sigma_{j, \delta_0},\, \epsilon \in (0, \epsilon_0),\, \iota\in\{\pm\}$, and let $\psi_{k,\epsilon}^{\iota}(y,\lambda)$ be as defined in \eqref{F7}. 
    For $y\in \T$, set
    \begin{equation}
        \begin{split}
            \mathcal{S}_{3, 0,\epsilon}^\iota(y, \lambda) &:=   C_1(1-\chi_{\ast,\delta_0\slash10}(y))\frac{b''(y)\psi(y,\lambda)-\omega_{0k}(y)}{(b(y)-\lambda+i\iota\epsilon)^2}\\
&+\left[\begin{aligned}
&C_2(y)\partial_y\psi(y,\lambda)+C_3(y)D_\lambda\psi(y,\lambda)\\
&+C_4(y)\psi(y,\lambda)+\Omega_{0k,1}(y)
\end{aligned}\right]\frac{1-\chi_{\ast,\delta_0\slash10}(y)}{b(y)-\lambda+i\iota\epsilon}\\
            &+\left[\begin{aligned}
&C_7(y)D_\lambda^2\psi+C_8(y)D_\lambda\psi\\
&+C_9\psi(y,\lambda)+\Omega_{0k,2}(y)
\end{aligned}\right](1-\chi_{\ast,\delta_0\slash10}(y))\log(b(y)-\lambda+i\iota\epsilon),
        \end{split}
    \end{equation}
    where 
    \begin{equation}
        \begin{split}
            &\Omega_{0k,1} = C_5(y)\omega_{0k}'(y)+C_6(y)\omega_{0k}(y),\quad\ \Omega_{0k,2} = C_{10}(y)\omega_{0k}''(y)+C_{11}(y)\omega_{0k}'(y)+C_{12}(y)\omega_{0k}(y),\\
            &C_1(y)=\frac{1}{(b'(y))^2},\quad\ C_2(y)=-\frac{b''}{(b')^3},\quad\ C_3(y)=\frac{3b''}{(b')^3},\\
            &C_4(y)=\frac{2b'''}{(b')^3}-\frac{3(b'')^2}{(b')^4},\quad\ C_5(y)=-\frac{2}{(b')^3},\quad\ C_6(y) =\frac{3b''}{(b')^4}, \\
            &C_7(y)=-\frac{3b''}{(b')^4},\quad\ C_8(y) =-\frac{6b'''}{(b')^4}+\frac{12(b'')^2}{(b')^5},\\
            &C_{9}(y)=-\frac{3b^{(4)}}{(b')^4}+\frac{16b''b'''}{(b')^5}-\frac{15(b'')^3}{(b')^6}-\frac{k^2b''}{(b')^4},\quad\ C_{10}(y)=\frac{3}{(b')^4},\\
            &C_{11}(y)=-\frac{12b''}{(b')^5},\quad\ C_{12}(y) =\frac{k^2}{(b')^4}-\frac{4b'''}{(b')^5}+\frac{15(b'')^2}{(b')^6}.
        \end{split}
    \end{equation}
    Then we have the bounds
    \begin{equation}
        \norm{\p_\lambda^3 \psi_{k, \epsilon}^\iota   - \mathcal{S}_{3, 0,\epsilon}^\iota }_{H^1_k} \lesssim  \frac{\big\|\omega_{0k}\big\|_{H^4_k}}{|k|^{1/2}}.
    \end{equation}
\end{lemma}
\begin{remark}
    From this Lemma, the coefficients for lower derivative terms begin to depend on $k$, such as $C_9(y)$, $C_{12}(y)$ shown above. They come from substituting $\partial_y^2\psi$ and higher $y$-derivatives of $\psi$ by lower derivatives based on the Rayleigh equation. They do not worsen the $k$-dependence of our final estimates, since they can be all absorbed by the highest derivative terms, in this case $\omega_{0k}''$ and $D_\lambda^2\psi$, which determine the final $k$-dependence.
\end{remark}
\begin{proof}
    The Rayleigh equation for $\partial_\lambda^3\psi_{k,\epsilon}^{\iota}(y,\lambda)-\mathcal{S}^\iota_{3,0,\epsilon}$ is 
    \begin{equation}
        \left(\partial_y^2-k^2-\frac{b''(y)}{b(y)-\lambda+i\iota\epsilon}\right)\left(\partial_\lambda^3\psi_{k,\epsilon}^{\iota}(y,\lambda)-\mathcal{S}^\iota_{3,0,\epsilon}\right)=\mathcal{R}^\iota_{3,0,\epsilon}(y,\lambda).
    \end{equation}
Here $\mathcal{R}^\iota_{3,0,\epsilon}$ can be simplified by 
\begin{equation}
    \begin{split}
        \mathcal{R}^\iota_{3,0,\epsilon}(y,\lambda)=&c_1(y)\frac{\log(b(y)-\lambda+i\iota\epsilon)}{b(y)-\lambda+i\iota\epsilon}+\frac{c_2(y)+\mathfrak{c}_1(y)\partial_yD_\lambda^2\psi+\mathfrak{c}_2(y)\partial_yD_\lambda\psi+\mathfrak{c}_3(y)\partial_y\psi}{b(y)-\lambda+i\iota\epsilon}\\
        &+c_3(y)\log(b(y)-\lambda+i\iota\epsilon)+c_4(y),
    \end{split}
\end{equation}
with $k^2\|c_1\|_{H^1_k}+\|c_2\|_{H^1_k}+\|c_3\|_{L^2}+\|c_{4}\|_{L^2}\lesssim\|\omega_{0k}\|_{H^4_k}$. The conclusion follows directly from Lemmas \ref{ndraygeneral}, \ref{ndrayg}.

\end{proof}

\begin{lemma}\label{ndpsi_l2D}
    Assume that $\kappa k\in\mathbb{Z}\backslash\{0\}, \lambda\in \Sigma\backslash \cup_{j}\Sigma_{j, \delta_0},\, \epsilon \in (0, \epsilon_0),\, \iota\in\{\pm\}$, and let $\psi_{k,\epsilon}^{\iota}(y,\lambda)$ be as defined in \eqref{F7}. 
    For $y\in \T$, set
\begin{equation}\label{ndS21}
        \begin{split}
            \mathcal{S}_{2, 1, \epsilon}^\iota(y, \lambda) &:= \left[C_1(y)D_\lambda\psi(y,\lambda)+C_2(y)\psi(y,\lambda)+\Omega_{0k,1}(y)\right]\frac{1-\chi_{\ast,\delta_0\slash10}(y)}{b(y)-\lambda+i\iota\epsilon}\\
            &+\left[
                \begin{aligned}
                    &C_5(y)D_\lambda^2\psi(y,\lambda)+C_6(y)D_\lambda\psi(y,\lambda)\\
                    &+C_7(y)\psi(y,\lambda)+\Omega_{0k,2}(y)
                \end{aligned}    
            \right](1-\chi_{\ast,\delta_0\slash10}(y))\log(b(y)-\lambda+i\iota\epsilon),
        \end{split}
    \end{equation}
    where 
    \begin{equation}
        \begin{split}
&\Omega_{0k,1}(y)=C_3(y)\omega_{0k}'(y)+C_4(y)\omega_{0k}(y),\\
&\Omega_{0k,2}(y)=C_8(y)\omega_{0k}''(y)+C_9(y)\omega_{0k}'(y)+C_{10}(y)\omega_{0k}(y),\\
            &C_1(y)=\frac{b''}{(b')^2},\quad\ C_2(y)=\frac{b'''}{(b')^2}-\frac{2(b'')^2}{(b')^3},\quad\ C_3(y)=-\frac{1}{(b')^2},\quad\ C_4(y)=\frac{2b''}{(b')^3},\\
            &C_5(y)=-\frac{2b''}{(b')^3},\quad\ C_6(y)=-\frac{4b'''}{(b')^3}+\frac{9(b'')^2}{(b')^4},\quad\ C_7(y)=-\frac{2b^{(4)}}{(b')^3}+\frac{12b''b'''}{(b')^4}-\frac{12(b'')^3}{(b')^5},\\
            &C_8(y)=\frac{2}{(b')^3},\quad\ C_9(y)=-\frac{9b''}{(b')^4},\quad\ C_{10}(y)=-\frac{3b'''}{(b')^4}+\frac{12(b'')^2}{(b')^5}.
        \end{split}
    \end{equation}
    Then we have the bounds
    \begin{equation}
        \norm{\p_\lambda^2 D_\lambda \psi_{k, \epsilon}^\iota   - \mathcal{S}_{2, 1, \epsilon}^\iota }_{H^1_k} \lesssim  \frac{\big\|\omega_{0k}\big\|_{H^4_k}}{|k|^{1/2}}.
    \end{equation}
\end{lemma}
\begin{proof}
    Similar to the arguments in Lemma \ref{ndpsi_lD}, one can rewrite $\mathcal{S}_{2,1, \epsilon}^\iota$ as follows
    \begin{equation}
        \begin{split}
            \mathcal{S}_{2,1, \epsilon}^\iota=&b'\left(\partial_\lambda^3\psi-r_{3,0,\epsilon}^\iota\right)\\
&+\left(-\frac{b''}{b'}\psi+\frac{\omega_{0k}}{b'}\right)\frac{(1-\chi_{\ast,\delta_0\slash10}(y))}{(b-\lambda+i\iota\epsilon)^2}\\
&+\left\{\begin{aligned}
&\frac{b''}{(b')^2}\partial_y\psi+\frac{\omega_{0k}'}{(b')^2}
+\left[-\frac{b'''}{(b')^2}+\frac{(b'')^2}{(b')^3}\right]\psi\\
&-\frac{b''}{(b')^3}\omega_{0k}-\frac{2b''}{(b')^2}D_\lambda\psi
\end{aligned}\right\}\frac{1-\chi_{\ast,\delta_0\slash10}(y)}{b-\lambda+i\iota\epsilon}\\
            &+\left\{\begin{aligned}
            &\frac{3b''}{(b')^4}\omega_{0k}'+\frac{k^2(b')^2b''+(b')^2b^{(4)}-4b'b''b'''+3(b'')^3}{(b')^5}\psi\\
            &-\frac{k^2(b')^2-b'b'''+3(b'')^2}{(b')^5}\omega_{0k}+\frac{2b'b'''-3(b'')^2}{(b')^4}D_\lambda\psi\\
            &+\frac{b''}{(b')^3}D_\lambda^2\psi-\frac{\omega_{0k}''}{(b')^3}
            \end{aligned}
            \right\}\\
&\qquad\times(1-\chi_{\ast,\delta_0\slash10}(y))\log(b-\lambda+i\iota\epsilon),
        \end{split}
    \end{equation}
    where $r_{3,0,\epsilon}^\iota(y,\lambda):=\p_\lambda^3 \psi_{k, \epsilon}^\iota   - \mathcal{S}_{3, 0, \epsilon}^\iota$ with $\|r_{3,0,\epsilon}^\iota\|_{H^1_k} \lesssim  |k|^{-1/2}\|\omega_{0k}\|_{H^4_k}$ from Lemma \ref{ndpsi_l3}.
    Then $\partial_\lambda^2D_\lambda\psi_{k, \epsilon}^\iota- \mathcal{S}_{2,1, \epsilon}^\iota$ satisfies 
    \begin{equation}
        \left(\partial_y^2-k^2-\frac{b''(y)}{b(y)-\lambda+i\iota\epsilon}\right)\left(\partial_\lambda^2D_\lambda\psi_{k, \epsilon}^\iota- \mathcal{S}_{2,1, \epsilon}^\iota\right)=\mathcal{R}_{2,1,\epsilon}^{\iota}(y,\lambda).
    \end{equation}
    Here $\mathcal{R}^\iota_{2,1,\epsilon}$ can be regrouped as 
    \begin{equation}
        \begin{split}
            \mathcal{R}_{2,1,\epsilon}^{\iota}(y,\lambda)=&c_1(y)\frac{\log(b(y)-\lambda+i\iota\epsilon)}{b(y)-\lambda+i\iota\epsilon}+\frac{c_2(y)+\mathfrak{c}_1(y)\partial_yD_\lambda^2\psi+\mathfrak{c}_2(y)\partial_yD_\lambda\psi+\mathfrak{c}_3(y)\partial_y\psi}{b(y)-\lambda+i\iota\epsilon}\\
            &+c_3(y)\log(b(y)-\lambda+i\iota\epsilon)+c_4(y),
        \end{split}
    \end{equation}
    with $|k|\|c_1\|_{H^1_k}+\|c_2\|_{H^1_k}+\|c_3\|_{L^2}+\|c_{4}\|_{L^2}\lesssim\|\omega_{0k}\|_{H^4_k}$. 
    The conclusion follows directly from Lemmas \ref{ndraygeneral}, \ref{ndrayg}.
\end{proof}

\begin{lemma}
    Assume that $\kappa k\in\mathbb{Z}\backslash\{0\}, \lambda\in \Sigma\backslash \cup_{j}\Sigma_{j, \delta_0},\, \epsilon \in (0, \epsilon_0),\, \iota\in\{\pm\}$, and let $\psi_{k,\epsilon}^{\iota}(y,\lambda)$ be as defined in \eqref{F7}. 
    For $y\in \T$, set
    \begin{equation}\label{ndS12}
        \begin{split}
            \mathcal{S}_{1, 2, \epsilon}^\iota(y, \lambda) &:= \left[
                \begin{aligned}
                    &C_1(y)D_\lambda^2\psi(y,\lambda)+C_2(y)D_\lambda\psi(y,\lambda)\\
                    &+C_3(y)\psi(y,\lambda)+\Omega_{0k}(y)
                \end{aligned}    
            \right](1-\chi_{\ast,\delta_0\slash10}(y))\log(b(y)-\lambda+i\iota\epsilon),
        \end{split}
    \end{equation}
    where 
    \begin{equation}
        \begin{split}
            &\Omega_{0k}(y)=C_4(y)\omega_{0k}''(y)+C_{5}(y)\omega_{0k}'(y)+C_{6}(y)\omega_{0k}(y), \\
            &C_1(y)=-\frac{b''}{(b')^2},\quad\ C_2(y)=-\frac{2b'''}{(b')^2}+\frac{4(b'')^2}{(b')^3},\quad\ C_3(y)=-\frac{b^{(4)}}{(b')^2}+\frac{6b''b'''}{(b')^3}-\frac{6(b'')^3}{(b')^4},\\
            &C_4(y)=\frac{1}{(b')^2},\quad\ C_5(y)=-\frac{4b''}{(b')^3},\quad\ C_6(y)=-\frac{2b'''}{(b')^3}+\frac{6(b'')^2}{(b')^4}.
        \end{split}
    \end{equation}
    Then we have the bounds
    \begin{equation}\label{ndlamD2:bdd}
        \norm{\p_\lambda D_\lambda^2 \psi_{k, \epsilon}^\iota   - \mathcal{S}_{1, 2, \epsilon}^\iota }_{H^1_k} \lesssim  \frac{\big\|\omega_{0k}\big\|_{H^4_k}}{|k|^{1/2}}.
    \end{equation}
\end{lemma}
\begin{proof}
    Similar to the arguments in Lemma \ref{ndpsi_lD}, one can rewrite $\mathcal{S}_{1,2, \epsilon}^\iota$ as follows
    \begin{equation*}
        \begin{split}
            \mathcal{S}_{1,2, \epsilon}^\iota=&-(b')^2(\partial_\lambda^3\psi-r_{3,0,\epsilon}^\iota)+(b''\psi-\omega_{0k})\frac{1-\chi_{\ast,\delta_0\slash10}(y)}{(b-\lambda+i\iota\epsilon)^2}+2b'(\partial_\lambda^2D_\lambda\psi-r_{2,1,\epsilon}^\iota)\\
            &+b''(\partial_\lambda^2\psi-r_{2,0,\epsilon}^\iota)+\frac{b''}{b'}(D_\lambda\psi-\partial_y\psi)\frac{1-\chi_{\ast,\delta_0\slash10}(y)}{b-\lambda+i\iota\epsilon}\\
            &-\frac{k^2(b''\psi-\omega_{0k})}{(b')^2}(1-\chi_{\ast,\delta_0\slash10}(y))\log(b-\lambda+i\iota\epsilon),\\
        \end{split}
    \end{equation*}
    where $r_{2,1,\epsilon}^\iota(y,\lambda):=\p_\lambda^2 D_\lambda \psi_{k, \epsilon}^\iota   - \mathcal{S}_{2, 1, \epsilon}^\iota$ with $\|r_{2,1,\epsilon}^\iota\|_{H^1_k} \lesssim  |k|^{-1/2}\|\omega_{0k}\|_{H^4_k}$ from Lemma \ref{ndpsi_l2D}.
    Then $\partial_\lambda D_\lambda^2\psi_{k, \epsilon}^\iota- \mathcal{S}_{1,2, \epsilon}^\iota$ satisfies 
    \begin{equation}
        \left(\partial_y^2-k^2-\frac{b''(y)}{b(y)-\lambda+i\iota\epsilon}\right)\left(\partial_\lambda D_\lambda^2\psi_{k, \epsilon}^\iota- \mathcal{S}_{1,2, \epsilon}^\iota\right)=\mathcal{R}_{1,2,\epsilon}^{\iota}(y,\lambda).
    \end{equation}
    $\mathcal{R}^\iota_{1,2,\epsilon}$ is of the form
    \begin{equation}
        \begin{split}
            \mathcal{R}_{1,2,\epsilon}^{\iota}(y,\lambda)=&c_1(y)\frac{\log(b(y)-\lambda+i\iota\epsilon)}{b(y)-\lambda+i\iota\epsilon}+\frac{c_2(y)+\mathfrak{c}_1(y)\partial_yD_\lambda^2\psi+\mathfrak{c}_2(y)\partial_yD_\lambda\psi+\mathfrak{c}_3(y)\partial_y\psi}{b(y)-\lambda+i\iota\epsilon}\\
            &+c_3(y)\log(b(y)-\lambda+i\iota\epsilon)+c_4(y),
        \end{split}
    \end{equation}
    with $|k|\|c_1\|_{H^1_k}+\|c_2\|_{H^1_k}+\|c_3\|_{L^2}+\|c_{4}\|_{L^2}\lesssim\|\omega_{0k}\|_{H^4_k}$. 
    The conclusion follows directly from Lemmas \ref{ndraygeneral}, \ref{ndrayg}.
\end{proof}

\section{Bounds on \texorpdfstring{$\psi_{k,\epsilon}^{\iota}$}{psi}: Degenerate case}


In this section we aim to capture the  smallness in the spectral density function which is at the heart of the depletion mechanism. As was already observed in \cite{IIJ24}, we can capture this smallness by introducing the modified spectral density function $\phi^\iota_{k,\epsilon}(y,\lambda)$: for $\kappa k\in\Z\backslash\{0\}, \iota\in\{\pm\}, \lambda\in \Sigma_{j,\delta_0}, 0<\epsilon<\epsilon_0$, (recall $\epsilon_0$ from Proposition~\ref{DeC7})
\begin{equation}\label{BDC1}
\psi^\iota_{k,\epsilon}(y,\lambda)=\phi^\iota_{k,\epsilon}(y,\lambda)+\frac{\chi_{j, \delta_0}(y)}{b''(y)}\omega_{0k}(y).
\end{equation}
Note that we suppress the dependence on $j$ in $\phi_{k, \epsilon}^\iota$ since it is already implicit in the domain of $\lambda$. 
Then $\phi^\iota_{k,\epsilon}(y,\lambda)$ satisfies for $y\in \T $,
\begin{equation}\label{BDC2}
(-\partial_y^2 + k^2)\phi^\iota_{k,\epsilon}(y,\lambda)+\frac{b''(y)}{b(y)-\lambda+i\iota\epsilon}\phi^\iota_{k,\epsilon}(y,\lambda)=g^\iota_{k,\epsilon}(y,\lambda),\end{equation}
where for $ \kappa k\in\Z\backslash\{0\}, \iota\in\{\pm\}, \lambda\in \Sigma_{j,\delta_0},  0<\epsilon<\epsilon_0$
\begin{equation}\label{BDC3}
g^\iota_{k,\epsilon}(y,\lambda):=\frac{1-\chi_{j, \delta_0}(y)}{b(y)-\lambda+i\iota\epsilon}\omega_{0k}(y)-(k^2-\partial_y^2)\Big[\frac{\chi_{j, \delta_0}(y)}{b''(y)}\omega_{0k}(y)\Big].
\end{equation}
We record the following bounds for $g^\iota_{k,\epsilon}(y,\lambda)$. 

\begin{lemma}
\label{g:bdds}
    We have the following bounds for $g^\iota_{k,\epsilon}(y,\lambda)$.
    \begin{subequations}
         \begin{align}
    |g^\iota_{k,\epsilon}(y,\lambda)| + |\p_\lambda g^\iota_{k,\epsilon}(y,\lambda)| + |\p_\lambda^2 g^\iota_{k,\epsilon}(y,\lambda)| + |\p_\lambda^3 g^\iota_{k,\epsilon}(y,\lambda)| &\lesssim \frac{\norm{\omega_{0k}}_{H_k^4} }{|k|^{3/2}}, \\
     |\p_yg^\iota_{k,\epsilon}(y,\lambda)| + |\D g^\iota_{k,\epsilon}(y,\lambda)| +  |\p_\lambda \D g^\iota_{k,\epsilon}(y,\lambda)| +  |\p_\lambda^2 \D g^\iota_{k,\epsilon}(y,\lambda)| &\lesssim \frac{\norm{\omega_{0k}}_{H_k^4}   }{|k|^{1/2}}, \\ 
    \norm{\p_y^2g^\iota_{k,\epsilon}(y,\lambda)}_{L^2} + \norm{\D^2g^\iota_{k,\epsilon}(y,\lambda)}_{L^2} + \norm{\p_\lambda\D^2g^\iota_{k,\epsilon}(y,\lambda)}_{L^2}  &\lesssim \norm{\omega_{0k}}_{H_k^4}. 
    \end{align}
    \end{subequations}
 Additionally, when $|k \delta(\Lambda)| \gtrsim 1$, we have the following bound for $\alpha \in \{0,1 \}$.
\begin{subequations}
\begin{equation}
\begin{split}
&\left|\int_\T\varrho_{j,k}^\alpha(y;\Lambda)\p_y^\alpha\mathcal G_k^j(y,z;\Lambda)
\sum_{r=0}^3|\p_\lambda^r g_{k,\epsilon}^\iota(z,\lambda)|\,dz\right|\\
&\qquad\lesssim\frac{\|\omega_{0k}\|_{H_k^4}}{|k|^{7/2}},
\end{split}
\end{equation}
\begin{equation}
\begin{split}
&\left|\int_\T\varrho_{j,k}^\alpha(y;\Lambda)\p_y^\alpha\mathcal G_k^j(y,z;\Lambda)
\left(|\p_y g_{k,\epsilon}^\iota(z,\lambda)|+\sum_{r=0}^2|\p_\lambda^r\D g_{k,\epsilon}^\iota(z,\lambda)|\right)\,dz\right|\\
&\qquad\lesssim\frac{\|\omega_{0k}\|_{H_k^4}}{|k|^{5/2}},
\end{split}
\end{equation}
\begin{equation}
\begin{split}
&\left|\int_\T\varrho_{j,k}^\alpha(y;\Lambda)\p_y^\alpha\mathcal G_k^j(y,z;\Lambda)
\left(|\p_y^2g_{k,\epsilon}^\iota(z,\lambda)|+|\D^2g_{k,\epsilon}^\iota(z,\lambda)|
+|\p_\lambda\D^2g_{k,\epsilon}^\iota(z,\lambda)|\right)\,dz\right|\\
&\qquad\lesssim\frac{\|\omega_{0k}\|_{H_k^4}}{|k|^{3/2}}.
\end{split}
\end{equation}
\end{subequations}
    
\end{lemma}

When $|k|\delta(\Lambda) \gtrsim 1$, Lemma \ref{g:bdds} is sufficient to get the needed bounds on $\phi_{k, \epsilon}^\iota$. However,  when $|k|\delta(\Lambda) \ll 1$ to capture the necessary vanishing at the critical layer and the quadratic smallness, we need a further decomposition. We write 
\begin{align}
    g_{k,j*} := g^\iota_{k,\epsilon}(y_{j*}, \lambda),
\end{align}
where we can safely drop the $\lambda, \iota, \epsilon$ dependence from $g_{k,j*}$ since these only enter through the first term in \eqref{BDC3} and 
\begin{align}
    \frac{1-\chi_{j ,\delta_0}(y_{j*})}{b(y_{j*})-\lambda+i\iota\epsilon}\omega_{0k}(y_{j*}) = 0.
\end{align}
We decompose the solution $\phi_{k, \epsilon}^\iota$ as 
\begin{align}
\label{BDC3.5}
   \phi_{k, \epsilon}^\iota(y, \lambda) = g_{k,j*}\chi_{j, \frac{\delta_0}{4k }}(y)\Phi_{ \epsilon}^\iota(y - y_{j*},\lambda) + E_{k, \epsilon}^\iota(y - y_{j*}, \lambda), 
\end{align}
where $\Phi_{ \epsilon}^\iota$ will solve an ``inner problem'' \eqref{Inner_problem} and $E_{k, \epsilon}^\iota$ will be the remainder satisfying \eqref{E:eq}.  

\subsection{Inner problem}
\label{Inner}

For $\lambda \in \Sigma_{j, \delta_0}$, $\epsilon \in (0, \epsilon_0)$ we wish to solve the following equation:
\begin{align}
    &-\p_y^2 G(y,z, \lambda + i \iota \epsilon) + \frac{b_j''(y)}{b_j(y) - \lambda + i \iota \epsilon} G(y,z, \lambda + i \iota \epsilon) = \delta(y - z ) + \delta(y+ z ), \\
    &y \in [-\delta_0, \delta_0],  \quad z\in [0,\delta_0], \quad  G( \pm \delta_0, z, \lambda + i \iota \epsilon) = 0,
\end{align}
where $b_j(y)  := b(y + y_{j*})$. One can check that   
\begin{align}
\label{reg}
    \varphi_{\epsilon}^\iota(y) := b_j(y) -\lambda + i\iota \epsilon
\end{align}
is an exact homogeneous solution. Reduction of order then immediately gives 
\begin{align}
\label{sing}
    H_{r, \epsilon}^\iota(y) &:= \varphi_{\epsilon}^\iota(y) \int_y^{ \delta_0} \frac{ds}{(\varphi_{\epsilon}^\iota)^2(s)}, \\
    H_{l, \epsilon}^\iota(y) &:= \varphi_{\epsilon}^\iota(y) \int_{-\delta_0}^y \frac{ds}{(\varphi_{\epsilon}^\iota)^2(s)},
\end{align}
are also homogeneous solutions which satisfy the right and left boundary conditions, respectively. We will occasionally refer to $H_{r, \epsilon}^\iota , H_{l, \epsilon}^\iota$ as the ``singular'' homogeneous solutions and $\varphi_{\epsilon}^\iota$ as the ``regular'' homogeneous solution. It is instructive to consider the case where $b_j$ is even about 0 which motivates our construction. In that case, it is possible to restrict to considering $y \in [0,\delta_0]$ and impose a Neumann boundary condition at the origin. It then follows that $\varphi_{\epsilon}^\iota$ satisfies the boundary condition at 0. This leads to a Wronskian of $\varphi_{\epsilon}^\iota$ and $H_{r, \epsilon}^\iota$  which is $-1$. This is notable because it is independent of $\lambda$ which significantly simplifies the analysis. Therefore, the Green's function can be expressed as 
\begin{align}
    G(y,z, \lambda + i \iota \epsilon) =  \varphi_{\epsilon}^\iota(y)H_{r, \epsilon}^\iota(z) \mathbbm{1}_{y\leq z} +  \varphi_{\epsilon}^\iota(z)H_{r, \epsilon}^\iota(y) \mathbbm{1}_{y\geq z}.
\end{align}
Around the origin, the Green's function is given by the regular homogeneous solution $\varphi_{\epsilon}^\iota$ which vanishes at the critical layers $b_j(y) = \lambda$ in the limit $\epsilon \to 0$. This will have several technical benefits. A similar simplification happens in the analysis of the Rayleigh equation linearized around a smooth monotone decreasing vortex when considering $k = 1$ \cite{BCV19, IJ}.

We now turn to the case where $b_j$ is not symmetric around 0. We want to keep the structure that there are two symmetrically placed sources and that the $y$ dependence of the Green's function is given by the regular solution on the interval $[-z,z]$. This motivates considering the following modified inner problem 
\begin{align}
\label{Inn:modG}
    &-\p_y^2G(y,z, \lambda + i \iota \epsilon) + \frac{b_j''(y)}{b_j(y) - \lambda + i \iota \epsilon}G(y,z, \lambda + i \iota \epsilon) = a_\epsilon^\iota(z)\delta(y-z) + \delta(y+ z), 
\end{align}
where 
\begin{align}
\label{a:jump}
    a_\epsilon^\iota(z):= \frac{H_{l, \epsilon}^\iota(-z)}{H_{r, \epsilon}^\iota(z) },
\end{align}
with the convention that the value of $a$ at $ z = \delta_0$ is given by L'H\^opital's rule. A direct calculation verifies that the solution of \eqref{Inn:modG} is given by 
\begin{align}
    \label{in:green}
    G(y,z, \lambda + i \iota \epsilon) = 
    \begin{cases}
    \varphi_{\epsilon}^\iota(-z)H_{l, \epsilon}^\iota(y), & y \in [-\delta_0, -z] \\
    \varphi_{\epsilon}^\iota(y) H_{l, \epsilon}^\iota(-z), & y \in [-z, z] \\
   a_\epsilon^\iota(z)\varphi_{\epsilon}^\iota(z)H_{r, \epsilon}^\iota(y), & y \in [z,\delta_0]
    \end{cases}.
\end{align}

Intuitively,  $a_\epsilon^\iota -1$ captures the deviation from $b_j$ being even about 0. This is quantified in the lemma below. 
\begin{lemma}
\label{a:bdds}
    We have the following bounds on $a_\epsilon^\iota(y)$ for $y \in [\delta(\Lambda), \delta_0]$, $\lambda \in \Sigma_{j, \delta_0}\setminus\{b(y_{j*})\}$, $\epsilon \in (0, \epsilon_0)$ and $\alpha \in [0,4]\cap \Z$:
\begin{subequations}
   \begin{align}
        |  \p_y^\alpha( a_\epsilon^\iota(y) -1)| &\lesssim \varrho^{1-\alpha}(y;\Lambda), \label{a_y}\\
    |\p_\lambda^\alpha  (a_\epsilon^\iota(y) -1)| &\lesssim \varrho^{1- 2\alpha}(y;\Lambda), \label{a_lambda} \\
    | D_{j,\lambda}^\alpha (a_\epsilon^\iota(y) -1) | &\lesssim \varrho^{1-\alpha}(y;\Lambda), \label{a_good} \\
   \lim_{\epsilon \to 0} \p_\lambda^\alpha (a_\epsilon^+(y) - a_\epsilon^-(y)) &= 0, \label{a_dif}
   \end{align} 
\end{subequations}
where 
    \begin{align}
    \label{djgood}
        D_{j,\lambda} : = b_j'(y) \p_\lambda + \p_y. 
    \end{align}
\end{lemma}

\begin{proof}
For $y \in [ \delta_0/2, \delta_0]$, the bounds are immediate so we consider the more difficult case where $y \in [\delta(\Lambda), \delta_0/2]$.
We introduce the quantity 
\begin{align}
    d(y) :=  b_j(-y) - b_j(y).
\end{align}
By the assumption on $b$ and Taylor's theorem, we have
\begin{align}
    |d(y)| \lesssim |y|^3.
\end{align}
We wish to expand $H_{l, \epsilon}^\iota(-y)$ in terms of $H_{r, \epsilon}^\iota(y)$ and use $d$ to quantify the error. We first write 
\begin{align}
    \varphi_{\epsilon}^\iota(-y) = b_j(-y) - \lambda + i \iota \epsilon = \varphi_{\epsilon}^\iota(y) + d(y).
\end{align}
Therefore,
\begin{align}
    H_{l, \epsilon}^\iota(-y) &= \varphi_{\epsilon}^\iota(-y) \int_{-\delta_0}^{-y} \frac{ds}{(\varphi_{\epsilon}^\iota)^2(s)} = (\varphi_{\epsilon}^\iota(y) + d(y)) \int_{-\delta_0}^{-y}\frac{ds}{(\varphi_{\epsilon}^\iota(-s) - d(s))^2 } \\
    &= \left (\varphi_{\epsilon}^\iota(y) + d(y) \right) \left(    \int_{-\delta_0}^{-y} \frac{ds}{(\varphi_{\epsilon}^\iota)^2(-s)}   + \int_{-\delta_0}^{-y} \frac{1}{(\varphi_{\epsilon}^\iota)^2(-s)} \left[ \left(\sum_{n \geq 0} \frac{d^n(s)}{(\varphi_{\epsilon}^\iota)^n(-s)} \right)^2    -1    \right] ds    \right) \\
    &= H_{r, \epsilon}^\iota(y) + d(y)\int_y^{\delta_0}\frac{ds}{
    (\varphi_{\epsilon}^\iota)^2(s)}  + \varphi_{\epsilon}^\iota(y ) \int_{y}^{\delta_0} \frac{1}{(\varphi_{\epsilon}^\iota)^2(s)}
    \left[ \left(\sum_{n \geq 0} \frac{d^n(-s)}{(\varphi_{\epsilon}^\iota)^n(s)} \right)^2    -1    \right] ds   \\
    &+ d(y) \int_{y}^{\delta_0} \frac{1}{(\varphi_{\epsilon}^\iota)^2(s)} \left[ \left(\sum_{n \geq 0} \frac{d^n(-s)}{(\varphi_{\epsilon}^\iota)^n(s)} \right)^2    -1    \right] ds\\
    &= H_{r, \epsilon}^\iota(y) + R(y),
\end{align}
where 
\begin{align}
    R(y) &:=  d(y)\int_y^{\delta_0}\frac{ds}{
    (\varphi_{\epsilon}^\iota)^2(s)}  + \varphi_{\epsilon}^\iota(y ) \int_{y}^{\delta_0} \frac{1}{(\varphi_{\epsilon}^\iota)^2(s)}
    \left[ \left(\sum_{n \geq 0} \frac{d^n(-s)}{(\varphi_{\epsilon}^\iota)^n(s)} \right)^2    -1    \right] ds \\
    &+ d(y) \int_{y}^{\delta_0
    } \frac{1}{(\varphi_{\epsilon}^\iota)^2(s)} \left[ \left(\sum_{n \geq 0} \frac{d^n(-s)}{(\varphi_{\epsilon}^\iota)^n(s)} \right)^2    -1    \right] ds.
\end{align}

Therefore,
\begin{align}
    a_\epsilon^\iota(y) - 1 = \frac{H_{r, \epsilon}^\iota(y) + R(y)}{H_{r, \epsilon}^\iota(y)} -1 = \frac{R(y)}{H_{r, \epsilon}^\iota(y)}. 
\end{align}

By definition of $\delta(\Lambda), \delta_0$, we have that 
\begin{align}
  \frac{|b_j''(0)|y^2}{4} \lesssim  |\varphi_{\epsilon}^\iota(y)| \lesssim |b_j''(0)|y^2  , y \in [\delta(\Lambda) , \delta_0   ].
\end{align}
Therefore,  
\begin{align}
    |\p_y^\alpha R(y)| &\lesssim |y|^{-\alpha},\\
    |\p_y^\alpha (H_{r, \epsilon}^\iota)^{-1}(y)| & \lesssim |y|^{1 -\alpha},\\
    |\p_\lambda^\alpha R(y)| &\lesssim |y|^{-2\alpha},\\
    |\p_\lambda^\alpha (H_{r, \epsilon}^\iota)^{-1}(y)| &\lesssim |y|^{1 -2\alpha}.
\end{align}
The bounds then follow from the product rule and that  $y \approx \varrho(y;\Lambda)$ for $y \in [\delta(\Lambda), \delta_0/2]$. 

Finally, we have
\begin{align}
    a_\epsilon^+(y) - a_\epsilon^-(y) &= \frac{ H_{r, \epsilon}^-(y )H_{l, \epsilon}^+(-y)  - H_{r, \epsilon}^+(y)H_{l, \epsilon}^-(-y)    }{H_{r, \epsilon}^-(y)  H_{r, \epsilon}^+(y)   }. 
\end{align}
The proof then follows from the following estimates, valid for $y \in [\delta(\Lambda), \delta_0/2], \alpha \in \{0,1,2,3,4\}$
\begin{align}
    \left| \varrho^\alpha(y;\Lambda ) \p_y^\alpha \frac{1}{H_{r, \epsilon}^\iota(y)  H_{r, \epsilon}^\iota(y)} \right| \lesssim \varrho^2(y;\Lambda),
\end{align}
\begin{align}
    \left| \varrho^\alpha(y;\Lambda ) \p_y^\alpha\left( \int_y^{\delta_0} \frac{ds}{(\varphi_{\epsilon}^+)^2(s)} - \int_y^{\delta_0} \frac{ds}{ (\varphi_{\epsilon}^-)^2(s) } \right) \right| \lesssim  \frac{\epsilon}{\varrho^5(y; \Lambda)}, 
\end{align}
and the product rule.
\end{proof}
For $\epsilon \in (0, \epsilon_0)$ we define $\Phi_{\epsilon}^\iota$ as the solution of the following equation
\begin{subequations}
\label{Inner_problem}
    \begin{align}
-\p_y^2 \Phi_{\epsilon}^\iota  + \frac{b_j''(y)}{b_j(y) - \lambda + i \iota \epsilon} \Phi_\epsilon^\iota 
     &=  (1- \chi_{\delta(\Lambda)}(y))\mathbbm{1}_{y\leq 0} +  a_\epsilon^\iota(y)(1- \chi_{\delta(\Lambda)}(y))\mathbbm{1}_{y\geq 0},\label{Inner_problem:forcing} \\ 
     \Phi_\epsilon^\iota(\pm \delta_0 , \lambda ) &= 0.
\end{align}
\end{subequations}
With the definition of \eqref{Inner_problem}, $\Phi_\epsilon^\iota$ is given by 
\begin{align}
\label{Phi:rep}
    \Phi_\epsilon^\iota (y, \lambda) = \int_0^{\delta_0} G(y,z, \lambda + i \iota \epsilon) (1 - \chi_{\delta(\Lambda)}(z))\mathbbm{1}_{z \geq 0}  \, dz .
\end{align}

It will also be convenient to express the right hand side of \eqref{Inner_problem:forcing} as 
\begin{align}
\label{Inner_pro2}
      -\p_y^2 \Phi_{\epsilon}^\iota  + \frac{b_j''(y)}{b_j(y) - \lambda + i \iota \epsilon} \Phi_\epsilon^\iota  = 1 - \chi_{\delta(\Lambda)}(y) + (a_\epsilon^\iota(y) - 1) (1 - \chi_{\delta(\Lambda)} (y)) \mathbbm{1}_{y\geq 0}.
\end{align}
\begin{remark}
    Compared to the equation in \eqref{toy:inner}, we no longer have a right hand side that is identically equal to 1 but instead vanishes on an approximately $\delta(\Lambda)$ sized neighborhood of the origin. This has the benefit that $\Phi_{\epsilon}^\iota$ defined by \eqref{Inner_problem} vanishes at the critical layer in the limit as $\epsilon \to 0$, see Lemma \ref{phi:van}. The excluded portion is consistent with quadratic vanishing and therefore we include it in $E_{k, \epsilon}^\iota$.
 \end{remark}

As a direct consequence of  \eqref{Phi:rep} and the support assumption for the right hand side of \eqref{Inner_problem}, we have the following bounds. 
\begin{lemma}\label{lem:phi-bounds}
Let $\Phi_\epsilon^\iota$ solve \eqref{Inner_problem}. For $\lambda \in \Sigma_{j,\delta_0} \setminus \{b(y_{j*})\},  \epsilon \in (0, \epsilon_0)$, $\theta \in (0,1/4]$, and $\alpha \in \{ 0,1\} $ we have
    \begin{subequations}
    \label{phi:lam}
  \begin{align}
     | \varrho^\alpha(y;\Lambda) \p_y^\alpha \Phi_\epsilon^\iota(y,\lambda )| &\lesssim \varrho^{2- \theta}(y; \Lambda), \\
 |\varrho^\alpha(y;\Lambda) \p_y^\alpha \p_\lambda \Phi_\epsilon^\iota(y,\lambda )|  &\lesssim  \varrho^{- \theta}(y; \Lambda),   \\
 |\varrho^\alpha(y;\Lambda) \p_y^\alpha \p_\lambda^2 \Phi_\epsilon^\iota(y,\lambda )|  &\lesssim \frac{1}{\varrho(y;\Lambda) \delta(\Lambda)}, \\
|\varrho^\alpha(y;\Lambda) \p_y^\alpha \p_\lambda^3 \Phi_\epsilon^\iota(y,\lambda )|   &\lesssim  \frac{1}{\varrho(y;\Lambda) \delta^{3}(\Lambda)}, \\
|\varrho^\alpha(y;\Lambda) \p_y^\alpha \p_\lambda^4 \Phi_\epsilon^\iota(y,\lambda )|   &\lesssim  \frac{1}{\varrho(y;\Lambda) \delta^{5}(\Lambda)},
    \end{align}
    \end{subequations}
  \begin{subequations}
  \label{phi:good}
  \begin{align}
 |\varrho^\alpha(y;\Lambda) \p_y^\alpha D_{j,\lambda} \Phi_\epsilon^\iota(y,\lambda )|  &\lesssim \varrho(y;\Lambda),   \\
 |\varrho^\alpha(y;\Lambda) \p_y^\alpha \p_\lambda D_{j,\lambda}   \Phi_\epsilon^\iota(y,\lambda )|  &\lesssim  \frac{1}{\delta(\Lambda)}, \\
 |\varrho^\alpha(y;\Lambda) \p_y^\alpha D_{j,\lambda}^2 \Phi_\epsilon^\iota(y,\lambda )|  &\lesssim  \frac{\varrho(y;\Lambda)}{\delta(\Lambda)}, \\
 |\varrho^\alpha(y;\Lambda) \p_y^\alpha \p_\lambda D_{j,\lambda}^2 \Phi_\epsilon^\iota(y,\lambda )|  &\lesssim  \frac{\varrho(y;\Lambda)}{\delta^3(\Lambda)}.
    \end{align}
    \end{subequations}
\end{lemma}

\begin{proof}
The proof follows by direct calculation. We briefly highlight the origin of the $\theta$ factor. Assume that $y  > 4\delta(\Lambda)$. From \eqref{in:green}, we have 
\begin{align}
    |\Phi_\epsilon^\iota (y,\lambda) - \varphi_\epsilon^\iota(y)\int_{y}^{\delta_0}  H_{l, \epsilon}^\iota(-z) dz| \lesssim \varrho^2(y; \Lambda). 
 \end{align}
 However,  for $y \in (4 \delta(\Lambda), \delta_0/2)$
\begin{align}
\label{log:loss}
    |\int_{y}^{\delta_0}  H_{l, \epsilon}^\iota(-z) dz| \approx |\log \varrho(y; \Lambda)|, \quad  |\varphi_\epsilon^\iota(y)| \approx \varrho^2(y;\Lambda).
\end{align}
Instead of propagating a logarithmic factor through the estimates, we opt to track an algebraic loss encoded by $\theta$. This logarithmic correction disappears after the second spectral derivative because $\p_\lambda^2 \varphi_\epsilon^\iota(y) = 0$.

\end{proof}
The following lemma covers the essential vanishing properties for $\Phi_\epsilon^\iota$. 
\begin{lemma}
\label{phi:van}

Let $\Phi_\epsilon^\iota$ solve \eqref{Inner_problem} with $\lambda \in \Sigma_{j,\delta_0} \setminus \{b(y_{j*})\}$, $\epsilon \in (0, \epsilon_0)$, and $\theta \in (0,1/4]$. The following limits hold. 
   \begin{subequations}
       \begin{align}
       \lim_{\epsilon \to 0} D_{j, \lambda}^\alpha \Phi_\epsilon^\iota(y_\lambda, \lambda) &= 0, \quad b_j(y_\lambda) = \lambda , \alpha \in \{ 0,1,2\},\label{lem:phi1}\\
        \lim_{\epsilon \to 0}     \p_y^\alpha \p_\lambda^\beta (\Phi_\epsilon^+(y,\lambda ) - \Phi_\epsilon^-(y,\lambda ) )    &= 0, \quad   \alpha \in \{ 0,1\}, \beta \in \{0,1,2,3\}, \label{lem:phi2}\\
        \lim_{\epsilon \to 0}     D_{j,\lambda}^\alpha (\Phi_\epsilon^+(y,\lambda ) - \Phi_\epsilon^-(y,\lambda ) )    &= 0, \quad  \alpha \in \{0,1,2\}. \label{lem:phi3}
   \end{align}
   \end{subequations}
   The convergence is understood in a pointwise sense for fixed $y,\lambda$. 
   \end{lemma}
\begin{proof} 
The proof of \eqref{lem:phi1} follows by a direct calculation using \eqref{Phi:rep}. The proof of \eqref{lem:phi2} follows from the observation that 
    \begin{align}
        \Phi_\epsilon(y, \lambda) := \Phi_\epsilon^+(y, \lambda) - \Phi_\epsilon^-(y, \lambda)
    \end{align}
    satisfies the equation 
    \begin{align}
        -\p_y^2 \Phi_\epsilon + \frac{b_j''(y)}{b_j(y) - \lambda + i \epsilon} \Phi_\epsilon &= - \Phi_\epsilon^-(y,\lambda) b_j''(y) \left( \frac{1}{b_j(y) - \lambda + i \epsilon } - \frac{1}{b_j(y) -\lambda - i \epsilon   } \right) \\
        &+ (a_\epsilon^+(y ) - a_\epsilon^-(y))(1 - \chi_{\delta(\Lambda)}(y))\mathbbm{1}_{y\geq 0}.
     \end{align}
     By Lemma \ref{a:bdds} and \eqref{lem:phi1} we have the following convergence in the sense of distributions: 
     \begin{align}
        \lim_{\epsilon \to 0}  \frac{2 i \epsilon \Phi_\epsilon^-(y,\lambda) b_j''(y)  }{ (b_j(y) - \lambda)^2 + \epsilon^2 }  + (a_\epsilon^+(y ) - a_\epsilon^-(y))(1 - \chi_{\delta(\Lambda)}(y))\mathbbm{1}_{y\geq 0} = 0.
     \end{align}
    The case for higher vanishing then follows by the bounds in \eqref{phi:lam}, \eqref{phi:good} and the assumption that $\lambda \neq b(y_{j*})$. 
\end{proof}



\subsection{Remainder term $E_{k, \epsilon}^\iota$}
Now that we have defined $\Phi_\epsilon^\iota$, \eqref{BDC2} and \eqref{BDC3.5} imply that  $E_{k ,\epsilon}^\iota$ must satisfy 
\begin{align}
\label{E:eq}
    &(-\partial_y^2 + k^2)E^\iota_{k,\epsilon}(y,\lambda)+\frac{b_j''(y)}{b_j(y)-\lambda+i\iota\epsilon}E^\iota_{k,\epsilon}(y,\lambda) = \mathcal{R}_{0, 0, \epsilon}^\iota(y,\lambda),
\end{align}
where 
\begin{align}
\label{eq:Erhs}
    \mathcal{R}_{0, 0, \epsilon}^\iota(y,\lambda) &:= g^\iota_{k,\epsilon}(y + y_{j*},\lambda)  - g_{k,j*}\chi_{ \frac{\delta_0} {4k}   }(y)
\notag\\
&\quad + g_{k,j*}\chi_{\frac{\delta_0}{4k}}(y)\chi_{\delta(\Lambda)}(y) \\
    &+ g_{k,j*}\chi_{ \frac{\delta_0} {4k}   } (y) (1 - a_\epsilon^\iota(y))(1 - \chi_{\delta(\Lambda)}(y))\mathbbm{1}_{y\geq 0}  \notag \\ 
    &- g_{k, j*}\chi_{ \frac{\delta_0} {4k}   }(y)k^2 \Phi_\epsilon^\iota(y,\lambda) 
     + g_{k,j*}\chi_{ \frac{\delta_0} {4k}   }''(y) \Phi_\epsilon^\iota(y,\lambda) + g_{k, j*}2\chi_{ \frac{\delta_0} {4k}   }'(y) \p_y\Phi_\epsilon^\iota(y,\lambda). \notag
\end{align}
The essential property of $E_{k, \epsilon}^{\iota}$ is that we have removed the term proportional to $1 - \chi_{\delta(\Lambda)}(y)$ from the right hand side via the introduction of $\Phi_\epsilon^\iota$. This would have led to a logarithmic loss as in \eqref{log:loss}.  
Integrating over a short interval whose length is comparable to $\delta(\Lambda)$ is consistent with quadratic smallness which is what allows us to cut out the critical layers and ensure that $\Phi_\epsilon^\iota$ vanishes at the critical layers when $\epsilon \to 0$. Adding $\chi_{ \frac{\delta_0} {4k}   }$ ensures that the term $k^2 \Phi_\epsilon^\iota$ can be viewed perturbatively for $|\delta(\Lambda)k| \lesssim 1$. However, the logarithmic loss of $\Phi_\epsilon^\iota$ relative to quadratic smallness (captured by $\theta > 0$) is ultimately responsible for the loss of decay in $k$.

The following lemma will be useful for bounding spectral derivatives and good derivatives of $\mathcal{R}_{0,0,\epsilon}^\iota(y,\lambda)$.
\begin{lemma}
\label{R00:bdd}
Let $\lambda \in \Sigma_{j,\delta_0} \setminus \{b(y_{j*})\},  \epsilon \in (0, \epsilon_0), \theta \in (0,1/4], \iota\in \{-, +\}$, and  $ \mathcal{R}_{0, 0, \epsilon}^\iota(y,\lambda)$ be defined as in \eqref{eq:Erhs}. Then   the following bounds hold for $|y| \lesssim |k|^{-1}, |k |\delta(\Lambda) \ll 1$ with a constant depending on $\theta$.
\begin{subequations}\label{R00:lam}
\begin{align}
    |\mathcal{R}_{0,0,\epsilon}^\iota(y,\lambda)| &\lesssim \frac{  \norm{\omega_{0k}}_{H_k^4}   }{|k|}  \left(   \frac{\mathbbm{1}_{|y| \lesssim \delta(\Lambda)   }}{|k|^{1/2}} +  \varrho^{1/2}(y;\Lambda) +   |k|^{3/2}\varrho^{2 - \theta}(y;\Lambda) \right), \\
    |\p_\lambda \mathcal{R}_{0, 0, \epsilon}^\iota(y,\lambda)| &\lesssim \frac{ \norm{\omega_{0k}}_{H_k^4}   }{|k|^{3/2}}  \left(   \frac{\mathbbm{1}_{|y| \lesssim \delta(\Lambda)   } }{\delta^2(\Lambda)} + \frac{1}{\varrho(y;\Lambda)} +  \frac{|k|^{2}}{\varrho^{ \theta}(y;\Lambda)} \right), \\
    |\p_\lambda^2 \mathcal{R}_{0, 0, \epsilon}^\iota(y,\lambda)| &\lesssim    \frac{ \norm{\omega_{0k}}_{H_k^4}   }{|k|^{3/2}}  \left(   \frac{\mathbbm{1}_{|y| \lesssim \delta(\Lambda)   } }{\delta^4(\Lambda)} + \frac{1}{\varrho^3(y;\Lambda)} +  \frac{|k|^{2}}{\delta(\Lambda)\varrho(y;\Lambda)} \right), \\
    |\p_\lambda^3 \mathcal{R}_{0, 0, \epsilon}^\iota(y,\lambda)|  &\lesssim    \frac{ \norm{\omega_{0k}}_{H_k^4}   }{|k|^{3/2}}  \left(   \frac{\mathbbm{1}_{|y| \lesssim \delta(\Lambda)   } }{\delta^6(\Lambda)} + \frac{1}{\varrho^5(y;\Lambda)} +  \frac{|k|^{2}}{\delta^3(\Lambda)\varrho(y;\Lambda)} \right),
\end{align}
\end{subequations}
\begin{subequations}\label{R00:good}
\begin{align}
    |D_{j,\lambda} \mathcal{R}_{0, 0, \epsilon}^\iota(y,\lambda)| &\lesssim \frac{\norm{\omega_{0k}}_{H_k^4}  }{|k|^{1/2}} \left(   1 +   \frac{\mathbbm{1}_{|y| \lesssim \delta(\Lambda) }}{|k| \delta(\Lambda)}  + |k| \varrho(y;\Lambda) +\mathbbm{1}_{|y| \approx |k|^{-1}} |k|^{\theta} \right), \\
    |\p_\lambda D_{j,\lambda} \mathcal{R}_{0, 0, \epsilon}^\iota(y,\lambda)| &\lesssim \frac{\norm{\omega_{0k}}_{H_k^4}  }{|k|^{1/2}} \left(   \frac{\mathbbm{1}_{|y| \lesssim \delta(\Lambda) }}{|k| \delta^3(\Lambda)}   +  \frac{|k|}{\delta(\Lambda)} + |k|^{2 + \theta }\mathbbm{1}_{|y| \approx |k|^{-1}}  \right), \\
    |\p_\lambda^2 D_{j,\lambda} \mathcal{R}_{0, 0, \epsilon}^\iota(y,\lambda)| &\lesssim \frac{\norm{\omega_{0k}}_{H_k^4}  }{|k|^{1/2}} \left(   \frac{\mathbbm{1}_{|y| \lesssim  \delta(\Lambda) }}{|k|\delta^5(\Lambda)}   +  \frac{|k|}{\delta^3(\Lambda)}  \right),
\end{align}
\begin{equation}
\begin{aligned}
&|D_{j,\lambda}^2 ( \mathcal{R}_{0, 0, \epsilon}^\iota(y,\lambda) - g^\iota_{k,\epsilon}(y + y_{j*},\lambda))| \\
     &\lesssim \norm{\omega_{0k}}_{H_k^4}   \left(  \frac{\mathbbm{1}_{|y| \approx \delta(\Lambda)  }}{\delta^2(\Lambda) |k|^{3/2} }   + \frac{1}{|k|^{3/2} \varrho(y; \Lambda)  }   + \frac{|k|^{1/2} \varrho(y;\Lambda) }{\delta(\Lambda)} + |k|^{1/2 + \theta} \mathbbm{1}_{|y| \approx |k|^{-1}}      \right),
\end{aligned}
\end{equation}
\begin{equation}
\begin{aligned}
&|\p_\lambda D_{j,\lambda}^2 ( \mathcal{R}_{0, 0, \epsilon}^\iota(y,\lambda) - g^\iota_{k,\epsilon}(y + y_{j*},\lambda))| \\
     &\lesssim \norm{\omega_{0k}}_{H_k^4}   \left(  \frac{\mathbbm{1}_{|y| \approx \delta(\Lambda)  }}{\delta^4(\Lambda) |k|^{3/2} } + \frac{1}{|k|^{3/2} \varrho^3(y; \Lambda)  }   + \frac{|k|^{1/2} \varrho(y;\Lambda) }{\delta^3(\Lambda)} + |k|^{5/2 + \theta} \mathbbm{1}_{|y| \approx |k|^{-1}}      \right).
\end{aligned}
\end{equation}
\end{subequations}
\end{lemma}
\begin{proof}

The proof follows by direct computation using Lemmas \ref{g:bdds}, \ref{a:bdds}, and \ref{lem:phi-bounds}.

\end{proof}

We have the following Lemma which quantifies the desired smallness properties for $\mathcal{R}_{0,0,\epsilon}^\iota(y,\lambda)$. 
\begin{lemma}
\label{gr:bdds}
Let $\lambda \in \Sigma_{j,\delta_0} \setminus \{b(y_{j*})\},  \epsilon \in (0, \epsilon_0), \theta \in (0, 1/4]$, $\iota\in \{-, +\}$, and  $ \mathcal{R}_{0, 0, \epsilon}^\iota(y,\lambda)$ be defined as in \eqref{eq:Erhs}. Then   the following bounds hold for $\alpha \in \{0,1\}$ and $|k \delta(\Lambda)| \ll 1$ with a constant depending on $\theta$.
\begin{subequations}
    \begin{align}
   |\int_\T {\varrho_{j,k}^\alpha}(y; \Lambda)  \p_y^\alpha  \mathcal{G}_k^j(y,z; \Lambda) \mathcal{R}_{0, 0, \epsilon}^\iota(z - y_{j*},\lambda) dz |&\lesssim   
   \frac{ \norm{\omega_{0k}}_{H_k^4}   \varrho_{j,k}^2(y; \Lambda)}{|k|^{3/2 - \theta}}, \\
   |\int_\T {\varrho_{j,k}^\alpha}(y; \Lambda)  \p_y^\alpha  \mathcal{G}_k^j(y,z; \Lambda)\p_\lambda \mathcal{R}_{0, 0, \epsilon}^\iota(z - y_{j*},\lambda) dz |  &\lesssim \frac{ \norm{\omega_{0k}}_{H_k^4}   }{|k|^{3/2 - \theta}  }, \\
  |\int_\T {\varrho_{j,k}^\alpha}(y; \Lambda)  \p_y^\alpha  \mathcal{G}_k^j(y,z; \Lambda) {\D} \mathcal{R}_{0, 0, \epsilon}^\iota(z - y_{j*},\lambda) dz| &\lesssim \frac{\norm{\omega_{0k}}_{H_k^4}  \varrho_{j, k}(y; \Lambda) }{|k|^{3/2 -\theta}},   \\
   |\int_\T {\varrho_{j,k}^\alpha}(y; \Lambda)  \p_y^\alpha  \mathcal{G}_k^j(y,z; \Lambda) \p_\lambda^2 \mathcal{R}_{0, 0, \epsilon}^\iota(z - y_{j*},\lambda) dz| &\lesssim    \frac{ \norm{\omega_{0k}}_{H_k^4}   }{|k|^{3/2- \theta } \delta(\Lambda) {\varrho_j}(y; \Lambda)},  \\ 
   |\int_\T {\varrho_{j,k}^\alpha}(y; \Lambda)  \p_y^\alpha  \mathcal{G}_k^j(y,z; \Lambda)\p_\lambda {\D} \mathcal{R}_{0, 0, \epsilon}^\iota(z - y_{j*},\lambda) dz| &\lesssim \frac{\norm{\omega_{0k}}_{H_k^4}  }{|k|^{3/2 - \theta } \delta(\Lambda)},  \\
    |\int_\T {\varrho_{j,k}^\alpha}(y; \Lambda)  \p_y^\alpha  \mathcal{G}_k^j(y,z; \Lambda)\p_\lambda^3 \mathcal{R}_{0, 0, \epsilon}^\iota(z - y_{j*},\lambda) dz|  &\lesssim    \frac{ \norm{\omega_{0k}}_{H_k^4}   }{|k|^{3/2 - \theta} \delta^3(\Lambda) \varrho_{j}(y; \Lambda)  },  \\
    |\int_\T {\varrho_{j,k}^\alpha}(y; \Lambda)  \p_y^\alpha  \mathcal{G}_k^j(y,z; \Lambda) \p_\lambda^2 {\D} \mathcal{R}_{0, 0, \epsilon}^\iota(z - y_{j*},\lambda) dz| &\lesssim \frac{\norm{\omega_{0k}}_{H_k^4}  }{|k|^{3/2 - \theta} \delta^3(\Lambda) }, \\
    |\int_\T {\varrho_{j,k}^\alpha}(y; \Lambda)  \p_y^\alpha  \mathcal{G}_k^j(y,z; \Lambda) D_\lambda^2 ( \mathcal{R}_{0, 0, \epsilon}^\iota(z - y_{j*},\lambda) - {g^\iota_{k,\epsilon}(z,\lambda)}) dz|  &\lesssim \frac{\norm{\omega_{0k}}_{H_k^4}  {\varrho_{j,k}}(y; \Lambda)  }{|k|^{3/2-\theta}\delta(\Lambda)}, \\
      |\int_\T {\varrho_{j,k}^\alpha}(y; \Lambda)  \p_y^\alpha  \mathcal{G}_k^j(y,z; \Lambda) \p_\lambda D_\lambda^2 ( \mathcal{R}_{0, 0, \epsilon}^\iota(z - y_{j*},\lambda) - {g^\iota_{k,\epsilon}(z,\lambda)}) dz|  &\lesssim \frac{\norm{\omega_{0k}}_{H_k^4}  {\varrho_{j,k}}(y; \Lambda)  }{|k|^{3/2-\theta}\delta^3(\Lambda)}, \\
     |\int_\T {\varrho_{j,k}^\alpha}(y; \Lambda)  \p_y^\alpha  \mathcal{G}_k^j(y,z; \Lambda) D_\lambda^2 {g^\iota_{k,\epsilon}(z,\lambda)} dz|  &\lesssim \frac{\norm{\omega_{0k}}_{H_k^4}}{|k|^{3/2}},\\
     |\int_\T {\varrho_{j,k}^\alpha}(y; \Lambda)  \p_y^\alpha  \mathcal{G}_k^j(y,z; \Lambda) \p_\lambda D_\lambda^2 {g^\iota_{k,\epsilon}(z,\lambda)} dz|  &\lesssim \frac{\norm{\omega_{0k}}_{H_k^4}}{|k|^{3/2}}.
\end{align}
\end{subequations}
\end{lemma}
\begin{proof}
    The proof follows by Lemmas \ref{mGk50} and \ref{R00:bdd}{, using Lemma~\ref{g:bdds} for the data terms}.
\end{proof}

\section{Higher order spectral regularity:  Degenerate case }
\label{sec:d:bdds}

In this section we prove bounds for the higher order spectral regularity of $E_{k, \epsilon}^\iota, \phi_{k, \epsilon}^\iota$. The key technical difficulty is identifying terms which are too singular and removing them from the spectral density function as in \cite{IIJ24}. Due to the more refined bounds needed in the current paper, we need to compute significantly more terms.  To begin, we state analogues of the non-degenerate estimates in Section~\ref{sec:nd:bdds} which will be used in a similar way in this section.  

In what follows we use the following regime-dependent notation for a sufficiently small $c \in (0,1)$:
\begin{align}
\label{sdf:cases}
\p_\lambda^\alpha \D^\beta \digamma_{k,\epsilon}^\iota(y,\lambda) = 
\begin{cases}
   \p_\lambda^\alpha \D^\beta   E_{k, \epsilon}^\iota (y - y_{j*},\lambda), & |k\delta(\Lambda)| \leq 2c\\
    \p_\lambda^\alpha \D^\beta \phi_{k, \epsilon}^\iota(y,\lambda), & |k \delta(\Lambda)| \geq c
\end{cases}.
\end{align}
On the overlap either representative may be chosen; the estimates hold for each, and the displayed derivatives are taken on the chosen representative.




\begin{lemma}\label{dint}
    Assume that $\kappa k\in\mathbb{Z}\backslash\{0\}, \lambda\in \Sigma_{j, \delta_0} \setminus \{b(y_{j*})\}$, $\iota\in\{\pm\}, \epsilon\in(0,\epsilon_0)$, $\alpha \in\{0,1,2,\cdots\}$, and that $\sigma_1 \in [0,1], \sigma_2\in [-2,0], \sigma_1 + \sigma_2 \in [-2, 1]$. The following implicit constants are allowed to depend on $\alpha$. 
    In this lemma, functions on $S_{3\delta(\Lambda)}^j$ are extended by zero to $\T$, and the local $H^1$ functions are assumed to have zero endpoint traces.
    \begin{enumerate}
    \item For any function $f(y)\in L^2(S_{3 \delta(\Lambda)}^j)$ with $\text{supp }f \subseteq S_{3 \delta(\Lambda)}^j$, we have the following estimates.
    \begin{align}\label{dlogl2}
        &\left\|\int_{\T} \mathcal{G}_{k}^j(y,z, \Lambda) \chi_{ j, \delta(\Lambda) }(z) f(z)\log^\alpha \left( \frac{b(z)-\lambda+i\iota\epsilon   }{\delta^2(\Lambda)} \right)dz\right\|_{X^{\sigma_1, \sigma_2}} \\
        &\lesssim \norm{f}_{L^2(S_{3 \delta(\Lambda)}^j)} \delta^{\sigma_1}(\Lambda)d_{j,k}^{\sigma_2 + 3/2}.
        \notag
    \end{align}
    \item For any function $f(y) \in H^1(S_{3 \delta(\Lambda)}^j)$  with $\text{supp }f \subseteq S_{3 \delta(\Lambda)}^j$
    we have the following estimates.
    \begin{align}\label{dfracl2}
        \left\|\int_{\T}\mathcal{G}_k^j(y,z, \Lambda ) \chi_{j,\delta(\Lambda)}(z)f(z)\frac{\log^\alpha( \frac{b(z)-\lambda+i\iota\epsilon}{\delta^2(\Lambda)})}{b(z)-\lambda+i\iota\epsilon}dz\right\|_{X^{\sigma_1, \sigma_2}}
        \lesssim \norm{f}_{X^{0,0}} \delta^{\sigma_1}(\Lambda) d_{j,k}^{\sigma_2}.
    \end{align}
   \item For a  function $f(y,\lambda ) \in H^1(\T)$, we assume that there exist functions  $c_1(y,\lambda), c_2(y,\lambda) \in H^1(S_{3 \delta(\Lambda)}^j)$, $c_3(y,\lambda), c_4(y,\lambda) \in L^2(S_{3 \delta(\Lambda)}^j)$ with support in $ S_{3 \delta(\Lambda)}^j$, such that in a distributional sense we have  
   \begin{align}\label{drayf}
       &\left(-\partial_y^2+k^2+\frac{b''(y)}{b(y)-\lambda+i\iota\epsilon}\right)f(y,\lambda) \\
            &= c_1(y,\lambda)\frac{\log \frac{(b(y)-\lambda+i\iota\epsilon)}{\delta^2(\Lambda)}}{b(y)-\lambda+i\iota\epsilon}+c_2(y,\lambda)\frac{1}{b(y)-\lambda+i\iota\epsilon}
            +c_{3}(y,\lambda)\log^\alpha \frac{(b(y)-\lambda+i\iota\epsilon)}{\delta^2(\Lambda)}  +c_4(y,\lambda). \notag
   \end{align}
    Then we have the following:
 \begin{align}\label{dfracl3}
        &\left\|\int_{\T}\mathcal{G}_k^j(y,z, \Lambda ) \chi_{j,\delta}^*(z)\frac{\p_z f(z)}{b(z)-\lambda+i\iota\epsilon}dz\right\|_{X^{\sigma_1, \sigma_2}}\\
        &\lesssim \delta^{\sigma_1}(\Lambda)d_{j,k}^{\sigma_2 -1} \left[ \norm{c_1}_{X^{0,0}} + \norm{c_2}_{X^{0,0}} +         \delta^{1/2}(\Lambda) d_{j,k} (  \norm{c_3}_{L^2(S_{3 \delta(\Lambda)}^j)}  + \norm{c_4}_{L^2(S_{3 \delta(\Lambda)}^j)}  )  \right]. \notag
    \end{align}
    
    \end{enumerate}
\end{lemma}

\begin{proof}
The proof of \eqref{dlogl2} immediately follows from the Green-function bounds and the allowed ranges for $\sigma_1, \sigma_2$. The proof of \eqref{dfracl2} follows by the proof of Lemma \ref{DeC4} and the additional identity 
    \begin{align}
       \p_y \log^{\alpha +1 }  \left( \frac{b(y) - \lambda + i \epsilon}{\delta^2(\Lambda)}  \right)  = (\alpha + 1) \frac{b'(y)}{b(y) - \lambda + i \epsilon} \log^{\alpha }  \left( \frac{b(y) - \lambda + i \epsilon}{\delta^2(\Lambda)}  \right),
    \end{align}
  to handle the extra logarithmic factor appearing relative to those appearing in $T_\Theta$ in \eqref{DeC3}. To prove  \eqref{dfracl3}, observe by extending $c_j$ by 0, we can apply Proposition \ref{DeC7} which yields  
  \begin{align}
      \norm{f}_{X^{0,0}} \lesssim \norm{c_1}_{X^{0,0}} + \norm{c_2}_{X^{0,0}} + \delta^{1/2}(\Lambda)d_{j,k}(\norm{c_3}_{L^2(S_{3 \delta(\Lambda)}^j) } + \norm{c_4}_{L^2(S_{3 \delta(\Lambda)}^j) }).
  \end{align}
 Integrating by parts, using \eqref{drayf}, and \eqref{dfracl2} with $\alpha = 1,2$ then imply the desired result.
\end{proof}
We will apply \eqref{dfracl3} with $f = \digamma_{k, \epsilon}^\iota, \D \digamma_{k, \epsilon}^\iota, \D^2 \digamma_{k, \epsilon}^\iota $. More precisely, we  will decompose the right hand side of the equations that  $\D^\alpha \digamma_{k, \epsilon}^\iota$, $\alpha = 0,1,2$ satisfy (\eqref{dhr1}, \eqref{gd1F}, \eqref{gd2F}) into parts supported in $S_{3 \delta(\Lambda)}^j$ and the complement. For the parts of the solution generated by the complement, since these terms are not singular, they can be estimated directly using Lemma \ref{mGk50}. For the parts of the solution generated by the terms supported in $S_{3 \delta(\Lambda)}^j$, we justify that \eqref{drayf} holds by appealing to the appropriate singular expansions proved below.

\begin{lemma}\label{DCM1}
Assume that $\kappa k\in\Z\backslash\{0\}, \lambda\in \Sigma_{j, \delta_0} \setminus \{b(y_{j*}) \}$, $\iota\in\{\pm\}, \epsilon \in (0, \epsilon_0) $. Then for any $\theta \in (0, 1/4]$ we have the following bound: 
\begin{equation}\label{DCM2}
\begin{split}
\norm{\digamma_{k, \epsilon}^\iota }_{X^{0, -2}} \lesssim \frac{\big\|\omega_{0k}\big\|_{H^4_k}}{|k|^{3/2- \theta}}.
\end{split}
\end{equation}
\end{lemma}
\begin{proof}
From \eqref{BDC2} and \eqref{E:eq} we have 
\begin{align}
\label{dhr1}
   \left( -\p_y^2 + k^2 + \frac{b''(y)}{b(y) - \lambda  + i\iota \epsilon } \right) \digamma_{k, \epsilon}^\iota (y,\lambda ) = r(y,\lambda),
\end{align}
where 
\begin{align}
\label{rhs:F}
    r(y,\lambda) := 
    \begin{cases}
        \mathcal{R}_{0,0 ,\epsilon}^\iota(y- y_{j*},\lambda ), & |k \delta(\Lambda)| \ll 1\\
        g_{k, \epsilon}^\iota(y,\lambda ), & |k \delta(\Lambda )| \gtrsim 1
    \end{cases}. 
\end{align}
We reformulate \eqref{dhr1} as 
\begin{align}
    \digamma_{k, \epsilon}^\iota + T_\Theta \digamma_{k, \epsilon}^\iota = \int_\T \mathcal{G}_k^j(y,z; \Lambda) r(z,\lambda) dz .
\end{align}
The bound then follows from Proposition~\ref{DeC7} and Lemmas~\ref{gr:bdds} and~\ref{g:bdds}.

\end{proof}

\begin{lemma}\label{lem:lam1}
Assume that $\kappa k\in\Z\backslash\{0\}, \lambda\in \Sigma_{j, \delta_0} \setminus \{b(y_{j*}) \}$, $\iota\in\{\pm\}, \epsilon \in (0, \epsilon_0) $. Then for any $\theta \in (0, 1/4]$ we have the following bound:
\begin{equation}\label{DCM7}
\norm{  \p_\lambda \digamma_{k, \epsilon}^\iota  - \mathcal{S}^\iota_{1,0,\epsilon} }_{X^{0,0}} \lesssim \frac{\big\|\omega_{0k}\big\|_{H^4_k}}{|k|^{3/2- \theta}},
\end{equation}
where
\begin{equation}\label{DCM7.0}
\begin{split}
&\mathcal{S}^\iota_{1, 0, \epsilon}(y,\lambda):=\, -\chi_{j, \delta}^*(y)\digamma^\iota_{k,\epsilon}(y,\lambda)\frac{b''(y)}{(b'(y))^2}\log\frac{b(y)-\lambda+i\iota\epsilon}{\delta^2(\Lambda)}.\\
\end{split}
\end{equation}


\end{lemma}
\begin{proof}
Equations \eqref{BDC2}, \eqref{E:eq}, and \eqref{rhs:F} yield
\begin{align}
    \left(-\partial_y^2 + k^2 + \frac{b''(y)}{b(y) - \lambda + i \iota \epsilon } \right) \partial_\lambda  \digamma_{k, \epsilon}^\iota(y,\lambda) =  - \frac{b''(y)  \digamma_{k, \epsilon}^\iota  }{(b(y) -\lambda + i \iota \epsilon)^2} + \partial_\lambda r(y,\lambda ). 
\end{align}
Denoting the right-hand side of the equation satisfied by $\partial_\lambda  \digamma_{k, \epsilon}^\iota(y,\lambda) - \mathcal{S}_{1,0, \epsilon}^\iota $ as $\mathcal{R}_{1,0, \epsilon}^\iota$, we solve for $\p_\lambda \digamma_{k, \epsilon}^\iota   -  \mathcal{S}_{1,0, \epsilon}^\iota$ as follows:
\begin{align}
    \p_\lambda \digamma_{k, \epsilon}^\iota   -  \mathcal{S}_{1,0, \epsilon}^\iota   + T_\Theta (\p_\lambda \digamma_{k, \epsilon}^\iota   -  \mathcal{S}_{1,0, \epsilon}^\iota) = \int_\T \mathcal{G}_k^j(y,z; \Lambda) \mathcal{R}_{1,0, \epsilon}^\iota(z,\lambda) dz .
\end{align}
After a direct computation to determine $\mathcal{R}_{1,0, \epsilon}^\iota$,
\eqref{DCM7} follows from  Proposition \ref{DeC7}, the bound \eqref{DCM2}, and Lemmas \ref{dint}, \ref{g:bdds}, \ref{gr:bdds}, and \ref{mGk50}. 
\end{proof}
\begin{lemma}\label{lem:gd1}
Assume that $\kappa k\in\Z\backslash\{0\}, \lambda\in \Sigma_{j, \delta_0} \setminus \{b(y_{j*}) \}$, $\iota\in\{\pm\}, \epsilon \in (0, \epsilon_0) $. Then for any $\theta \in (0, 1/4]$ we have the following bound:
\begin{equation}\label{gDCM2}
\begin{split}
\norm{ \D \digamma_{k, \epsilon}^\iota }_{X^{0,-1}} \lesssim \frac{\big\|\omega_{0k}\big\|_{H^4_k}}{|k|^{3/2- \theta}}.
\end{split}
\end{equation}
\end{lemma}
\begin{proof}
Equations \eqref{BDC2}, \eqref{E:eq}, and \eqref{rhs:F} yield
\begin{align}
\label{gd1F}
&\left(
-\partial_y^2+k^2
+\frac{b''(y)}{b(y)-\lambda+i\iota \epsilon}
\right)\D \digamma_{k, \epsilon}^\iota(y,\lambda)
\nonumber\\
&\quad =
-2b''(y)\partial_y\partial_\lambda \digamma_{k, \epsilon}^\iota(y,\lambda)
-\frac{b'''(y)\digamma_{k, \epsilon}^\iota(y,\lambda)}
       {b(y)-\lambda+i\iota \epsilon} -b'''(y)\partial_\lambda \digamma_{k, \epsilon}^\iota(y,\lambda)  +\D r(y,\lambda).
\end{align}
Denoting the right-hand side  of \eqref{gd1F} as $\mathcal{R}_{0,1, \epsilon}^\iota$, we solve for $\D \digamma_{k, \epsilon}^\iota$ as follows:
\begin{align}
  \D \digamma_{k, \epsilon}^\iota   + T_\Theta (\D \digamma_{k, \epsilon}^\iota) = \int_\T \mathcal{G}_k^j(y,z; \Lambda) \mathcal{R}_{0,1, \epsilon}^\iota(z,\lambda) dz .
\end{align}
The bound \eqref{gDCM2} follows from  Proposition \ref{DeC7},  and  Lemmas  \ref{lem:lam1}, \ref{dint}, \ref{g:bdds}, \ref{gr:bdds}, and \ref{mGk50}.
\end{proof}

\begin{lemma}\label{DCM100}
Assume that $\kappa k\in\Z\backslash\{0\}, \lambda\in \Sigma_{j, \delta_0} \setminus \{b(y_{j*}) \}$, $\iota\in\{\pm\}, \epsilon \in (0, \epsilon_0)$. Then for any $\theta \in (0, 1/4]$ we have the following bound: 
\begin{equation}\label{DCM9.4}
\norm{ \delta (\Lambda) (\p_\lambda^2 \digamma_{k, \epsilon}^\iota   - \mathcal{S}_{2, 0, \epsilon}^\iota )}_{X^{1,0}} \lesssim  \frac{\big\|\omega_{0k}\big\|_{H^4_k}}{|k|^{3/2- \theta}},
\end{equation}
where 
\begin{equation}\label{DCM7.4}
\begin{split}
\mathcal{S}_{2, 0, \epsilon}^\iota(y, \lambda) &:=   \frac{\chi_{j, \delta}^*(y) C_1(y)\digamma_{k,\epsilon}^\iota(y,\lambda)}{b(y) - \lambda + i \iota \epsilon}   \\
&+ \chi_{j, \delta}^*(y) \left( C_2(y) \D \digamma_{k,\epsilon}^\iota(y,\lambda) + C_3(y)\digamma_{k,\epsilon}^\iota(y,\lambda)   \right) \log (\frac{b(y) -\lambda  + i\iota \epsilon }{\delta^2(\Lambda)}),
\end{split}
\end{equation}
\begin{align}
    C_1(y) = \frac{b''(y)}{(b'(y))^2}, \quad C_2(y) = \frac{ -2 b''(y)}{(b'(y))^3}, \quad  C_3(y) = \frac{3 (b''(y))^2 - 2b'''(y) b'(y)  }{(b'(y))^4}.
\end{align}

\end{lemma}
\begin{remark}
Using the heuristic that $\digamma_{k, \epsilon}^\iota$ has dimension 2 and $\lambda$ has dimension 2, one may expect that \eqref{DCM9.4} holds with $X^{0, 0}$ and no factor of $\delta(\Lambda)$. However, due to the asymptotics of the modified Green's function, it can not convert bounds on the data of the form $\varrho_{j}^{-2}(y; \Lambda)$ into boundedness globally. However, this boundedness can be recovered in a $\delta(\Lambda)$ neighborhood of the critical point which is what \eqref{DCM9.4} captures. This ``failure to gain two powers'' persists for higher spectral derivatives and is the reason why \eqref{DCM35} needs a factor of $\delta^3(\Lambda)$. 
\end{remark}    
\begin{proof}
Equations \eqref{BDC2}, \eqref{E:eq}, and \eqref{rhs:F} yield
\begin{equation}
\begin{aligned}
&\left(-\partial_y^2+k^2+\frac{b''(y)}{b(y)-\lambda+i\iota\epsilon}\right)
\partial_\lambda^2\digamma_{k,\epsilon}^\iota(y,\lambda)\\
&\quad=-2\frac{b''(y)\digamma_{k,\epsilon}^\iota(y,\lambda)}{(b(y)-\lambda+i\iota\epsilon)^3}
-2\frac{b''(y)\partial_\lambda\digamma_{k,\epsilon}^\iota(y,\lambda)}{(b(y)-\lambda+i\iota\epsilon)^2}
+\partial_\lambda^2r(y,\lambda).
\end{aligned}
\end{equation}
Denoting the right-hand side of the equation satisfied by $\partial_\lambda^2  \digamma_{k, \epsilon}^\iota(y,\lambda) - \mathcal{S}_{2,0, \epsilon}^\iota $ as $\mathcal{R}_{2,0, \epsilon}^\iota$, we solve for $\p_\lambda^2 \digamma_{k, \epsilon}^\iota   -  \mathcal{S}_{2,0, \epsilon}^\iota$ as follows:
\begin{align}
    \p_\lambda^2 \digamma_{k, \epsilon}^\iota   -  \mathcal{S}_{2,0, \epsilon}^\iota   + T_\Theta (\p_\lambda^2 \digamma_{k, \epsilon}^\iota   -  \mathcal{S}_{2,0, \epsilon}^\iota) = \int_\T \mathcal{G}_k^j(y,z; \Lambda) \mathcal{R}_{2,0, \epsilon}^\iota(z,\lambda) dz .
\end{align}
After a direct computation to determine $\mathcal{R}_{2,0, \epsilon}^\iota$,
\eqref{DCM9.4} follows from  Proposition \ref{DeC7},  the bounds \eqref{DCM2}, \eqref{gDCM2}, and Lemmas \ref{dint}, \ref{g:bdds}, \ref{gr:bdds}, and \ref{mGk50}. 
\end{proof}



\begin{lemma}\label{lem:gd2}
Assume that $\kappa k\in\Z\backslash\{0\}, \lambda\in \Sigma_{j, \delta_0} \setminus{ \{b(y_{j*}) \} }$, $\iota\in\{\pm\}, \epsilon \in (0, \epsilon_0)$. Then for any $\theta \in (0, 1/4]$ we have the following bound: 
\begin{equation}\label{DCM5}
\begin{split}
\norm{ \delta(\Lambda) (\p_\lambda \D \digamma_{k, \epsilon}^\iota   - \mathcal{S}_{1,1, \epsilon}^\iota) }_{X^{0,0}} \lesssim \frac{\big\|\omega_{0k}\big\|_{H^4_k}}{|k|^{3/2- \theta}},
\end{split}
\end{equation}
where 
\begin{equation}\label{DCM7.6}
\begin{split}
\mathcal{S}_{1,1, \epsilon}^\iota(y,\lambda) := -\chi_{j, \delta}^*(y) \D\left(  \digamma^\iota_{k,\epsilon}(y,\lambda)\frac{b''(y)}{(b'(y))^2}     \right)\log\frac{b(y)-\lambda+i\iota\epsilon}{\delta^2(\Lambda)}.
\end{split}
\end{equation}

\end{lemma}
\begin{remark}
    We have to use the $\delta(\Lambda)$ and the space $X^{0,0}$ instead of the potentially more natural space $X^{1,0}$ because of the appearance of $\p_\lambda^2 \digamma_{k, \epsilon}^\iota$ which carries a factor of $\delta(\Lambda)$, see \eqref{DCM9.4}.
\end{remark}

\begin{proof}
Equations \eqref{BDC2}, \eqref{E:eq}, and \eqref{rhs:F} yield
\begin{align}
&\left(
-\partial_y^2+k^2
+\frac{b''(y)}{b(y)-\lambda+i\iota \epsilon}
\right) \p_\lambda \D \digamma_{k, \epsilon}^\iota(y,\lambda)
\\
&\quad = -2b''(y)\p_y\p_\lambda^2 \digamma_{k, \epsilon}^\iota(y,\lambda )
  -b'''(y)\p_\lambda^2 \digamma_{k, \epsilon}^\iota(y,\lambda )
  -\frac{b'''(y)}{b(y)-\lambda+i \iota \epsilon}\p_\lambda \digamma_{k, \epsilon}^\iota(y,\lambda )\notag \\
&\quad-\frac{b''(y)}{\bigl(b(y)-\lambda+i \iota \epsilon\bigr)^2}\D \digamma_{k, \epsilon}^\iota(y,\lambda )
  -\frac{b'''(y)}{\bigl(b(y)-\lambda+i\iota \epsilon\bigr)^2}\digamma_{k, \epsilon}^\iota(y,\lambda )  +  \D\p_\lambda r. \notag
\end{align}
Denoting the right-hand side of the equation satisfied by $\partial_\lambda  \D \digamma_{k, \epsilon}^\iota(y,\lambda) - \mathcal{S}_{1,1, \epsilon}^\iota $ as $\mathcal{R}_{1,1, \epsilon}^\iota$, we solve for $\p_\lambda \D \digamma_{k, \epsilon}^\iota   -  \mathcal{S}_{1,1, \epsilon}^\iota$ as follows:
\begin{align}
    \p_\lambda \D \digamma_{k, \epsilon}^\iota   -  \mathcal{S}_{1,1, \epsilon}^\iota   + T_\Theta (\p_\lambda \D \digamma_{k, \epsilon}^\iota   -  \mathcal{S}_{1,1, \epsilon}^\iota) = \int_\T \mathcal{G}_k^j(y,z; \Lambda) \mathcal{R}_{1,1,\epsilon}^\iota(z,\lambda) dz .
\end{align}
After a direct computation to determine $\mathcal{R}_{1,1, \epsilon}^\iota$,
\eqref{DCM5} follows from  Proposition \ref{DeC7}, the estimates \eqref{DCM2}, \eqref{DCM9.4}, and Lemmas \ref{dint}, \ref{g:bdds}, \ref{gr:bdds}, and \ref{mGk50}. 
\end{proof}

\begin{lemma}\label{lem:gd4}
Assume that $\kappa k\in\Z\backslash\{0\}, \lambda\in \Sigma_{j, \delta_0} \setminus \{b(y_{j*}) \}$, $\iota\in\{\pm\}, \epsilon \in (0, \epsilon_0) $. Then for any $\theta \in (0, 1/4]$ we have the following bound:

\begin{subequations}
    \label{gDCM22}
    \begin{align}
\norm{ \delta(\Lambda) \D^2 \digamma_{k, \epsilon}^\iota }_{X^{0, -1}} \lesssim  \frac{\big\|\omega_{0k}\big\|_{H^4_k}}{|k|^{3/2- \theta}}, \quad  |k \delta(\Lambda)| \ll 1, \\
\norm{  \D^2 \digamma_{k, \epsilon}^\iota }_{X^{0, 0}} \lesssim  \frac{\big\|\omega_{0k}\big\|_{H^4_k}}{|k|^{3/2- \theta}}, \quad  |k \delta(\Lambda)| \gtrsim 1. 
\end{align}
\end{subequations}

\end{lemma}
\begin{remark}
    \label{rem:lamgd}
    The reason for the separate cases in \eqref{gDCM22} is related to how smallness manifests in the modified Green's function in the different cases. When $|k \delta(\Lambda)| \ll 1$, the modified Green's function only enforces localized smallness near the critical point which is why we use $X^{0, -1}$ and a $\delta(\Lambda)$ weight. However, when $|k \delta(\Lambda)| \gtrsim 1$, the modified Green's function is small everywhere. This can compensate for the global $\delta(\Lambda)$ factor on the right hand side of the equation and allows the use of $X^{0,0}$.  
\end{remark}

\begin{proof}
Equations \eqref{BDC2}, \eqref{E:eq}, and \eqref{rhs:F} yield
\begin{align}
\label{gd2F}
&\left(-\p_y^2 + k^2 + \frac{b''(y)}{b(y) - \lambda + i \iota \epsilon } \right) \D^2 \digamma^\iota (y,\lambda)\\
    &= -4b''(y)\p_y\p_\lambda\D \digamma^\iota (y,\lambda)
  -2b'''(y)\p_\lambda \D \digamma^\iota (y,\lambda)
  +2\bigl(b''(y)\bigr)^2\p_\lambda^2 \digamma^\iota (y,\lambda) \notag\\
&-2b'''(y)\p_y\p_\lambda \digamma^\iota (y,\lambda)-b''''(y)\p_\lambda \digamma^\iota (y,\lambda)
  -\frac{2b'''(y)}{b(y)-\lambda+i \iota \epsilon}\D \digamma^\iota (y,\lambda)
  -\frac{b''''(y)}{b(y)-\lambda+i\iota \epsilon}\digamma^\iota (y,\lambda)\notag \\
  &+ \D^2 r. \notag
\end{align}
Denoting the right-hand side  of \eqref{gd2F} as $\mathcal{R}_{0,2, \epsilon}^\iota$, we solve for $\D^2 \digamma_{k, \epsilon}^\iota$ as follows:
\begin{align}
  \D^2 \digamma_{k, \epsilon}^\iota   + T_\Theta (\D^2 \digamma_{k, \epsilon}^\iota) = \int_\T \mathcal{G}_k^j(y,z; \Lambda) \mathcal{R}_{0,2, \epsilon}^\iota(z,\lambda) dz .
\end{align}
The bound \eqref{gDCM22} follows from  Proposition \ref{DeC7}, and  Lemmas  \ref{lem:lam1} , \ref{lem:gd2}, \ref{dint}, \ref{g:bdds}, \ref{gr:bdds}, and \ref{mGk50}.
\end{proof}

\begin{lemma}\label{DCM103}
Assume that $\kappa k\in\Z\backslash\{0\}, \lambda\in \Sigma_{j, \delta_0} \setminus \{  b(y_{j*})   \}$, $\iota\in\{\pm\}, \epsilon \in (0, \epsilon_0)$. Then for any $\theta \in (0, 1/4]$ we have the following bound: 
\begin{equation}\label{DCM9.43}
\norm{ \delta^3 (\Lambda) (\p_\lambda^3 \digamma_{k, \epsilon}^\iota   - \mathcal{S}_{3,0, \epsilon}^\iota )}_{X^{1, 0}} \lesssim  \frac{\big\|\omega_{0k}\big\|_{H^4_k}}{|k|^{3/2- \theta}},
\end{equation}
where
\begin{align}
    \mathcal{S}_{3, 0 , \epsilon}^\iota(y,\lambda) &:=  \frac{\chi_{j, \delta}^*(y)C_1(y)\digamma_{k, \epsilon}^\iota(y,\lambda)}{ (b(y) - \lambda + i \iota  \epsilon)^2    } \\
    &+ \frac{\chi_{j, \delta}^*(y)(C_2(y) \p_y \digamma_{k, \epsilon}^\iota(y,\lambda)    + C_3(y) \D \digamma_{k, \epsilon}^\iota(y,\lambda) + C_4(y) \digamma_{k, \epsilon}^\iota(y,\lambda) )  }{b(y) - \lambda + i \iota \epsilon} \\
    & + \chi_{j, \delta}^*(y)(  C_5(y) \D^2 \digamma_{k, \epsilon}^\iota(y,\lambda)  +  C_6(y) \D \digamma_{k, \epsilon}^\iota(y,\lambda) \\
    &+  C_7(y)\digamma_{k, \epsilon}^\iota(y,\lambda) ) \log \frac{b(y)  -\lambda + i  \iota \epsilon  }{\delta^2(\Lambda)},
\end{align}
\begin{align}
&C_1(y) = \frac{b''(y)}{(b'(y))^2} , \quad C_2(y)  = - \frac{b''(y)}{(b'(y))^3}, \quad C_3(y) = \frac{3 b''(y)}{(b'(y))^3}, \quad C_4(y) =    \frac{2b'''(y)}{(b'(y))^3}
-\frac{3(b''(y))^2}{(b'(y))^4}, \\
 &C_5(y) = -\frac{3b''(y)}{(b'(y))^4}, \quad C_6(y) = -\frac{6b'''(y)}{(b'(y))^4}
+\frac{12(b''(y))^2}{(b'(y))^5}, \notag \\
&C_7(y) = -\frac{k^2b''(y)}{(b'(y))^4}
-\frac{3b^{(4)}(y)}{(b'(y))^4}
+\frac{16b''(y)b'''(y)}{(b'(y))^5}
-\frac{15(b''(y))^3}{(b'(y))^6}. \notag 
\end{align}


\end{lemma}

\begin{proof}
Equations \eqref{BDC2}, \eqref{E:eq}, and \eqref{rhs:F} yield
    \begin{align}
      &(-\partial_y^2 + k^2 + \frac{b''(y)}{b(y) - \lambda + i \iota  \epsilon }) \partial_\lambda^3 \digamma_{k, \epsilon}^\iota(y,\lambda)  \\
& = 
-\frac{3b''(y)\partial_\lambda^2  \digamma_{k, \epsilon}^\iota(y,\lambda)}
{(b(y)-\lambda+i \iota \epsilon)^2}
-\frac{6b''(y)\partial_\lambda  \digamma_{k, \epsilon}^\iota(y,\lambda)}
{(b(y)-\lambda+i \iota \epsilon)^3}
-\frac{6b''(y) \digamma_{k, \epsilon}^\iota(y,\lambda)}
{(b(y)-\lambda+i \iota \epsilon)^4}
+\partial_\lambda^3 r(y,\lambda).
\end{align}
Denoting the right-hand side of the equation satisfied by $\partial_\lambda^3  \digamma_{k, \epsilon}^\iota(y,\lambda) - \mathcal{S}_{3,0, \epsilon}^\iota $ as $\mathcal{R}_{3,0, \epsilon}^\iota$, we solve for $\p_\lambda^3 \digamma_{k, \epsilon}^\iota   -  \mathcal{S}_{3,0, \epsilon}^\iota$ as follows:
\begin{align}
    \p_\lambda^3 \digamma_{k, \epsilon}^\iota   -  \mathcal{S}_{3,0, \epsilon}^\iota   + T_\Theta (\p_\lambda^3 \digamma_{k, \epsilon}^\iota   -  \mathcal{S}_{3,0, \epsilon}^\iota) = \int_\T \mathcal{G}_k^j(y,z; \Lambda) \mathcal{R}_{3,0, \epsilon}^\iota(z,\lambda) dz .
\end{align}
After a direct computation to determine $\mathcal{R}_{3,0, \epsilon}^\iota$,
\eqref{DCM9.43} follows from  Proposition \ref{DeC7},  the bounds \eqref{DCM2}, \eqref{gDCM2}, and Lemmas \ref{dint}, \ref{g:bdds}, \ref{gr:bdds}, and \ref{mGk50}.

\end{proof}

\begin{lemma}\label{lem:gd3}
Assume that $\kappa k\in\Z\backslash\{0\}, \lambda\in \Sigma_{j, \delta_0} \setminus \{  b(y_{j*})   \}$, $\iota\in\{\pm\}, \epsilon \in (0, \epsilon_0)$. Then for any $\theta \in (0, 1/4]$ we have the following bound:
\begin{equation}\label{DCM35}
\begin{split}
\norm{ \delta^3(\Lambda) (\p_\lambda^2 \D \digamma_{k, \epsilon}^\iota   - \mathcal{S}_{2,1,\epsilon}^\iota) }_{X^{0,0}} \lesssim \frac{\big\|\omega_{0k}\big\|_{H^4_k}}{|k|^{3/2- \theta}},
\end{split}
\end{equation}
where 
\begin{equation}\label{DCM7.64}
\begin{split}
\mathcal{S}^\iota_{2,1,\epsilon}(y,\lambda)&:=
\frac{\chi_{j, \delta}^*(y) \D (C_1(y)\digamma_{k,\epsilon}^\iota(y,\lambda))  }{b(y) - \lambda + i \iota \epsilon}   \\
& + \chi_{j, \delta}^*(y)\left( \D(C_2(y) \D \digamma_{k,\epsilon}^\iota(y,\lambda)) + \D (C_3(y) \digamma_{k,\epsilon}^\iota(y,\lambda))   \right) \log (\frac{b(y) -\lambda  + i\iota \epsilon }{\delta^2(\Lambda)}),
\end{split}
\end{equation}
\begin{align}
    C_1 = \frac{b''(y)}{(b'(y))^2}, \quad C_2(y) = \frac{ -2 b''(y)}{(b'(y))^3}, \quad  C_3(y) = \frac{3 (b''(y))^2  - 2b'''(y) b'(y)  }{(b'(y))^4}.
\end{align}

\end{lemma}
\begin{proof}
Equations \eqref{BDC2}, \eqref{E:eq}, and \eqref{rhs:F} yield
    \begin{align}
&\left(
-\partial_y^2+k^2
+\frac{b''(y)}{b(y)-\lambda+i\iota \epsilon}
\right) \p_\lambda^2 \D \digamma_{k, \epsilon}^\iota(y,\lambda)
\nonumber\\
&\quad =-2b''(y)\p_y\p_\lambda^3 \digamma_{k, \epsilon}^\iota(y,\lambda)
  -b'''(y)\p_\lambda^3 \digamma_{k, \epsilon}^\iota(y,\lambda)
  -\frac{b'''(y)}{b(y)-\lambda+i\iota \epsilon}\p_\lambda^2 \digamma_{k, \epsilon}^\iota(y,\lambda)\\
&\quad-\frac{2b''(y)}{\bigl(b(y)-\lambda+i \iota \epsilon\bigr)^2}
  \p_\lambda\D \digamma_{k, \epsilon}^\iota(y,\lambda)
  -\frac{2b'''(y)}{\bigl(b(y)-\lambda+i \iota \epsilon\bigr)^2}\p_\lambda \digamma_{k, \epsilon}^\iota(y,\lambda)\\
&\quad-\frac{2b''(y)}{\bigl(b(y)-\lambda+i\iota \epsilon\bigr)^3}\D \digamma_{k, \epsilon}^\iota(y,\lambda)
  -\frac{2b'''(y)}{\bigl(b(y)-\lambda+i\iota\epsilon\bigr)^3} \digamma_{k, \epsilon}^\iota(y,\lambda) +  \p_\lambda^2\D r.
\end{align}
Denoting the right-hand side of the equation satisfied by $\partial_\lambda^2  \D \digamma_{k, \epsilon}^\iota(y,\lambda) - \mathcal{S}_{2,1, \epsilon}^\iota $ as $\mathcal{R}_{2,1, \epsilon}^\iota$, we solve for $\p_\lambda^2  \D \digamma_{k, \epsilon}^\iota   -  \mathcal{S}_{2,1, \epsilon}^\iota$ as follows:
\begin{align}
    \p_\lambda^2 \D \digamma_{k, \epsilon}^\iota   -  \mathcal{S}_{2,1, \epsilon}^\iota   + T_\Theta (\p_\lambda^2 \D \digamma_{k, \epsilon}^\iota   -  \mathcal{S}_{2,1, \epsilon}^\iota) = \int_\T \mathcal{G}_k^j(y,z; \Lambda) \mathcal{R}_{2,1, \epsilon}^\iota(z,\lambda) dz .
\end{align}
After a direct computation to determine $\mathcal{R}_{2,1, \epsilon}^\iota$,
\eqref{DCM35} follows from  Proposition \ref{DeC7}, the preceding estimates, and Lemmas \ref{dint}, \ref{g:bdds}, \ref{gr:bdds}, and \ref{mGk50}.
\end{proof}

\begin{lemma}\label{lem:gd35}
Assume that $\kappa k\in\Z\backslash\{0\}, \lambda\in \Sigma_{j, \delta_0} \setminus \{  b(y_{j*})   \}$, $\iota\in\{\pm\}, \epsilon \in (0, \epsilon_0)$. Then for any $\theta \in (0, 1/4]$ we have the following bound:
\begin{align}\label{DCM36}
\norm{ \delta^3(\Lambda) (\p_\lambda \D^2 \digamma_{k, \epsilon}^\iota   - \mathcal{S}_{1,2,\epsilon}^\iota) }_{X^{0,-1}} &\lesssim \frac{\big\|\omega_{0k}\big\|_{H^4_k}}{|k|^{3/2- \theta}}, \quad |k \delta(\Lambda)| \ll 1,   \\
\norm{ \delta^2(\Lambda) (\p_\lambda \D^2 \digamma_{k, \epsilon}^\iota   - \mathcal{S}_{1,2,\epsilon}^\iota) }_{X^{0,0}} &\lesssim \frac{\big\|\omega_{0k}\big\|_{H^4_k}}{|k|^{3/2- \theta}}, \quad |k \delta(\Lambda)| \gtrsim 1 ,
\end{align}
where 
\begin{align}\label{DCM7.77}
&\mathcal{S}^\iota_{1,2,\epsilon}(y,\lambda):= \\
&\chi_{j, \delta}^*(y) ( C_1(y) \D^2 \digamma_{k, \epsilon}^\iota(y,\lambda )  +  C_2(y) \D \digamma_{k, \epsilon}^\iota(y,\lambda )    + C_3(y) \digamma_{k, \epsilon}^\iota(y,\lambda )       ) \log \left( \frac{b(y) -\lambda + i \iota \epsilon }{\delta^2(\Lambda)} \right),
\end{align}
\begin{align}
C_1(y)=&-\frac{b''(y)}{(b'(y))^{2}}, \quad 
C_2(y)=-\frac{2b'''(y)}{(b'(y))^{2}}+\frac{4(b''(y))^{2}}{(b'(y))^{3}}, \notag \\
C_3(y)=&-\frac{b''''(y)}{(b'(y))^{2}}+\frac{6b''(y)b'''(y)}{(b'(y))^{3}}-\frac{6(b''(y))^{3}}{(b'(y))^{4}}.
\end{align}

\end{lemma}
See Remark \ref{rem:lamgd} for a discussion of why two cases are considered here. 

\begin{proof}
Equations \eqref{BDC2}, \eqref{E:eq}, and \eqref{rhs:F} yield
\begin{align}
   &\left(
-\partial_y^2+k^2
+\frac{b''(y)}{b(y)-\lambda+i\iota \epsilon}
\right)  \p_\lambda\D^2 \digamma_{k, \epsilon}^\iota(y,\lambda) \\
&=-4b''(y)\p_y\p_\lambda^2(\D \digamma_{k, \epsilon}^\iota(y,\lambda))
  -2b'''(y)\p_\lambda^2(\D \digamma_{k, \epsilon}^\iota(y,\lambda))
  +2\bigl(b''(y)\bigr)^2\p_\lambda^3\digamma_{k, \epsilon}^\iota(y,\lambda)\notag \\
&-2b'''(y)\p_y\p_\lambda^2\digamma_{k, \epsilon}^\iota(y,\lambda)-b''''(y)\p_\lambda^2\digamma_{k, \epsilon}^\iota(y,\lambda)
  -\frac{2b'''(y)}{b(y)-\lambda+i\iota\epsilon}\p_\lambda\D \digamma_{k, \epsilon}^\iota(y,\lambda) \notag \\
  &-\frac{b''''(y)}{b(y)-\lambda+i\iota \epsilon}\p_\lambda \digamma_{k, \epsilon}^\iota(y,\lambda)
-\frac{b''(y)}{\bigl(b(y)-\lambda+i\iota \epsilon\bigr)^2}\D^2\digamma_{k, \epsilon}^\iota(y,\lambda)
  -\frac{2b'''(y)}{\bigl(b(y)-\lambda+i\iota \epsilon\bigr)^2}\D \digamma_{k, \epsilon}^\iota(y,\lambda) \notag \\ 
  &-\frac{b''''(y)}{\bigl(b(y)-\lambda+i\iota \epsilon\bigr)^2}\digamma_{k, \epsilon}^\iota(y,\lambda) + \D^2\p_\lambda r. \notag
\end{align}
Denoting the right-hand side of the equation satisfied by $\partial_\lambda  \D^2 \digamma_{k, \epsilon}^\iota(y,\lambda) - \mathcal{S}_{1,2, \epsilon}^\iota $ as $\mathcal{R}_{1,2, \epsilon}^\iota$, we solve for $\p_\lambda  \D^2 \digamma_{k, \epsilon}^\iota   -  \mathcal{S}_{1,2, \epsilon}^\iota$ as follows:
\begin{align}
    \p_\lambda \D^2 \digamma_{k, \epsilon}^\iota   -  \mathcal{S}_{1,2, \epsilon}^\iota   + T_\Theta (\p_\lambda \D^2 \digamma_{k, \epsilon}^\iota   -  \mathcal{S}_{1,2, \epsilon}^\iota) = \int_\T \mathcal{G}_k^j(y,z; \Lambda) \mathcal{R}_{1,2, \epsilon}^\iota(z,\lambda) dz .
\end{align}
After a direct computation to determine $\mathcal{R}_{1,2, \epsilon}^\iota$,
\eqref{DCM36} follows from  Proposition \ref{DeC7}, the preceding estimates, and Lemmas \ref{dint}, \ref{g:bdds}, \ref{gr:bdds}, and \ref{mGk50}.
\end{proof}

As a final step, we need control of the spectral density function at the critical values. These appear as boundary terms when performing integration-by-parts in the representation formula \eqref{F3.4}. We omit the proofs as they are consequences of the equations satisfied by $\p_\lambda^\beta(\phi_{k, \epsilon}^+(y,\lambda)), \beta \in \{0,1\}$, Lemma \ref{mGk50}, and Proposition \ref{DeC7}.

\begin{lemma}\label{boundary}
For $\kappa k \in \mathbb{Z} \setminus \{0 \}$,  $\lambda=b(y_{j*})$, $0<\epsilon<\epsilon_0$, and $\theta\in(0,1/4]$, we have the following bound:
\begin{equation}\label{bdphi}
    \norm{\phi_{k,\epsilon}^{+}-\phi_{k,\epsilon}^{-}}_{X^{\theta,0}}
    \lesssim \delta^2(\Lambda)\big\|\omega_{0k}\big\|_{H_k^4}.
\end{equation}
In particular, 
\begin{equation}\label{bdphi0}
    \lim_{\epsilon\to0+}
    \left[\phi_{k,\epsilon}^{+}(y,b(y_{j*}))
    -\phi_{k,\epsilon}^{-}(y,b(y_{j*}))\right]=0,
    \qquad y\in\T.
\end{equation}
\end{lemma}

\begin{lemma}\label{boundary1}
For $\kappa k \in \mathbb{Z}\setminus \{0 \}$,  $\lambda=b(y_{j*})$, $0<\epsilon<\epsilon_0$, and $\theta\in(0,1/4]$, we have the following bound:
\begin{equation}\label{bdlam}
    \norm{\p_\lambda\phi_{k,\epsilon}^{+}-\p_\lambda\phi_{k,\epsilon}^{-}}_{X^{1,0}}
    \lesssim \delta^{1-\theta}(\Lambda)\big\|\omega_{0k}\big\|_{H_k^4}.
\end{equation}
In particular,
\begin{equation}\label{bdlam0}
    \lim_{\epsilon\to0+}
    \left[\p_\lambda\phi_{k,\epsilon}^{+}(y,b(y_{j*}))
    -\p_\lambda\phi_{k,\epsilon}^{-}(y,b(y_{j*}))\right]=0,
    \qquad y\in\T\backslash\{y_{j*}\},
\end{equation}
where the convergence is locally uniform in $\T\backslash\{y_{j*}\}$.
\end{lemma}

\section{Proof of Main Results}
\label{sec:main:results}
In this section we give a proof of Theorems \ref{thm1} and \ref{thm2}. Theorem \ref{thm1} follows from Theorem \ref{thm2} after summing in $k$ so we focus on the proof of the bounds in \eqref{main:bdd}. We follow the general approach from \cite{IIJ24}. 
By standard well-posedness, it suffices to consider $t\geq1$. In the remainder of the section we assume that $b(y) \neq b(y_{1*}), b(y_{2*})$ and let $M_k:= \norm{\omega_{0k}}_{H_k^4}$.

Before proving the main results, we state some useful steps that will be used frequently in the following proofs. To handle quantities of the form
\begin{align}
    \frac{f(y,\lambda)}{b(y)  -\lambda + i \epsilon},
\end{align}
we will make use of the identity 
\begin{align}
    f(y,\lambda) = f(y, b(y)) - (b(y)  -\lambda)\int_0^1 \p_\lambda f(y, \lambda + s(b(y) - \lambda)) ds,
\end{align}
and the bound
\begin{align}
\left| \lim_{\epsilon \to 0} \int_{\R} \frac{e^{i \lambda k t}  
\chi_l(\lambda -c)}{b(y) - \lambda + i \epsilon} d \lambda \right| \lesssim 1, 
\end{align}
which is \textit{independent} of $l > 0, c$. This is a consequence of the fact that the Fourier transform of $\text{p.v.} \frac{1}{\lambda}$ is proportional to $\text{sgn}(\lambda)$, Parseval's identity, and the smoothness of $\chi$. We will also  make use of the bound  
\begin{align}
    \int_{I_{bj}} |\chi_{j, \delta}^*(y) \frac{f(y,\lambda)}{|b'(y)|^2} \log \frac{b(y) - \lambda + i \epsilon}{\delta^2(\Lambda)}| d \lambda  \lesssim \norm{f(y, \cdot )}_{L_{\lambda}^\infty(I_{bj})},
\end{align}
where $I_{bj}$ is the set $ |b(y) - \lambda| \leq c|b(y) - b(y_{j*})|$ and $c \in (0,1)$ is a sufficiently small constant depending on $b$. Finally, for $\lambda \in \Sigma_{j, \delta_0}$ and $\lambda\in I_{bj}$ we record the bound 
\begin{align}
\label{prelim:pointwise}
    |f(y, \lambda)| \lesssim 
    \norm{f}_{X^{\sigma_1, \sigma_2}}
    \begin{cases}
        (\delta(\Lambda))^{-\sigma_1 - \sigma_2},  &  |\delta(\Lambda) k | \ll 1\\
        |k \delta(\Lambda)|^{1/2} (\delta(\Lambda))^{-\sigma_1} |k|^{\sigma_2}, & |k \delta(\Lambda)| \gtrsim 1 
    \end{cases}.
\end{align}
The factor of $|k \delta(\Lambda)|^{1/2}$ is the reason for the extra factor of $|k|^{1/2}$ relative to what dimension counting would suggest.

\subsection{Proof of \eqref{damp1}}
Using the representation formula \eqref{F3.4}, we have
\begin{equation}\label{PMR1}
\begin{split}
\psi_k(t,y)&=\frac{1}{2\pi i}\lim_{\epsilon\to0+}\int_{\Sigma}e^{-ik\lambda t}\Big[\psi_{k,\epsilon}^+(y,\lambda)-\psi^-_{k,\epsilon}(y,\lambda)\Big]\,d\lambda.
\end{split}
\end{equation}
 We can decompose for $t\ge1$,
\begin{equation}\label{PMR2}
\psi_k(t,y):=\psi_{k, \rm{nd}}(t,y)+\psi_{k, \rm{d}}(t,y),
\end{equation}
where
\begin{align}
 \psi_{k, \rm{nd}}   &=  \frac{1}{2\pi i}\lim_{\epsilon\to0+}\int_{\Sigma} (1 - \sum_j\chi_{bj,c\delta_0^2}(\lambda ))e^{-ik\lambda t}\Big[\psi_{k,\epsilon}^+(y,\lambda)-\psi^-_{k,\epsilon}(y,\lambda)\Big]\,d\lambda, \\
 \psi_{k, \rm{d}}   &=  \sum_{j} \frac{1}{2\pi i}\lim_{\epsilon\to0+}\int_{\Sigma} \chi_{bj,c\delta_0^2}(\lambda )e^{-ik\lambda t}\Big[\psi_{k,\epsilon}^+(y,\lambda)-\psi^-_{k,\epsilon}(y,\lambda)\Big]\,d\lambda .
\end{align}
\par
We first estimate $\psi_{k,\rm d}$. By \eqref{BDC1}, the difference of the two branches is unchanged when $\psi_{k,\epsilon}^\iota$ is replaced by $\phi_{k,\epsilon}^\iota$. Integrating by parts twice and using Lemmas~\ref{boundary} and \ref{boundary1}, we obtain
\begin{align}
\psi_{k,\rm d}(t,y)
&=-\frac{1}{2\pi i(kt)^2}\sum_j\lim_{\epsilon\to0+}
\int_\Sigma e^{-ik\lambda t}\p_\lambda^2
\Big[\chi_{bj,c\delta_0^2}(\lambda)
(\phi_{k,\epsilon}^+-\phi_{k,\epsilon}^-)(y,\lambda)\Big]\,d\lambda \\
&=-\frac{1}{2\pi i(kt)^2}\sum_j\lim_{\epsilon\to0+}
\int_\Sigma e^{-ik\lambda t}\chi_{bj,c\delta_0^2}(\lambda)
\Big[\p_\lambda^2\phi_{k,\epsilon}^+(y,\lambda)
-\p_\lambda^2\phi_{k,\epsilon}^-(y,\lambda)\Big]\,d\lambda
+{\rm ND}(t,y). \notag
\end{align}
Here ${\rm ND}$ contains the terms with derivatives on the fixed spectral cutoffs; the non-degenerate estimates give
$|{\rm ND}(t,y)|\lesssim M_k/(|k|^{3-\theta}t^2)$.
\par
We next observe that by \eqref{BDC3.5}, \eqref{phi:lam}, and Lemma~\ref{phi:van} we have the following essential vanishing property:
\begin{align}
\label{vanishing:rep}
    \lim_{\epsilon\to0+}\int_{\Sigma}
    \left|\chi_{bj,c\delta_0^2}(\lambda)g_{k,j*}\chi_{j,\frac{\delta_0}{4k}}(y)\right|
    \left|\p_\lambda^2\Phi_\epsilon^+(y-y_{j*},\lambda)
    -\p_\lambda^2\Phi_\epsilon^-(y-y_{j*},\lambda)\right|\,d\lambda=0.
\end{align}
The same conclusion holds with zero or one spectral derivative.
Therefore,
\begin{align}
    &\lim_{\epsilon\to0+}\int_{\Sigma} \chi_{bj,c\delta_0^2}(\lambda )e^{-ik\lambda t}\Big[\p_\lambda^2\phi_{k,\epsilon}^+(y,\lambda)-\p_\lambda^2\phi^-_{k,\epsilon}(y,\lambda)\Big]\,d\lambda \\
    &= \lim_{\epsilon\to0+}\int_{\Sigma} \chi_{bj,c\delta_0^2}(\lambda )e^{-ik\lambda t}\Big[\p_\lambda^2\digamma_{k,\epsilon}^+(y,\lambda)- \p_\lambda^2\digamma^-_{k,\epsilon}(y,\lambda)\Big]\,d\lambda.
\end{align}
 Here the derivatives of $\digamma_{k,\epsilon}^\iota$ are read as in \eqref{sdf:cases}. 

We estimate each $j$-summand separately. First suppose that $|b(y)-b(y_{j*})||kt|>1$. Choose $C>1$ sufficiently large. The preliminary estimates and Lemmas~\ref{DCM1}--\ref{DCM100} give
\begin{equation}
\begin{aligned}
&\Bigl|\lim_{\epsilon\to0+}\int_\Sigma e^{-ik\lambda t}
\chi_{bj,c\delta_0^2}(\lambda)
\chi_{bj,C|b(y)-b(y_{j*})|}(\lambda)\\
&\qquad\qquad\times
\Big[\p_\lambda^2\digamma_{k,\epsilon}^+(y,\lambda)
-\p_\lambda^2\digamma_{k,\epsilon}^-(y,\lambda)\Big]\,d\lambda\Bigr|
\lesssim\frac{M_k}{|k|^{1-\theta}}.
\end{aligned}
\end{equation}

On the support of $1-\chi_{bj,C|b(y)-b(y_{j*})|}$, it appears as if there is a logarithmic loss since $\p_\lambda^2 \digamma_{k, \epsilon}^\iota(y,\lambda) \approx \frac{1}{|b(y_{j*})-\lambda|}$. However, we observe that the annular factor $\chi_{j,\delta}^*(y)$ vanishes. This allows us to do one more integration by parts. Then, Lemmas~\ref{DCM100} and \ref{DCM103} yield
\begin{equation}
\begin{aligned}
&\Bigl|\lim_{\epsilon\to0+}\int_\Sigma e^{-ik\lambda t}
\chi_{bj,c\delta_0^2}(\lambda)
\bigl(1-\chi_{bj,C|b(y)-b(y_{j*})|}(\lambda)\bigr)\\
&\qquad\qquad\times
\Big[\p_\lambda^2\digamma_{k,\epsilon}^+(y,\lambda)
-\p_\lambda^2\digamma_{k,\epsilon}^-(y,\lambda)\Big]\,d\lambda\Bigr|\\
&\qquad\lesssim
\frac{M_k}{|k|^{1-\theta}|kt|\,|b(y)-b(y_{j*})|}
\lesssim\frac{M_k}{|k|^{1-\theta}}.
\end{aligned}
\end{equation}
Together with the prefactor $(kt)^{-2}$, this gives the required bound for these summands.

Now suppose that $|b(y)-b(y_{j*})||kt|\leq1$. Using the inner cancellation with zero derivatives, we split the original $j$-summand with the cutoff $\chi_{bj,C(|k|t)^{-1}}$. One integration by parts, Lemma~\ref{boundary}, and Lemmas~\ref{DCM1} and \ref{lem:lam1} give
\begin{equation}
\begin{aligned}
&\Bigl|\lim_{\epsilon\to0+}\int_\Sigma e^{-ik\lambda t}
\chi_{bj,c\delta_0^2}(\lambda)
\chi_{bj,C(|k|t)^{-1}}(\lambda)\\
&\qquad\qquad\times
\Big[\digamma_{k,\epsilon}^+(y,\lambda)
-\digamma_{k,\epsilon}^-(y,\lambda)\Big]\,d\lambda\Bigr|
\lesssim\frac{M_k}{|k|^{1-\theta}k^2t^2}.
\end{aligned}
\end{equation}
On the support  of $\bigl(1-\chi_{bj,C(|k|t)^{-1}}(\lambda)\bigr)$, $\chi_{j,\delta}^*(y)=0$. Three integrations by parts and Lemmas~\ref{DCM1}, \ref{lem:lam1}, \ref{DCM100}, and \ref{DCM103} give
\begin{equation}
\begin{aligned}
&\Bigl|\lim_{\epsilon\to0+}\int_\Sigma e^{-ik\lambda t}
\chi_{bj,c\delta_0^2}(\lambda)
\bigl(1-\chi_{bj,C(|k|t)^{-1}}(\lambda)\bigr)\\
&\qquad\qquad\times
\Big[\digamma_{k,\epsilon}^+(y,\lambda)
-\digamma_{k,\epsilon}^-(y,\lambda)\Big]\,d\lambda\Bigr|
\lesssim\frac{M_k}{|k|^{1-\theta}k^2t^2}.
\end{aligned}
\end{equation}
Finally, two integrations by parts and the non-degenerate estimates in Lemma~\ref{ndpsi}, \eqref{ndpsi_l1}, and Lemma~\ref{ndpsi_l2}, with a smaller fixed neighborhood of the critical values, give
$|\psi_{k,\rm nd}(t,y)|\lesssim M_k/(|k|^{3-\theta}t^2)$.
Thus $|\psi_k(t,y)|\lesssim M_k/(|k|^{3-\theta}t^2)$ for $t\geq1$. By continuity in $y$, the bound also holds at the critical points. This completes the proof of \eqref{damp1}.
\par




\subsection{Proof of \eqref{damp2}}

Using \eqref{damp1} and \eqref{pro2}, which is proved independently below, we obtain
\begin{align}
|\p_y\psi_k(t,y)|
&\leq |\p_y(e^{iktb}\psi_k)(t,y)|+|ktb'(y)|\,|\psi_k(t,y)| \notag\\
&\lesssim \left(\frac{1}{|k|^{5/2-\theta}t^{3/2}}
+\frac{|b'(y)|}{|k|^{2-\theta}t}\right)M_k.
\end{align}
This proves \eqref{damp2}.

\subsection{Proof of \eqref{depl1}}

Applying $\p_y^2-k^2$ to \eqref{F3.4} and using \eqref{F7} gives
\begin{align}\label{PMR9}
\omega_k(t,y)
=\frac{1}{2\pi i}\lim_{\epsilon\to0+}\int_\Sigma e^{-ik\lambda t}
\left[
\frac{b''(y)\psi_{k,\epsilon}^+(y,\lambda)-\omega_{0k}(y)}
{b(y)-\lambda+i\epsilon}
-\frac{b''(y)\psi_{k,\epsilon}^-(y,\lambda)-\omega_{0k}(y)}
{b(y)-\lambda-i\epsilon}
\right]\,d\lambda.
\end{align}
We use the same cutoffs as in \eqref{PMR2}. The non-degenerate estimates give
$|(\p_y^2-k^2)\psi_{k,\rm nd}|\lesssim |k|^{-3+\theta}M_k$.
Near $y_{j*}$, this term and the summand corresponding to the other critical value are bounded by $M_k/(|k|^{2-\theta}t)$, by one integration by parts; their denominators are separated from zero.

By \eqref{BDC1}, the correction to the $j$-summand has zero limit, since
$\chi_{bj,c\delta_0^2}(b(y))(1-\chi_{j,\delta_0}(y))=0$.
The inner contribution in \eqref{BDC3.5} also has zero limit: subtract its numerator at $\lambda=b(y)$ and apply \eqref{phi:lam} and Lemma~\ref{phi:van}.
Thus the $j$-summand equals
\begin{align}\label{PMR10}
\frac{b''(y)}{2\pi i}\lim_{\epsilon\to0+}
\int_\Sigma e^{-ik\lambda t}\chi_{bj,c\delta_0^2}(\lambda)
\left[
\frac{\digamma_{k,\epsilon}^+(y,\lambda)}{b(y)-\lambda+i\epsilon}
-\frac{\digamma_{k,\epsilon}^-(y,\lambda)}{b(y)-\lambda-i\epsilon}
\right]\,d\lambda.
\end{align}
As before, spectral derivatives are taken on the $\phi$ branches before removing the inner terms.

Fix $y$ near $y_{j*}$ and choose $C$ sufficiently large. On $I_{bj}$, first write the branch difference over $b(y)-\lambda+i\epsilon$; the remaining kernel jump samples the lower branch at $\lambda=b(y)$. Subtract the numerator at $\lambda=b(y)$ and use the preliminary Cauchy and logarithmic bounds, together with Lemmas~\ref{DCM1} and \ref{lem:lam1}. Outside $I_{bj}$ the denominator is separated from zero. We obtain
\begin{align}\label{PMR4}
&\left|\lim_{\epsilon\to0+}\int_\Sigma e^{-ik\lambda t}
\chi_{bj,c\delta_0^2}(\lambda)\chi_{bj,C|b(y)-b(y_{j*})|}(\lambda)
\left[
\frac{\digamma_{k,\epsilon}^+(y,\lambda)}{b(y)-\lambda+i\epsilon}
-\frac{\digamma_{k,\epsilon}^-(y,\lambda)}{b(y)-\lambda-i\epsilon}
\right]\,d\lambda\right| \notag\\
&\hspace{25mm}\lesssim \frac{M_k}{|k|^{1-\theta}}|b(y)-b(y_{j*})|.
\end{align}
On the support of $1-\chi_{bj,C|b(y)-b(y_{j*})|}$, we have
$|b(y)-\lambda|\approx|\lambda-b(y_{j*})|$ and $\chi_{j,\delta}^*(y)=0$.
Lemmas~\ref{DCM1}, \ref{lem:lam1}, and \ref{DCM100} therefore give, for $\alpha\in\{0,1,2\}$ and sufficiently small $\epsilon$,
\begin{align}\label{PMR11}
\left|\p_\lambda^\alpha
\left(\frac{\digamma_{k,\epsilon}^\iota(y,\lambda)}
{b(y)-\lambda+i\iota\epsilon}\right)\right|
\lesssim \frac{M_k}{|k|^{1-\theta}|\lambda-b(y_{j*})|^\alpha}.
\end{align}

First suppose that $|b(y)-b(y_{j*})||kt|>1$.
Split \eqref{PMR10} with $\chi_{bj,C|b(y)-b(y_{j*})|}$.
The first part is bounded by \eqref{PMR4}. On the complementary support, two integrations by parts and \eqref{PMR11}, including the derivatives of the cutoffs, give the bound
\begin{align}
\frac{M_k}{|k|^{1-\theta}|kt|^2}
\int_{|b(y)-b(y_{j*})|}^\infty s^{-2}\,ds
\lesssim \frac{M_k}{|k|^{1-\theta}|kt|^2|b(y)-b(y_{j*})|}
\lesssim \frac{M_k|b(y)-b(y_{j*})|}{|k|^{1-\theta}}.
\end{align}

Now suppose that $|b(y)-b(y_{j*})||kt|\leq1$.
Split \eqref{PMR10} with $\chi_{bj,C(|k|t)^{-1}}$.
Within this cutoff, \eqref{PMR4} handles the part with
$\chi_{bj,C|b(y)-b(y_{j*})|}$.
The difference
$\chi_{bj,C(|k|t)^{-1}}-\chi_{bj,C|b(y)-b(y_{j*})|}$
is supported away from the critical layer. Its spectral interval has length
$\lesssim(|k|t)^{-1}$, so \eqref{PMR11} with $\alpha=0$ gives
$M_k/(|k|^{1-\theta}|kt|)$.
On the support of $1-\chi_{bj,C(|k|t)^{-1}}$, two integrations by parts give
\begin{align}
\frac{M_k}{|k|^{1-\theta}|kt|^2}
\int_{(|k|t)^{-1}}^\infty s^{-2}\,ds
\lesssim \frac{M_k}{|k|^{2-\theta}t}.
\end{align}
Together with $|b(y)-b(y_{j*})|\approx |b'(y)|^2$, the two cases prove
\begin{align}
|\omega_k(t,y)|\lesssim
\left(\frac{1}{|k|^{2-\theta}t}
+\frac{|b'(y)|^2}{|k|^{1-\theta}}\right)M_k
\end{align}
near the critical points.

For the uniform bound, apply Lemmas~\ref{DCM1} and \ref{lem:lam1} with $\theta/2$.
In \eqref{PMR10}, use the preliminary Cauchy bound on
$|\lambda-b(y)|\lesssim c\min\{|b(y)-b(y_{j*})|,|k|^{-2}\}$.
On the rest of $|\lambda-b(y_{j*})|\approx |b(y)-b(y_{j*})|$, direct integration costs at most $1+\log(2+|k|)$.
The region $|\lambda-b(y_{j*})|\ll |b(y)-b(y_{j*})|$ is bounded directly by Lemma~\ref{DCM1}.
For $|\lambda-b(y_{j*})|\gg |b(y)-b(y_{j*})|$, the quotient is bounded by $M_k/|k|^{1-\theta/2}$ when $|\lambda-b(y_{j*})|\lesssim |k|^{-2}$; beyond this scale, retain the factor $|k\delta(\Lambda)|^{1/2}$ in the preliminary pointwise estimate to obtain
\begin{align}
\left|\frac{\digamma_{k,\epsilon}^\iota(y,\lambda)}
{b(y)-\lambda+i\iota\epsilon}\right|
\lesssim \frac{M_k}{|k|^{3-\theta/2}|\lambda-b(y_{j*})|^{3/4}}.
\end{align}
This is integrable in $\lambda$. Consequently, \eqref{PMR9} gives
\begin{align}
|\omega_k(t,y)|
\lesssim \frac{M_k(1+\log(2+|k|))}{|k|^{3-\theta/2}}
\lesssim \frac{M_k}{|k|^{3-\theta}}.
\end{align}
Away from the critical points this also implies the first bound in \eqref{depl1}.
Combining the bounds and extending by continuity proves \eqref{depl1}.

\subsection{Proof of \eqref{pro1}}

Differentiate the profile in \eqref{PMR9} and integrate the derivative of the phase by parts. Since $\D(b(y)-\lambda+i\iota\epsilon)^{-1}=0$, we obtain
\begin{align}
\p_y(e^{iktb}\omega_k)(t,y)
&=\frac{1}{2\pi i}\lim_{\epsilon\to0+}\int_\Sigma e^{-ik(\lambda-b(y))t}
\Biggl[
\frac{\D\bigl(b''(y)\psi_{k,\epsilon}^+(y,\lambda)-\omega_{0k}(y)\bigr)}
{b(y)-\lambda+i\epsilon}
\notag\\
&\hspace{36mm}
-\frac{\D\bigl(b''(y)\psi_{k,\epsilon}^-(y,\lambda)-\omega_{0k}(y)\bigr)}
{b(y)-\lambda-i\epsilon}
\Biggr]\,d\lambda.
\end{align}
The endpoint terms vanish by Lemma~\ref{boundary}.
We now use the cutoffs in \eqref{PMR2}. The non-degenerate contribution is bounded by $M_k/|k|^{2-\theta}$.
Near $y_{j*}$, this contribution and the summand at the other critical value are bounded by $M_k/(|k|^{2-\theta}t)$, by one integration by parts and Lemmas~\ref{boundary} and \ref{boundary1}.

By \eqref{BDC1}, \eqref{BDC3.5}, and Lemma~\ref{phi:van}, the correction and inner terms vanish in the limit, including their good derivatives.
The terms containing $b'''\digamma_{k,\epsilon}^\iota$ are bounded by the proof of \eqref{depl1}.
It remains to estimate
\begin{align}\label{PMR3}
\frac{b''(y)}{2\pi i}\lim_{\epsilon\to0+}\int_\Sigma
e^{-ik(\lambda-b(y))t}\chi_{bj,c\delta_0^2}(\lambda)
\left[
\frac{\D\digamma_{k,\epsilon}^+(y,\lambda)}{b(y)-\lambda+i\epsilon}
-\frac{\D\digamma_{k,\epsilon}^-(y,\lambda)}{b(y)-\lambda-i\epsilon}
\right]\,d\lambda.
\end{align}
All spectral derivatives are taken on the $\phi$ branches before removing the inner terms.

Fix $y$ near $y_{j*}$ and choose $C$ sufficiently large.
On $I_{bj}$, subtract the numerator at $\lambda=b(y)$ and use the preliminary Cauchy and logarithmic bounds, now with Lemmas~\ref{DCM1}, \ref{lem:gd1}, and \ref{lem:gd2}.
Outside $I_{bj}$ the denominator is separated from zero. Thus
\begin{equation}\label{PMR6}
\begin{aligned}
&\left|\lim_{\epsilon\to0+}\int_\Sigma
e^{-ik(\lambda-b(y))t}
\chi_{bj,c\delta_0^2}(\lambda)\chi_{bj,C|b(y)-b(y_{j*})|}(\lambda)
\left[
\frac{\D\digamma_{k,\epsilon}^+(y,\lambda)}{b(y)-\lambda+i\epsilon}
-\frac{\D\digamma_{k,\epsilon}^-(y,\lambda)}{b(y)-\lambda-i\epsilon}
\right]\,d\lambda\right|\\
&\hspace{25mm}\lesssim \frac{M_k}{|k|^{1-\theta}}|b(y)-b(y_{j*})|^{1/2}.
\end{aligned}
\end{equation}
On the support of $1-\chi_{bj,C|b(y)-b(y_{j*})|}$, the annular factor vanishes and
$|b(y)-\lambda|\approx|\lambda-b(y_{j*})|$.
Lemmas~\ref{lem:gd1} and \ref{lem:gd2} give, for $\alpha\in\{0,1\}$ and sufficiently small $\epsilon$,
\begin{align}\label{PMR5}
\left|\p_\lambda^\alpha
\left(\frac{\D\digamma_{k,\epsilon}^\iota(y,\lambda)}
{b(y)-\lambda+i\iota\epsilon}\right)\right|
\lesssim
\frac{M_k}{|k|^{1-\theta}|\lambda-b(y_{j*})|^{1/2+\alpha}}.
\end{align}

First suppose that $|b(y)-b(y_{j*})||kt|\leq1$.
Split \eqref{PMR3} with $\chi_{bj,C(|k|t)^{-1}}$.
Within this cutoff, \eqref{PMR6} handles the part with
$\chi_{bj,C|b(y)-b(y_{j*})|}$.
For the difference of these two cutoffs, \eqref{PMR5} with $\alpha=0$ gives
\begin{align}
\frac{M_k}{|k|^{1-\theta}}
\int_{0<|\lambda-b(y_{j*})|\lesssim(|k|t)^{-1}}
\frac{d\lambda}{|\lambda-b(y_{j*})|^{1/2}}
\lesssim \frac{M_k}{|k|^{3/2-\theta}t^{1/2}}.
\end{align}
On the support of $1-\chi_{bj,C(|k|t)^{-1}}$, one integration by parts, including derivatives of the cutoffs, gives
\begin{align}
\frac{M_k}{|k|^{1-\theta}|kt|}
\int_{|\lambda-b(y_{j*})|\gtrsim(|k|t)^{-1}}
\frac{d\lambda}{|\lambda-b(y_{j*})|^{3/2}}
\lesssim \frac{M_k}{|k|^{3/2-\theta}t^{1/2}}.
\end{align}

Now suppose that $|b(y)-b(y_{j*})||kt|>1$.
Split \eqref{PMR3} with $\chi_{bj,C|b(y)-b(y_{j*})|}$.
The first part is bounded by \eqref{PMR6}.
On the complementary support, the same integration by parts starts at
$|\lambda-b(y_{j*})|\gtrsim |b(y)-b(y_{j*})|$ and gives
\begin{align}
\frac{M_k}{|k|^{1-\theta}|kt|\,|b(y)-b(y_{j*})|^{1/2}}
\lesssim \frac{M_k|b(y)-b(y_{j*})|^{1/2}}{|k|^{1-\theta}}.
\end{align}
Since $|b(y)-b(y_{j*})|^{1/2}\approx |b'(y)|$, the two cases give the first bound in \eqref{pro1} near the critical points.

For the uniform bound, apply Lemmas~\ref{DCM1}, \ref{lem:gd1}, and \ref{lem:gd2} with $\theta/2$ and use the same regions as at the end of the proof of \eqref{depl1}.
Near $\lambda=b(y)$, use the preliminary Cauchy bound at scale
$c\min\{|b(y)-b(y_{j*})|,|k|^{-2}\}$.
The remaining region with $|\lambda-b(y_{j*})|\approx |b(y)-b(y_{j*})|$ costs at most $1+\log(2+|k|)$, while $|\lambda-b(y_{j*})|\ll |b(y)-b(y_{j*})|$ is bounded directly by Lemma~\ref{lem:gd1}.
For $|\lambda-b(y_{j*})|\gg |b(y)-b(y_{j*})|$, the preliminary pointwise estimate gives
\begin{align}
\left|\frac{\D\digamma_{k,\epsilon}^\iota(y,\lambda)}
{b(y)-\lambda+i\iota\epsilon}\right|
\lesssim \frac{M_k}{|k|^{1-\theta/2}}
\begin{cases}
|\lambda-b(y_{j*})|^{-1/2},
&|\lambda-b(y_{j*})|\lesssim |k|^{-2},\\
|k|^{-1}|\lambda-b(y_{j*})|^{-3/4},
&|\lambda-b(y_{j*})|\gtrsim |k|^{-2}
\end{cases}.
\end{align}
Integrating these bounds and summing the contributions yields
\begin{align}
|\p_y(e^{iktb}\omega_k)(t,y)|
\lesssim \frac{M_k(1+\log(2+|k|))}{|k|^{2-\theta/2}}
\lesssim \frac{M_k}{|k|^{2-\theta}}.
\end{align}
Away from the critical points this also implies the first bound in \eqref{pro1}.
Combining the bounds and extending by continuity proves \eqref{pro1}.

\subsection{Proof of \eqref{pro2}}

We use the decomposition \eqref{PMR2}. Differentiating the profile and integrating the derivative of the phase by parts gives
\begin{align}
\p_y(e^{iktb}\psi_{k,\rm d})(t,y)
&=\frac{1}{2\pi i}\sum_j\lim_{\epsilon\to0+}
\int_\Sigma e^{-ik(\lambda-b(y))t}
\Bigl[\chi_{bj,c\delta_0^2}(\lambda)
\bigl(\D\phi_{k,\epsilon}^+-\D\phi_{k,\epsilon}^-\bigr)(y,\lambda)
\notag\\
&\hspace{38mm}
+b'(y)\p_\lambda\chi_{bj,c\delta_0^2}(\lambda)
\bigl(\phi_{k,\epsilon}^+-\phi_{k,\epsilon}^-\bigr)(y,\lambda)\Bigr]\,d\lambda.
\end{align}
The boundary terms vanish by Lemma~\ref{boundary}. The non-degenerate contribution and the terms with derivatives on the fixed cutoffs are bounded by $M_k/(|k|^{3-\theta}t^2)$, using two integrations by parts and the estimates of Section~\ref{sec:nd:bdds}.
Moreover, Lemmas~\ref{boundary} and \ref{boundary1} give
\begin{align}
\lim_{\epsilon\to0+}
\bigl(\D\phi_{k,\epsilon}^+-\D\phi_{k,\epsilon}^-\bigr)(y,b(y_{j*}))=0.
\end{align}
We can therefore integrate the remaining term once more by parts. As in the proof of \eqref{damp1}, we differentiate the $\phi$ branches before using \eqref{BDC3.5} and Lemma~\ref{phi:van} to remove the inner contribution. It remains to bound
\begin{align}\label{PMR7}
\frac{1}{|kt|}
\left|\lim_{\epsilon\to0+}\int_\Sigma e^{-ik(\lambda-b(y))t}
\chi_{bj,c\delta_0^2}(\lambda)
\bigl[\p_\lambda\D\digamma_{k,\epsilon}^+
-\p_\lambda\D\digamma_{k,\epsilon}^-\bigr](y,\lambda)\,d\lambda\right|.
\end{align}
The pointwise forms of Lemmas~\ref{lem:gd2} and \ref{lem:gd3} give
\begin{align}\label{PMR8}
\left|\p_\lambda\D\digamma_{k,\epsilon}^\iota
-\mathcal S_{1,1,\epsilon}^\iota\right|
&\lesssim \frac{M_k}{|k|^{1-\theta}\delta(\Lambda)},&
\left|\p_\lambda^2\D\digamma_{k,\epsilon}^\iota
-\mathcal S_{2,1,\epsilon}^\iota\right|
&\lesssim \frac{M_k}{|k|^{1-\theta}\delta^3(\Lambda)}.
\end{align}

First suppose that $|b(y)-b(y_{j*})||kt|\leq1$. Split the integral in \eqref{PMR7} with $\chi_{bj,C(|k|t)^{-1}}$, for $C$ sufficiently large. On its support, \eqref{PMR8} gives an $L^1_\lambda$ bound $M_k/(|k|^{1-\theta}|kt|^{1/2})$ for the regular term. If the annular factor is nonzero, then $y$ lies near $y_{j*}$ and $\delta(\Lambda)\approx |b'(y)|$ in the limit $\epsilon\to0+$. Thus \eqref{DCM7.6}, Lemmas~\ref{DCM1} and \ref{lem:gd1}, and the preliminary logarithmic estimate give
\begin{align}
\limsup_{\epsilon\to0+}
\int_\Sigma \chi_{bj,c\delta_0^2}(\lambda)
|\mathcal S_{1,1,\epsilon}^\iota(y,\lambda)|\,d\lambda
\lesssim \frac{M_k|b'(y)|}{|k|^{1-\theta}}
\lesssim \frac{M_k}{|k|^{1-\theta}|kt|^{1/2}}.
\end{align}
Here the integral over $I_{bj}$ uses the logarithmic estimate, and outside $I_{bj}$ the denominator is separated from zero. On the complementary support of $1-\chi_{bj,C(|k|t)^{-1}}$, the annular factor vanishes. One further integration by parts and \eqref{PMR8}, including the derivative of the time cutoff, give the same bound for that part of the integral:
\begin{align}
\frac{M_k}{|k|^{1-\theta}|kt|}
\int_{(|k|t)^{-1}}^\infty s^{-3/2}\,ds
\lesssim \frac{M_k}{|k|^{1-\theta}|kt|^{1/2}}.
\end{align}

Now suppose that $|b(y)-b(y_{j*})||kt|>1$. Use instead $\chi_{bj,c(|k|t)^{-1}}$, with $c$ sufficiently small. This cutoff is separated from $\lambda=b(y)$ and its annular factor vanishes. The first estimate in \eqref{PMR8} again gives $M_k/(|k|^{1-\theta}|kt|^{1/2})$. On the complementary support we integrate once more by parts. The regular term and the derivative of the time cutoff have the same bound as above. The critical-layer term is zero unless $y$ lies near $y_{j*}$. In that neighborhood, freeze the numerator of the Cauchy kernel in $\mathcal S_{2,1,\epsilon}^\iota$ at $\lambda=b(y)$ and use the preliminary Cauchy and logarithmic estimates on $I_{bj}$. The pole coefficient is bounded by $M_k/(|k|^{1-\theta}|b'(y)|)$ and the logarithmic coefficient by $M_k/(|k|^{1-\theta}|b'(y)|^3)$, by Lemmas~\ref{DCM1}, \ref{lem:gd1}, and \ref{lem:gd4}. Lemmas~\ref{lem:lam1} and \ref{lem:gd2} control the Taylor remainder. Outside $I_{bj}$ the denominator is separated from zero. Thus
\begin{align}
&\left|\lim_{\epsilon\to0+}\int_\Sigma e^{-ik(\lambda-b(y))t}
\chi_{bj,c\delta_0^2}(\lambda)
\bigl(1-\chi_{bj,c(|k|t)^{-1}}(\lambda)\bigr)
\mathcal S_{2,1,\epsilon}^\iota(y,\lambda)\,d\lambda\right|
\notag\\
&\hspace{25mm}\lesssim
\frac{M_k}{|k|^{1-\theta}|b'(y)|}.
\end{align}
Including the two factors of $(kt)^{-1}$, this contribution is bounded by
\begin{align}
\frac{M_k}{|k|^{1-\theta}|kt|^2|b'(y)|}
\lesssim \frac{M_k}{|k|^{5/2-\theta}t^{3/2}},
\end{align}
since $|b'(y)||kt|^{1/2}\gtrsim1$. The same bound follows in the first case after including the prefactor in \eqref{PMR7}. Summing in $j$ proves \eqref{pro2}. Continuity in $y$ extends this bound and the preceding profile and depletion bounds to the critical points.

\subsection{Proof of \eqref{pro3}}

The uniform bounds in \eqref{depl1} and \eqref{pro1} give
\begin{align}
|k|^2\|\omega_k(t)\|_{L^2}
+|k|\|\p_y(e^{iktb}\omega_k)(t)\|_{L^2}
\lesssim \frac{M_k}{|k|^{1-\theta}}.
\end{align}
It remains to bound the second $y$-derivative.
Differentiate the profile in \eqref{PMR9} twice and integrate the derivatives of the phase by parts, as in the proof of \eqref{pro1}.
Since $\D(b(y)-\lambda+i\iota\epsilon)^{-1}=0$, the numerator becomes
$\D^2(b''\psi_{k,\epsilon}^\iota-\omega_{0k})$.
The endpoint terms vanish by Lemmas~\ref{boundary} and \ref{boundary1}.
Using the cutoffs in \eqref{PMR2}, the non-degenerate contribution and the terms with derivatives on the fixed cutoffs are bounded by $M_k/|k|$ in $L^2$, by \eqref{ndD2:bdd}, \eqref{ndlamD2:bdd}, \eqref{ndS12}, and the preliminary Cauchy bound.

By \eqref{BDC1}, \eqref{BDC3.5}, and Lemma~\ref{phi:van}, the correction and inner contributions vanish in the limit, including their first two good derivatives.
The terms containing $2b'''\D\digamma_{k,\epsilon}^\iota+b''''\digamma_{k,\epsilon}^\iota$ are controlled as in the proofs of \eqref{pro1} and \eqref{depl1}.
Thus it remains to estimate \eqref{PMR3} with $\D\digamma_{k,\epsilon}^\iota$ replaced by $\D^2\digamma_{k,\epsilon}^\iota$.
As before, spectral derivatives are taken on the $\phi$ branches before removing the inner terms according to \eqref{sdf:cases}.

Fix $y$ near $y_{j*}$. On $I_{bj}$, subtract the numerator at $\lambda=b(y)$.
Lemmas~\ref{DCM1}, \ref{lem:gd1}, \ref{lem:gd4}, and \ref{lem:gd35}, together with \eqref{prelim:pointwise} and the logarithmic expansion \eqref{DCM7.77}, give the bound $M_k/|k|^{1-\theta}$ by the preliminary Cauchy and logarithmic estimates.
Outside $I_{bj}$ the denominator is separated from zero. Splitting according to whether $|\lambda-b(y_{j*})|$ is smaller than, comparable to, or larger than $|b(y)-b(y_{j*})|$, \eqref{DeC1}, \eqref{gDCM22}, and \eqref{prelim:pointwise} bound the remaining contribution by
\begin{align}
&\frac{M_k}{|k|^{1-\theta}}
\left(1+
\int_0^{|b(y)-b(y_{j*})|}
\frac{ds}{|b(y)-b(y_{j*})|^{1/2}s^{1/2}}
+\int_{|b(y)-b(y_{j*})|}^{C}\frac{ds}{s}\right)
\notag\\
&\hspace{20mm}\lesssim
\frac{M_k}{|k|^{1-\theta}}
\left(1+\big|\log|b(y)-b(y_{j*})|\big|\right).
\end{align}
Away from the critical points the same estimates give $M_k/|k|^{1-\theta}$.
Since $|b(y)-b(y_{j*})|\approx|y-y_{j*}|^2$, the logarithmic factor belongs to $L^2$.
Summing in $j$ yields $\|\p_y^2(e^{iktb}\omega_k)(t)\|_{L^2}\lesssim M_k/|k|^{1-\theta}$ and completes the proof of \eqref{pro3}.

\section{Acknowledgments}
RB was partially supported by NSF DMS-2202974 and ERC-EPSRC Horizon Europe Guarantee EP/X020886/1. SC was partially supported by NSF DMS-2453270 and NSF DMS-2247027. HJ was partially supported by NSF DMS-2453270, NSF DMS-2245021 and a grant from the Simons Foundation.

\textbf{AI Usage Declaration.} During the preparation of this manuscript, the authors used ChatGPT 5.6 Sol and ChatGPT 6 Astra to assist with developing and checking detailed calculations and technical arguments,  formatting bibliographic references, and revising the exposition. In particular, we highlight that the expansions in Sections \ref{sec:nd:bdds} and \ref{sec:d:bdds} were produced by ChatGPT and then checked with SymPy. The SymPy code used for these symbolic checks is available
in the accompanying repository~\cite{BCJgit}. The authors retain full responsibility for the mathematical arguments and the content of the manuscript.

\end{document}